\documentclass[a4paper]{article}

\usepackage[backgroundcolor=orange!20]{todonotes} 

\usepackage[utf8]{inputenc}
\usepackage[T2A]{fontenc} 

\usepackage{amsmath}
\usepackage{amsfonts}
\usepackage{amssymb}
\usepackage{amsthm}
\usepackage{amscd}
\usepackage{cases}
\usepackage{comment}

\usepackage{dsfont}

\usepackage{mathtools} 

\usepackage{commath} 

\usepackage{caption}
\usepackage{floatrow}
\usepackage{hyperref}
\usepackage[
    capitalize,
    nameinlink,
    noabbrev
]{cleveref}

\hypersetup{colorlinks=true, linkcolor=blue, citecolor=blue, urlcolor=blue,
  pdftitle={}, pdfauthor={}}

\theoremstyle{definition}
\newtheorem{theorem}{Theorem}[section]
\newtheorem{lemma}[theorem]{Lemma}
\newtheorem{proposition}[theorem]{Proposition}
\newtheorem{corollary}[theorem]{Corollary}
\newtheorem{definition}[theorem]{Definition}

\theoremstyle{remark}
\newtheorem{remark}[theorem]{Remark}
\newtheorem*{notation}{Notation}

\numberwithin{equation}{section}

\usepackage[left=3cm,right=3cm,top=2.5cm,bottom=2.5cm]{geometry}

\usepackage[
    backend=biber,
    style=alphabetic,
    firstinits=true,
    maxbibnames=99,
    minalphanames=3,
    useprefix=true
]{biblatex}
\usepackage[backgroundcolor=orange!20]{todonotes} 

\usepackage{microtype} 

\usepackage{xspace}

\newcommand*{\cf}{\textit{c.f.}\@\xspace}
\newcommand*{\ie}{\textit{i.e.}\@\xspace}

\newcommand{\Set}{\widehat{\mathsf{Set}}}
\newcommand{\dSet}{\ddot{\mathsf{Set}}}
\newcommand{\DAG}{\mathsf{DAG}}
\newcommand{\SCC}{\mathsf{SCD}}
\newcommand{\CSCC}{\mathsf{CSCC}}
\newcommand{\SAT}{\ddot{\mathsf{SAT}}}
\newcommand{\CNF}{\ddot{\mathsf{CNF}}}
\newcommand{\E}{\mathsf{SSD}}

\newcommand{\G}{\mathsf{G}}
\newcommand{\CG}{\mathsf{CG}}
\newcommand{\T}{\mathsf{T}}
\newcommand{\I}{\mathsf{IT}}
\newcommand{\D}{\mathsf{D}}

\newcommand{\cscc}{\mathfrak{cscc}}
\newcommand{\dgg}{\mathfrak{dag}}
\newcommand{\ssd}{\mathfrak{ssd}}
\newcommand{\scc}{\mathfrak{scd}}
\newcommand{\sat}{\mathfrak{sat}}
\newcommand{\dsat}{\mathfrak{s}\ddot{\mathfrak a}\mathfrak{t}}
\newcommand{\cnf}{\mathfrak{c}\ddot{\mathfrak n}\mathfrak{f}}
\newcommand{\cg}{\mathfrak{cg}}
\newcommand{\ip}{\mathfrak{ip}}
\newcommand{\ir}{\mathfrak{it}}
\newcommand{\g}{\mathfrak{g}}
\newcommand{\dd}{\mathfrak{d}}
\newcommand{\tu}{\mathfrak{t}}

\newcommand{\M}{M}
\newcommand{\BigO}{\mathcal O}

\newcommand{\Ring}[2]{\mathfrak{G}_{#1}^{#2}}
\newcommand{\Coeff}[2]{\mathfrak{C}_{#1}^{#2}}
\newcommand{\Convert}[2]{\mathcal{Q}_{#1}^{#2}}
\newcommand{\Robin}[1]{\Delta_{#1}}
\newcommand{\Change}[3]{\Phi_{#1}^{#2,#3}}

\newcommand{\n}{\overline}

\let\leq\leqslant
\let\geq\geqslant

\usepackage{xcolor}
\definecolor{lgreen}{rgb}{0.0, 0.48, 0.0}
\definecolor{lpurple}{rgb}{0.48, 0.0, 0.48}

\usepackage{tikz}

\definecolor{bblue}{rgb}{0.2, 0.4, 0.8}
\definecolor{bgreen}{rgb}{0.2, 0.6, 0.4}
\definecolor{bred}{rgb}{0.8, 0.4, 0.2}
\definecolor{bviolet}{rgb}{0.7, 0.2, 0.7}
\definecolor{blackred}{rgb}{0.6, 0.3, 0.3}
\definecolor{blackblue}{rgb}{0.3, 0.3, 0.6}
\definecolor{borange}{rgb}{0.8, 0.4, 0.2}

\usetikzlibrary{fit,arrows,trees,shapes,shapes.geometric,calc,matrix,
decorations.markings}
\tikzset{
  treenode/.style = {align=center, inner sep=0pt, text centered,
    font=\sffamily},
  arnBleuPetit/.style = {treenode, circle, bblue, draw=bblue,
    fill=bblue!10,
    minimum width=0.8em, minimum height=0.5em
  },
  arnRougePetit/.style = {treenode, circle, bred, draw=bred,
    fill=bred!40,
    minimum width=0.8em, minimum height=0.5em
  },
  arnBleuGrande/.style = {treenode, circle, bblue, draw=bblue,
    text width=1.5em, very thick,
    fill=bblue!10
  },
  arnOrangePetit/.style = {treenode, circle, black, draw=borange,
    fill=borange!20,
    minimum width=0.8em, minimum height=0.5em
  },
  arnVioletPetit/.style = {treenode, circle, black, draw=bviolet,
    fill=bviolet!08,
    minimum width=0.8em, minimum height=0.5em
  },
  arnVioletGrande/.style = {treenode, circle, bviolet, draw=bviolet,
    text width=1.5em, very thick,
    fill=bblue!10
  },
  arnOrangeGrande/.style = {treenode, circle, borange, draw=borange,
    text width=1.5em, very thick,
    fill=borange!10
  },
  arnVertPetit/.style = {treenode, circle, black, draw=bgreen,
    fill=bgreen!20,
    minimum width=0.8em, minimum height=0.5em
  },
  arnVertGrande/.style = {treenode, circle, bgreen, draw=bgreen,
    text width=1.5em, very thick,
    fill=bgreen!10},
  RougeTriangle/.style = {treenode, bred,
      draw=bred, fill=bred!20, regular polygon, regular polygon
      sides=3, very thick, text width=1.5em}
}

\title{Asymptotics for graphically divergent series:\\
dense digraphs and 2-SAT formulae}

\usepackage[affil-it]{authblk}
\author[$,a,c$]{Sergey \textsc{Dovgal},\thanks{sergey.a.dovgal@gmail.com}}
\author[$,b,d,e$]{Khaydar \textsc{Nurligareev}\thanks{khaydar.nurligareev@lip6.fr}}

\affil[$a$]{LaBRI, Universit\'e de Bordeaux, France}
\affil[$b$]{LIPN, Universit\'e Sorbonne Paris Nord, France}
\affil[$c$]{IMB, Universit\'e de Bourgogne, France}
\affil[$d$]{LIB, Universit\'e de Bourgogne, France}
\affil[$e$]{LIP6, Sorbonne Universit\'e, France}

\date{}

\begin{document}

\maketitle

\begin{abstract}
  We propose a new method for obtaining complete asymptotic expansions in a systematic manner, which is suitable for counting sequences of various graph families in dense regime.
  The core idea is to encode the two-dimensional array of expansion coefficients into a special bivariate generating function, which we call a \emph{coefficient generating function}.
  We show that coefficient generating functions possess certain general properties that make it possible to express asymptotics in a short closed form.
  Also, in most scenarios, we indicate a combinatorial meaning of the involved coefficients.
  Applications of our method include asymptotics of connected graphs, irreducible tournaments, strongly connected digraphs, 2-SAT formulae and contradictory strongly connected implication digraphs.
  Moreover, due to its flexibility, the method allows to treat a wide range of structural variations, including
  fixing the numbers of connected, irreducible, strongly connected and contradictory components,
  as well as source-like, sink-like and isolated ones,
  or adding weights and marking variables.
\end{abstract}


\tableofcontents

\section{Introduction}
\label{section:introduction}

\subsection{Motivation and historical context}

  Asymptotic methods are a powerful tool widely used in enumerative combinatorics.
  Typically, they help to determine how fast different counting sequences grow and how to compare their growth.
  This quantitative information allows scientists to predict the properties of large combinatorial objects and understand their structure.
  There is extensive literature on this account, for instance, the surveys of Bender~\cite{bender1974asymptotic}, Odlyzko~\cite{odlyzko1995asymptotic} and books of De~Bruijn~\cite{bruijn1981asymptotic}, Flajolet and Sedgewick~\cite{flajolet2009analytic}.

  While, in combinatorics, the main concern is often to get the dominant term of the asymptotics,
  there are certain reasons to go further.
  First of them is rather obvious: 
  the more terms in the asymptotic expansion, the more accurate the estimate of the behavior of a counting sequence.
  Ideally, we would like to have a complete asymptotic expansion that allows us to obtain estimates of any predetermined accuracy.

  Another reason to look for complete asymptotic expansions is that they may possess certain structure themselves.
  In other words, coefficients in these expansions often have combinatorial meanings on their own.   
  For instance, it follows from results of Dixon~\cite{dixon2005asymptotics} and Cori~\cite{cori2009indecomposable} that, for any \( r \geq 1 \), the probability \( t_n \) that a uniform random square-tiled surface is connected, satisfies
  \[
    t_n = 1 -
    \sum\limits_{k=1}^{r-1}
      \dfrac{\ip_k}{n^{\underline{k}}} +
    O\left(\dfrac{1}{n^r}\right) 
    \, ,
  \]
  where the sequence \( (\ip_k)_{k=1}^{\infty} \) counts indecomposable permutations and
  \( n^{\underline{k}} = n(n-1)\ldots(n-k+1) \) are the falling factorials (see also \cite{monteilSET}).
  The same way, the probability \( p_n \) that a uniform random graph is connected is equal to
  \[
    p_n = 1 - 
    \sum\limits_{k=1}^{r-1} \ir_k \cdot \binom{n}{k} \cdot
      \dfrac{2^{k(k+1)/2}}{2^{kn}} +
    O\left(\dfrac{n^r}{2^{nr}}\right)
    \, ,
  \]
  where $\ir_k$ is the number of irreducible tournaments of size $k$, see \cite{monteil2021asymptotics}.
  In some cases, coefficients are rather linear combinations of certain counting sequences.
  Thus, the asymptotic expansion of the number of permutations of size \( n \) is given by
  \[
    n! = \dfrac{n^n}{e^n} \sqrt{2\pi n}
    \left(
      1 + \sum\limits_{q=1}^{r-1}\dfrac{c_q}{n^q} +
      O\left(\dfrac{1}{n^r}\right)
    \right)
    \, ,
  \]
  where
  \[
    c_q = \sum\limits_{k=1}^{2q}
    \dfrac{(-1)^k h_{2q+2k,k}}{2^{q+k}(q+k)!}
  \]
  and \( h_{m,k} \) is the number of permutations of size \( m \) having \( k \) cycles, all of length at least three, see \cite[Proposition~B.1]{flajolet2009analytic}.
  The reader can find more examples of that types in \cite{monteilSET} and \cite{monteilSEQ}.

  Comprehension of this ``second level'' structure leads to better understanding of the structure of initial combinatorial objects
  and can be potentially used, for instance, to guess new recurrences and bijections.
  Let us illustrate this idea by the following example.
  Consider a uniform random graph  with \( n \) vertices and \( m = \frac{1}{2}n (1 + \mu n^{-1/3}) \) edges,
  as \( \mu \to -\infty \), while \( |\mu| = o(n^{1/3}) \).
  The first two terms of the asymptotic expansion of the probability that this graph is a union of trees and unicycles are
  \(
    1
     - 
    \frac{5}{24} |\mu|^{-3}
     + 
    \BigO \left( |\mu|^{-6} \right)
  \).
  On the one hand, this can be proved with the help of analytical tools
  (see Equation (10.3) in~\cite[Lemma~3]{janson1993birth} with \( y = 1/2 \)).
  In this case, the rational constant \( \frac{5}{24} \) appears formally as a term in the asymptotic expansion of a complex contour integral.
  On the other hand, the coefficient
  \(
    \frac{5}{24} = 
    \frac{1}{2!}(\frac{1}{4} + \frac{1}{6})
  \)
  can be interpreted as the total weight of bicyclic cubic cores,
  \ie connected multigraphs with two vertices of degree 3.   
  The corresponding term is negative, since these cores are nor trees neither unicycles;
  they form the skeletons of the majority of objects
excluded from a typical random graph in the given range.
  Further coefficients of this asymptotics can also be expressed via combinations of cubic cores weights.
  The presence of these weights in the asymptotic expansion means that the corresponding cores are excluded from the graph:
  in the absence of bicyclic cores, the graph is almost surely a set of trees and unicycles.
  
  Apparently, this phenomenon must be universal within various graph structures.
  Thus, the probability that the strongly connected components of a uniform random digraph
  with \( n \) vertices and \( m = n(1 + \mu n^{-1/3}) \) edges
  are only isolated vertices or cycles
  behaves as
  \(
    1
     - 
    \frac{1}{2} |\mu|^{-3}
     + 
    \BigO \left( |\mu|^{-6} \right)
  \),
  as \( \mu \to -\infty \), see~\cite{de2020birth}.
  Again, the factor \( \frac{1}{2} \) has an interpretation in terms of the weight of excluded structures.
  From similar point of view, the combinatorial interpretation of the coefficients of complete asymptotic expansions is discussed in \cite{monteilSET}, \cite{monteilSEQ} and \cite{monteilANTISEQ} (see also the PhD dissertation \cite{nurligareev2022irreducibility}, which covers most of the above-mentioned papers).
   
  The tool that serves to establish the asymptotic expansions of \cite{dixon2005asymptotics}, \cite{monteil2021asymptotics} and \cite{nurligareev2022irreducibility} in a form suitable for combinatorial interpretation is Bender's theorem \cite{bender1975asymptotic}.
  Given an analytic function \( F(x,y) \) and a formal power series \( A(z) = \sum_{n=0}^\infty a_n z^n \),
  typically with coefficients that grow factorially or superfactorially,
  Bender's theorem  provides the complete asymptotic expansion of the composition \( B(z) = F\big(A(z), z\big) \) in the form
  \[
    b_n
     = 
    \sum\limits_{k=0}^{r-1} c_k a_{n-k}
     +
    \mathcal O(a_{n-r})
    \, ,
  \]
  (for a simplified version of Bender's theorem that we use in this paper, see~\cref{theorem: Bender's}).
  It is necessary to emphasize here that Bender's theorem itself does not provide a combinatorial meaning for the coefficients.
  There is always additional work that depends on the initial data.
  The advantage of this theorem is that the coefficients \( c_k \) are calculated simultaneously:
  the corresponding formal power series \( C(z) = \sum_{n=0}^\infty c_n z^n \) is expressed in a closed form via the input, that is, \( F(x,y) \) and~\( A(z) \).
  This helps, in certain cases, to give the coefficients \( c_k \) a combinatorial interpretation.

  Inspired by the work of Bender~\cite{bender1975asymptotic}, Borinsky studied the asymptotic behavior of \emph{factorially divergent} series~\cite{borinsky2018generating},
  \ie series whose coefficients admit an expansion of the form
  \[
    a_n \sim \alpha^{n+\beta} \Gamma(n+\beta)
    \left(
      d_0
       + 
      \dfrac{d_1}{\alpha(n+\beta-1)}
       +
      \dfrac{d_1}{\alpha^2(n+\beta-1)(n+\beta-2)}
       + 
      \ldots
    \right)
    \, ,
  \]
  where \( \alpha \in \mathbb R_{>0} \) and \( \beta \in \mathbb R \) are fixed parameters.
  He proved that, for the set \( \mathbb{R}[[x]]^{\alpha}_{\beta} \) of series of that type, the following theorem holds.
\begin{theorem}
\label{Borinsky's theorem}
  For any \( \alpha \in \mathbb R_{>0} \) and \( \beta \in \mathbb R \),
  the subspace \( \mathbb{R}[[x]]^{\alpha}_{\beta} \) is a subring of \( \mathbb{R}[[x]] \). 
  Moreover, if \( f,g \in \mathbb{R}[[x]]^{\alpha}_{\beta} \), then \( f\cdot g \in \mathbb{R}[[x]]^{\alpha}_{\beta} \) and,
  with the additional conditions \( g_0 = 0 \) and \( g_1 = 1 \), we also have \( f\circ g, g^{-1} \in \mathbb{R}[[x]]^{\alpha}_{\beta} \).
  The \emph{transfer map}
  \[
   \left(
     \sum\limits_{n=0}^\infty a_n z^n
   \right)
   \xmapsto{\;\;\mathcal{A}^\alpha_\beta\;\;}
   \left(
     \sum\limits_{k=0}^\infty d_k z^k
   \right)
  \]
  is a \emph{derivation}, \ie this map obeys a \emph{Leibniz rule}
  \[
    \big(
      \mathcal{A}^\alpha_\beta (f\cdot g)
    \big)(z)
     =
    f(z)\big(
      \mathcal{A}^\alpha_\beta g
    \big)(z)
     +
    g(z)\big(
      \mathcal{A}^\alpha_\beta f
    \big)(z)
  \]
  and a \emph{chain rule}
  \begin{equation}
  \label{Borinsky's chain rule}
    \big(
      \mathcal{A}^\alpha_\beta (f\circ g)
    \big)(z)
     =
    f'\big(g(z)\big)
    \big(
      \mathcal{A}^\alpha_\beta g
    \big)(z)
     +
    \left( \dfrac{z}{g(z)} \right)^\beta
    \exp\left( \dfrac{g(z)-z}{\alpha zg(z)} \right)
    \big(
      \mathcal{A}^\alpha_\beta f
    \big)\big(g(z)\big)
    \, .  
  \end{equation}
\end{theorem}  

  The formalism of Borinsky makes it possible to derive the asymptotic expansions of certain implicitly defined power series.
  However, the peculiar form of the correction term in chain rule~\eqref{Borinsky's chain rule} makes it difficult to interpret the result combinatorially.
  

\subsection{Notations}

Before moving on to a description of our work, let us introduce some notations that will be useful for this purpose and will be used throughout the paper.

\subsubsection{Sets and expressions}
  The sets of integers and real numbers are designated by their usual notations, \( \mathbb Z \) and \( \mathbb R \), respectively.
  To represent their subsets, we employ subscripts.
  For instance, we write \( \mathbb R_{>1} \) and \( \mathbb Z_{>0} \) for the sets \( \{ x\in\mathbb R \mid x>1 \} \) and \( \{ n\in\mathbb Z \mid n>0 \} \), respectively.
  We also use indicator functions in one or several variables.
  Thus,
  \[
    \mathbf 1_{n=0}
    =
    \left\{\begin{array}{ll}
      1 & n = 0 \\
      0 & \mbox{otherwise}
    \end{array}\right.
	\qquad\mbox{and}\qquad
    \mathbf 1_{n=k}
    =
    \left\{\begin{array}{ll}
      1 & n = k \\
      0 & \mbox{otherwise}
    \end{array}\right.
  \]
  are functions in the variable \( n \in \mathbb Z \) and variables \( n,k \in \mathbb Z \), respectively.
  The domain of any such function will be clear from the context.

  We use Knuth's notation \( n^{\underline{k}} \) for falling factorials~\cite{graham1989concrete}:
  \[
    n^{\underline{k}} := n(n-1) \ldots (n-k+1) \, ,
  \]
  where \( n \) and \( k \) are non-negative integers.
 
  Typically, to designate a family of graphs or 2-SAT-formulae, we use the first letters of its name.
  To represent its generating function and the corresponding counting sequence, we employ these letters written in uppercase serifs and lowercase Gothic fonts, respectively.
  For instance, the exponential generating function of irreducible tournaments is denoted by \( \I(z) \),
  while \( \ir_n \) means the number of irreducible tournaments of size \(n\).
  All the families that will be used in our paper are introduced in \cref{section:enumeration}.

\subsubsection{Sequences}
  For a sequence \( (a_n)_{n=0}^{\infty} \) and an integer $\M$, not necessarily positive, we write
  \begin{equation}\label{eq:approx_def}
	a_n \approx
	\sum\limits_{k\geq\M} f_k(n),
  \end{equation}
  if for all integers \( r\geq\M+1 \), as \( n\to\infty \), one has an asymptotic expansion
  \[
	a_n =
	\sum\limits_{k=M}^{r-1}f_k(n) +
	\BigO\big(f_r(n)\big),
  \]
  where the sequence \( \big(f_k\big)_{n=M}^{\infty} \) satisfies \( f_{k+1}(n) = o\big(f_k(n)\big) \) for each \( k<r \).
  In this case, we also write
  \[
    a_n \sim f_{\M}(n)
  \]
  to identify the leading term of the asymptotics.
  Note that the sum in the right-hand side of \eqref{eq:approx_def} is formal and does not necessarily have to converge.


\subsubsection{Formal power series}
  We use an operator \( [z^n] \) to extract \( n \)th coefficient of formal power series in \( z \):
  \[
    \text{if} \quad
    A(z) = \sum_{n=0}^\infty a_n z^n \, ,
    \qquad \text{then} \quad
    [z^n] A(z) := a_n \, .
  \]
  The \emph{exponential Hadamard product} of formal power series \( A(z) \) and \( B(z) \) is designated by \( A(z) \odot B(z) \).
  In other words, if
  \[
    A(z) = \sum\limits_{n=0}^{\infty}
      a_n \frac{z^n}{n!}
    \qquad \text{and} \qquad
    B(z) = \sum\limits_{n=0}^{\infty}
      b_n \frac{z^n}{n!}
    \enspace ,
  \]
  then
  \[
    A(z) \odot B(z) :=
    \sum\limits_{n=0}^{\infty}
      a_n b_n \frac{z^n}{n!}
    \enspace .
  \]
  In the case of several arguments, we use the notation \( \odot_z \) to emphasize that
  the exponential Hadamard product is taken with respect to the argument \( z \).

  For different needs, we use several \emph{types} of generating functions (GFs).
  To highlight the type of GF used, we capitalize them.
  Overall, we employ the following four types:
  \emph{Exponential} GF, \emph{Graphic} GF,  \emph{Implication} GF and \emph{Coefficient} GF.
  The details are forthcoming in the next sections
 (see formulae \eqref{eq:ExponentialGF}, \eqref{eq:GraphicGF}, \eqref{eq:ImplicationGF} and \eqref{eq:CoefficientGF}, respectively).

\subsection{Our contribution}

  Our paper is devoted to the asymptotic study of exponential generating series of a certain kind: the corresponding counting sequences diverge as powers of quadratic functions.
  For this purpose, we introduce a new type of generating series and an asymptotic transfer.

\begin{definition}
\label{definition: asymptotics}
Let \( \alpha \in \mathbb R_{>1} \) and \( \beta \in \mathbb Z_{>0} \).
\begin{enumerate}
  \item
	By \( \Ring{\alpha}{\beta} \) (the Gothic ``G'' for ``graphic'')
    we denote the set of formal power series
	\begin{equation}
	\label{eq:EGF F}
  	  A(z) = \sum_{n = 0}^\infty a_n \frac{z^n}{n!}
	\end{equation}
	whose coefficients \( a_n \) satisfy an asymptotic expansion
	\begin{equation}
	\label{eq:expansion}
    	  a_n \approx \alpha^{\beta \binom{n}{2}}
    	  \left[
      	\sum_{m\geq\M} \alpha^{-mn}
      	\sum_{\ell=0}^{\infty}
	  	  n^{\underline{\ell}} \,
    		  a_{m,\ell}^\circ
      \right]
	\end{equation}
	for some integer $\M$
	with the additional assumption that, for each \( m \in \mathbb Z_{\geq\M} \), the support of the sequence \( (a^\circ_{m,\ell})_{\ell=0}^\infty \) is finite. 
	Such series will be called \emph{graphically divergent}.
  \item
    The \emph{Coefficient GF of type \( (\alpha, \beta) \)}
    associated with Exponential GF~\eqref{eq:EGF F}
	whose coefficients \( (a_n)_{n=0}^\infty \) admit an asymptotic expansion of the form~\eqref{eq:expansion}
    is the bivariate formal power series
    \begin{equation}
    	\label{eq:CoefficientGF}
      A^\circ(z, w) :=
      \sum_{m = \M}^\infty
      \sum_{\ell=0}^\infty
        a_{m,\ell}^\circ
        \dfrac{z^m}{\alpha^{\frac{1}{\beta} \binom{m}{2}}}
      w^\ell \, .
    \end{equation}
    We denote by \( \Coeff{\alpha}{\beta} \) (the Gothic ``C'' for ``coefficient'') the set of Coefficient GF corresponding to~\( \Ring{\alpha}{\beta} \).
    \footnote{
      In addition to the sets \( \Ring{\alpha}{\beta} \) and \( \Coeff{\alpha}{\beta} \), our paper contains two more objects that depend on parameters \( \alpha \) and \( \beta \):
      namely the operators \( \Robin{\alpha} \) and \( \Change{\alpha}{\beta_1}{\beta_2} \); see~\cref{definition: Robinson operator} and~\cref{definition: Coefficient operator}, respectively.
      Unlike Borinsky, we decided to set \( \alpha \) as the subscript (and \( \beta \) as the superscript), in order to write powers of the operator \( \Robin{\alpha} \) as \( \Robin{\alpha}^m \) and avoid confusing powers and indices.
    }
  \item
    The \emph{asymptotic transfer}
    \(
      \Convert{\alpha}{\beta} \colon
      \Ring{\alpha}{\beta} \to
      \Coeff{\alpha}{\beta}
    \)
    is the linear operator that transfers a formal power series (an Exponential GF) to the respective bivariate Coefficient GFs,
    \ie for \( A \in \Ring{\alpha}{\beta}\), we have
    \begin{equation}
	\label{eq:Q-alpha-beta-Conversion}
	  \mathcal Q^\alpha_\beta A = A^\circ.
    \end{equation}
    In other words, if \( (a_n)_{n=0}^\infty \) and \( (a_{m,\ell}^\circ)_{m=M,\ell = 0}^\infty \) satisfy~\eqref{eq:EGF F} and~\eqref{eq:expansion}, then
    \[
      \sum_{n=0}^\infty a_n \dfrac{z^n}{n!}
        \quad
      \xmapsto{\;\;\Convert{\alpha}{\beta}\;\;}
        \quad
      \sum_{m=\M}^{\infty}
      \sum_{\ell=0}^\infty
        a_{m,\ell}^\circ
        \dfrac{z^m}{\alpha^{\frac{1}{\beta}\binom{m}{2}}}
        w^\ell \, .
    \]
    Here we assume that \( \binom{m}{2} = \frac{m(m-1)}{2} \) for any \( m \in \mathbb Z \).
\end{enumerate}
\end{definition}

  We will see in \cref{sec: transfers-relations} that, for a fixed parameter \( \alpha \), the sets \( \Ring{\alpha}{\beta} \) are subsequently embedded one into another:
  \[
    \Ring{\alpha}{1} \subset
    \Ring{\alpha}{2} \subset
    \Ring{\alpha}{3} \subset
    \Ring{\alpha}{4} \subset
      \ldots
  \]
  On the contrary, from the formal point of view, linear spaces \( \Coeff{\alpha}{\beta} \) are the same for different values of \( \alpha \) and \( \beta \):
  \[
	\Coeff{\alpha}{\beta} = \mathbb R[w][[z]]
	\, .
  \]
  The choice of parameters \( \alpha \) and \( \beta \) corresponds to the choice of basis in this space.

  For fixed parameters \( \alpha \) and \( \beta \), we show that the set \( \Ring{\alpha}{\beta} \) admit a ring structure,
  and that the asymptotic transfers are consistent with the ring operations.
  More precisely, the following holds.

\begin{proposition}
\label{lemma: operation transfers}
  For any fixed \( \alpha\in \mathbb R_{>1} \) and \( \beta \in \mathbb Z_{>0} \), the set \( \Ring{\alpha}{\beta} \) form a ring.
  For each pair \( A,B\in\Ring{\alpha}{\beta} \), the operations of addition and multiplication satisfy
  \begin{equation}\label{eq: sum-transfer}
    \big(
	  \Convert{\alpha}{\beta} (A + B)
    	\big)(z, w)
      =
    (\Convert{\alpha}{\beta} A)(z,w)
      +
    (\Convert{\alpha}{\beta} B)(z,w)
  \end{equation}
  and
  \begin{equation}\label{eq: prod-transfer}
    \big(
	  \Convert{\alpha}{\beta} (A \cdot B)
    \big)(z, w)
      =
    A \big(
      \alpha^{\frac{\beta+1}{2}} z^\beta w
    \big) \cdot
    (\Convert{\alpha}{\beta} B)(z,w)
      +
    B \big(
      \alpha^{\frac{\beta+1}{2}} z^\beta w
    \big) \cdot
    (\Convert{\alpha}{\beta} A)(z,w)
    \, .
  \end{equation}
\end{proposition}

\begin{proposition}
\label{lemma: Bender's transfer}
  Let \( \alpha \in \mathbb R_{>1} \), \( \beta \in \mathbb Z_{>0} \) and \( A \in \Ring{\alpha}{\beta} \) with \( a_0=0 \).
  If \( F \) is a~function analytic in a~neighbourhood of the origin,
  and \( H(z) = \partial_x F(x)|_{x=A(z)} \),
  then \( F\circ A \in \Ring{\alpha}{\beta} \) and
  \begin{equation}\label{eq: Bender-transfer}
    \big(
	  \Convert{\alpha}{\beta} (F\circ A)
    \big)(z, w)
      =
    H \big(
      \alpha^{\frac{\beta+1}{2}} z^\beta w
    \big) \cdot
    (\Convert{\alpha}{\beta} A)(z,w)
    \, .
  \end{equation}
\end{proposition}

  \cref{lemma: operation transfers} and \cref{lemma: Bender's transfer},
  along with their multivariate versions \cref{lemma: operation transfers and marking variables} and \cref{lemma: Bender's transfer and marking variables},
  are our key tools for establishing asymptotic expansions of graphically divergent series.
  As in the case of \cref{Borinsky's theorem}, the proof is based on a careful estimation of the growth rate of the middle terms.
  However, compared to Borinsky's paper, we can see several differences.
  First, we deal with series whose growth rate is higher: roughly speaking, it is \( \alpha^{n^2} \) versus \( \alpha^n\cdot n! \).
  Second, we also keep track of polynomial ``fluctuations'' by means of two-dimensional arrays of coefficients.
  This allows us to get more information versus one-dimensional case of Borinsky, but the exact form of our coefficient generating functions needs to be carefully designed.
  We pay for the additional information by complexifying the function \( A^\circ \), which is reflected in the term \( \alpha^{\frac{\beta+1}{2}} z^\beta w \) inserted into the Leibniz and chain rules.
  It is worth mentioning that the main advantage of Borinsky's approach is the possibility of obtaining compositions and inverses within the ring under consideration.
  In our case, only compositions with analytic functions are allowed.
  In principle, this means that all series we deal with can be treated with the help of Bender's theorem.
  Thus, our principal contribution is not to compute asymptotic expansions, but to present them in a nice looking, concise form, easy to understand, convenient to interpret and adapted to use.

  We present several applications of this technique.
  First, we revisit in a simple manner the complete asymptotic expansions for connected digraphs and irreducible tournaments
  (the asymptotic coefficients were obtained by Wright~\cite{wright1970asymptotic} and~\cite{wright1970the}; a combinatorial interpretation of these coefficients was given by Monteil and Nurligareev~\cite{monteil2021asymptotics}).
  Second, we obtain a complete asymptotic expansion of strongly connected digraphs (here, the asymptotic coefficients were known~\cite{wright1971number}, but no interpretation was given).
  We also establish a complete asymptotic expansion for digraphs with a fixed number of strongly connected components (this result is new).
  In particular, we provide a~refined version of this result, for the case when the number of source-like, sink-like and isolated components are given.
  Finally, we establish complete asymptotic expansions
for satisfiable and contradictory strongly connected 2-CNF formulae (these results are also new).
  We also discuss the case where the number of components are fixed.
  The method can be potentially refined even further to restrain the strongly connected components.

  Note that the coefficients involved into the asymptotic expansions under discussion are virtually always expressed in a relatively simple way via enumerating sequences of other combinatorial families.
  As a consequence, somehow as a byproduct of our method, they have combinatorial meanings on their own.

\subsection{Structure of the paper}

  The paper consists of six sections, the first of which is the present introduction.
  \cref{section:enumeration} can be considered as a brief listing of all necessary prerequisites.
  We introduce different generating functions, as well as recall various graph families and 2-SAT model, and provide results related to their enumeration.
  The section is divided into two parts, but this division is rather arbitrary, since the parts share common ideas.
  The described material is largely not new and can be covered by \cite{de2019symbolic, dovgal2021exact} and, for example, \cite{flajolet2009analytic}.
  However, we encourage the reader to have a look at the presentation in order to get familiar with the differences
  (for instance, we prefer operators \( \Robin{\alpha} \) and \( \Change{\alpha}{\beta_1}{\beta_2} \) rather than exponential Hadamard product).

  \cref{section:transfers} is devoted to our method of the asymptotic transfer.
  This is the core part of the paper.
  We prove that graphically divergent series form a ring, and study properties of the asymptotic transfer.
  In \cref{section:digraphs}, we consider applications of our method to asymptotics of undirected and directed graphs,
  while \cref{section:2-SAT} is devoted to applications to 2-SAT formulae.
  Numerical values corresponding to the results of \cref{section:digraphs} and \cref{section:2-SAT} are presented in \cref{section:appendix}.
  Finally, in~\cref{section:discussion}, we discuss possible extensions of the method not covered by the current paper:
  the behavior of expansions when the edge probability tends to zero, the enumeration of 2-connected graphs and blocks, and possible extensions to enumeration of the $k$-SAT formulae.

\section{Generating functions and enumeration}
\label{section:enumeration}

\subsection{Digraphs}
\label{subsection:digraphs}

  In this section, we recall the symbolic method applied for enumeration of various (undirected or directed) graph families.
  We start with observing different graph structures that are employed in our investigation.
  Note that all objects we work with are \emph{labeled};
  this is assumed throughout, and the word ``labeled'' will be omitted below.
  Next, we describe two types of generating functions, exponential and graphic, that serve for graph enumeration purposes, and explain relations between them.
  In particular, we discuss multivariate generating functions that are useful for marking patterns and parameters.
  The presentation of the topic is completed by enumeration results for diverse digraph classes,
  including close formulae for generating functions of acyclic, strongly connected and semi-strong digraphs.
  Our exposition is mainly based on the book \cite{flajolet2009analytic} and the papers \cite{de2019symbolic, de2020birth}, to which we refer the reader for further details.

\subsubsection{Graph families}
  Let us recall various graph families that will be used throughout the paper.

  A \emph{graph} is a pair \( (V, E) \),
  where \( V \) is a finite set of \emph{vertices},
  typically represented by an interval \( [n] := \{ 1,\ldots, n\} \),
  and \( E \subset \big\{ \{x,y\} \mid x,y \in V, \, x \neq y \big\} \) is the set of \emph{edges}.
  In particular, loops and multiple edges are forbidden in this model.
  In contrast to directed graphs, which are discussed below,
  these graphs are referred to as \emph{undirected}.

  A graph is \emph{connected} if any pair of its vertices is joined by a path.
  In other words, for any pair \( x,y \in V \), there exists a sequence of vertices
  \begin{equation}\label{eq:vertex sequence}
    x = v_0,\, v_1,\, \ldots,\, v_{m-1},\, v_m = y,
  \end{equation}
  such that \( \{v_{i-1},v_i\} \in E \) for all \( i=1,\ldots,m \).
  Every graph can be uniquely represented as a disjoint union of its connected components.

  A \emph{directed graph} (or, simply, a \emph{digraph})
  is a pair \( (V, E) \), where \( V \) is a finite set of vertices,
  and \( E \subset \big\{ (x,y) \mid x,y \in V, \, x \neq y \big\} \) is the set of edges.
  Contrary to undirected graphs, the order of vertices in an edge is important,
  so these edges are referred to as \emph{directed}.

  A digraph is \emph{strongly connected} if any pair \( x,y \in V \) is joined by a directed path,
  meaning that there exists a sequence of vertices~\eqref{eq:vertex sequence}
  such that \( (v_{i-1},v_i) \in E \) for all \( i=1,\ldots,m \).
  Every graph consists of several strongly connected components.
  In contrast to the undirected case, two components can be joined by directed edges.
  However, all the edges that join two distinct components must have the same direction.
  
  Depending on its nature, a strongly connected component of a digraph can be
\begin{enumerate}
  \item
    \emph{source-like}, if it does not have incoming edges from other components;
  \item
    \emph{sink-like}, if it does not have outgoing edges towards other components;
  \item
    \emph{isolated}, if it is source-like and sink-like at the same time;
  \item
    \emph{purely source-like}, if it is source-like and not isolated;
  \item
    \emph{purely sink-like}, if it is sink-like and not isolated.
\end{enumerate}
  A digraph is \emph{semi-strong}, if all its strongly connected components are isolated.

  A \emph{tournament} is a digraph such that each pair of its vertices \( x,y \in V \) is joined by exactly one of two directed edges: either \( (x,y) \) or \( (y,x) \).
  A tournament is \emph{reducible}, if there exists a partition of its set of vertices into two nonempty subsets \( A \) and \( B \)
  such that any pair of vertices \( (a,b) \in A \times B \) are joined by the edge \( (a,b) \).
  Otherwise, the tournament is \emph{irreducible}.
  Equivalently, a tournament is irreducible if and only if it is strongly connected~\cite{rado1943theorems}.
  
  Finally, a digraph is \emph{acyclic}, if it does not contain directed cycles.
  Usually, the corresponding subclass of digraphs is referred to as \emph{directed acyclic graphs}.

\subsubsection{Exponential generating functions}

  Recall that \( (a_n)_{n = 0}^\infty \) is a \emph{counting sequence} of a family \( \mathcal A \) of (undirected or directed) graphs,
  if \( a_n \) denotes the number of graphs from \( \mathcal A \) with \( n \) vertices.
  The \emph{Exponential GF} of \( \mathcal A \) is
  \begin{equation}\label{eq:ExponentialGF}
    A(z) := \sum\limits_{n=0}^{\infty}
      a_n \dfrac{z^n}{n!}
    \, .
  \end{equation}
  Exponential GFs are commonly used for enumerating labeled combinatorial classes,
  since their behavior is consistent with the labeled product.
  More precisely, if \( (a_n)_{n=0}^{\infty} \) and \( (b_n)_{n=0}^{\infty} \) are the counting sequences of classes \( \mathcal{A} \) and \( \mathcal{B} \), respectively,
  then the counting sequence \( (c_n)_{n=0}^{\infty} \) of the labeled product
  \(
    \mathcal{C} = \mathcal{A} \star \mathcal{B}
  \)
  obeys the binomial convolution rule,
  \[
    c_n =
    \sum\limits_{k=0}^{n} \binom{n}{k} a_k b_{n-k},
  \]
  which corresponds to the relation
  \(
    C(z) = A(z) B(z)
  \)
  of the Exponential GFs.
  For further details, see \cite{flajolet2009analytic}.
  
  Generating functions serve to express structural relationships between different classes of combinatorial objects in the language of algebra and vice versa.
  Thus, the Exponential GFs \( \G(z) \) and \( \CG(z) \) of graphs and connected graphs, respectively,
  satisfy the so-called \emph{exponential formula} (see, for example, \cite[Example~5.2.1]{stanley1999enumerative2}):
  \begin{equation}\label{eq:G=exp(CG)}
    \G(z) = e^{\CG(z)}.
  \end{equation}
  A similar formula gives a link between the Exponential GFs \( \E(z) \) and \( \SCC(z) \) of semi-strong digraphs and strongly connected digraphs, respectively:
  \begin{equation}\label{eq:SSD=exp(SCD)}
    \E(z) = e^{\SCC(z)}.
  \end{equation}
  Another relation following from \cite[formula~(1)]{moon1968topics} links together the Exponential GFs \( \T(z) \) and \( \I(z) \) of tournaments and irreducible tournaments, respectively:
  \begin{equation}\label{eq:T=1/(1-IT)}
    \T(z) = \dfrac{1}{1-\I(z)}
    \, .
  \end{equation}
  Note that the Exponential GFs of graphs and tournaments are the same:
  \begin{equation}\label{eq:G=T}
    \G(z) = \T(z) =
    \sum\limits_{n=0}^{\infty}
      2^{\binom{n}{2}}\dfrac{z^n}{n!}
    \, .
  \end{equation}

\subsubsection{Graphic generating functions}

  For some families of directed graphs, it is more convenient to use \emph{Graphic GFs}:
  \begin{equation}\label{eq:GraphicGF}
    \widehat A(z) := \sum_{n=0}^{\infty}
      a_n \dfrac{z^n}{2^{\binom{n}{2}}n!}
    \, .
  \end{equation}
  This is the case where, instead of the labeled product, we employ a so-called \emph{arrow product}:
  given two families of directed graphs \( \mathcal{A} \) and \( \mathcal{B} \),
  we consider a new class \( \mathcal{C} \)
  of pairs \( (a,b) \), \( a \in \mathcal{A} \), \( b \in \mathcal{B} \),
  equipped with additional edges directed from vertices of \( a \) to vertices of~\( b \).
  Indeed, the convolution rule corresponding to the arrow product is
  \[
    c_n =
    \sum\limits_{k=0}^{n} \binom{n}{k}
     2^{k(n-k)} a_k b_{n-k}.
  \]
  Hence, their Graphic GFs satisfy
  \(
    \widehat C(z) = \widehat A(z) \widehat B(z).
  \)
  For details, see~\cite{de2019symbolic}.
  
  There are certain bridges between Exponential GFs and Graphic GFs.
  For instance, the Exponential GF \( \D(z) \) and the Graphic GF \( \widehat\D(z) \) of directed graphs are:
  \[
    \D(z) = \sum_{n = 0}^\infty 2^{n(n-1)} \dfrac{z^n}{n!}
    \qquad \text{and} \qquad
    \widehat\D(z) = \sum_{n = 0}^\infty 2^{\binom{n}{2}} \dfrac{z^n}{n!} = \G(z),
  \]
  respectively.
  To proceed from the Exponential GF of a family \( \mathcal A \) to its Graphic GF,
  we use the linear operator~\( \Robin{2} \) first defined by Robinson~\cite{robinson1973counting}:
  \begin{equation}\label{eq:Robinson's operator}
    \Robin{2}A(z) := \widehat A(z). 
  \end{equation}
  This operator divides \( z^n \) by \( 2^{\binom{n}{2}} \) and can be expressed in terms of the exponential Hadamard product.
  Namely, the conversion between the Exponential GF \( A(z) \) and the corresponding Graphic GF \( \widehat A(z) \) is done according to the formula
  \begin{equation}\label{eq:GF conversion}
    \widehat A(z) = \Robin{2}A(z) = A(z) \odot \Set(z)
    \qquad \text{and} \qquad
    A(z) = \Robin{2}^{-1}A(z) = \widehat A(z) \odot \G(z)
    \, ,
  \end{equation}
  where, in notations of de Panafieu and Dovgal~\cite{de2019symbolic},
  \[
    \Set(z) =\sum\limits_{n=0}^{\infty}
      \dfrac{z^n}{2^{\binom{n}{2}} n!}
  \]
  is the Graphic GF of \emph{digraphs without edges} (or, equivalently, \emph{sets of isolated vertices}).

\subsubsection{Marking variables}
\label{subsection: marking variables}

  It is said that \( u \) is a \emph{marking variable} for the number of occurrences of a pattern \( \pi \)
  (or, simply, that \( u \) \emph{marks} the number of \( \pi \))
  in a generating function
  \[
    F(z,u) = \sum_{m=0}^{\infty} F_m(z) u^m \, ,
  \]
  if \( F_m(z) \) is a generating function for objects having exactly \( m \) occurrences of \( \pi \) for every \( m \geq 0 \).
  This concept can be recursively extended to an arbitrary number of marking variables.

  As an example of particular importance, consider undirected graphs.
  Introducing a marking variable \( w \) for the number of edges, we get the corresponding Exponential GF to be
  \begin{equation}\label{eq:marking edges in graphs}
    \G(z, w) = \sum\limits_{n=0}^{\infty}
      (1+w)^{\binom{n}{2}} \dfrac{z^n}{n!}
    \enspace .
  \end{equation}
  In this case, the definition of a Graphic GF should be modified to
  \begin{equation}\label{eq:bivariate Graphic GF}
    \widehat A(z, w) :=
    \sum\limits_{n=0}^{\infty}
      a_n(w) \dfrac{z^n}{(1+w)^{\binom{n}{2}}n!}
    \, ,
  \end{equation}
  so that conversions~\eqref{eq:GF conversion} hold. In particular,
  \begin{equation}\label{eq:marking edges in graphs - 2}
    \Set(z,w) = \sum\limits_{n=0}^{\infty}
      \dfrac{z^n}{(1+w)^{\binom{n}{2}} n!}
    \, .
  \end{equation}
  Another option is to consider a marking variable \( t \) for the number of connected components.
  Doing that, we obtain the following generalization of formula~\eqref{eq:G=exp(CG)}:
  \begin{equation}\label{eq:G(t)=exp(tCG)}
    \G(z;t) = e^{t \cdot \CG(z)}.
  \end{equation}

  Studying directed graphs, the reader may need several marking variables.
  First of all, it is reasonable to introduce a marking variable \( t \) for the number of strongly connected components.
  For instance, the bivariate Exponential GFs of semi-strong digraphs and tournaments satisfy, respectively,
  \begin{equation}\label{eq:bivariate:SSD=exp(SCD)}
    \E(z;t) = e^{t \cdot \SCC(z)}
    \qquad\mbox{and}\qquad
    \T(z;t) = \dfrac{1}{1 - t \cdot \I(z)}
    \, ,
  \end{equation}
  which generalize relations \eqref{eq:SSD=exp(SCD)} and \eqref{eq:T=1/(1-IT)}.
  We could be interested in source-like components too.
  In the next section, we also introduce marking variables for the numbers of purely source-like, purely sink-like and isolated components of digraphs.

\begin{remark}
\label{remark:introducing Erdos-Renyi}
  Introducing a marking variable for the number of edges in a graph is closely related the Erd\H{o}s-R\'enyi model \( \mathbb G(n,p) \)~\cite{erdos1959random,gilbert1959random} and the similar model \( \mathbb D(n,p) \).
  Recall that, according to these models, each edge of a (undirected or directed) graph with \( n \) vertices appears independently with a fixed positive probability \( p \).
  In particular, in \( \mathbb G(n,p) \) a graph with \( m \) edges appears with the probability of
  \[
    p^{m}(1-p)^{\binom{n}{2}-m} =
    \left(\dfrac{p}{1-p}\right)^m
      (1-p)^{\binom{n}{2}}
    \, ,
  \]
  while in \( \mathbb D(n,p) \) the probability to obtain a fixed digraph with \( m \) directed edges is
  \[
    p^{m}(1-p)^{2\binom{n}{2}-m} =
    \left(\dfrac{p}{1-p}\right)^m
      (1-p)^{2\binom{n}{2}}
    \, .
  \]
  It can be shown that if \( \mathcal F \) is a undirected (resp. directed) graph family
  whose Exponential GF (resp. Graphic GF) is \( F(z, w) \),
  then the probability that a randomly generated graph from \( \mathbb G(n,p) \) (resp. \( \mathbb D(n,p) \)) belongs to \( \mathcal F \) is exactly
  \[
    \mathbb P_{\mathcal F}(n, p)
    =
    (1 - p)^{\binom{n}{2}} n! [z^n] F
    \left(
      z, \dfrac{p}{1-p}
    \right)
  \]
  (see \cite[Lemma 2.8]{de2020birth}).
  Putting the weight of a graph to be \( (\frac{p}{1-p})^m \),
  we get the total weight of all graphs equal to \( (1 - p)^{-\binom{n}{2}} \),
  and the probability of a specific family can be obtained by dividing its weight by the total weight.
  Note, that the case \( p = \frac12 \) corresponds to enumeration of graphs,
  since \( w = \frac{p}{1-p} = 1 \) in this case.
\end{remark}

\subsubsection{Enumeration of digraphs with marking variables}
\label{subsection: enumeration of digraphs with marking variables}

  Let us now proceed to general enumerative results that provide generating functions for digraphs with various marking variables.
  They are based on a technique developed by Gessel~\cite{gessel1996counting} for counting acyclic digraphs by sources and sinks.
  Some of results presented here seem to be new, others can be found in the works~\cite{robinson1973counting, de2019symbolic, de2020birth}.

\begin{proposition}[{\cite[Theorem 3.4]{de2019symbolic}}]
\label{proposition: digraphs with given scc}
  Let \( \mathcal A \) be a family of strongly connected digraphs and \( A(z) \) be its Exponential GF.
  If \( s \) marks the number of source-like components,
  then the bivariate Graphic GF \( \widehat{D}_{\mathcal A}(z; s) \) of digraphs
  whose strongly connected components belong to \( \mathcal A \)
  is given by 
  \[
    \widehat{\D}_{\mathcal A}(z; s) =
    \dfrac
      {\Robin{2} \big( e^{(s-1)A(z)} \big)}
      {\Robin{2} \big( e^{-A(z)} \big)}
    \, .
  \]
\end{proposition}

\begin{corollary}\label{corollary: GGF of digraphs marking components}
  If \( s \) marks the number of source-like components
  and \( t \) marks the total number of strongly connected components,
  then the Graphic GF \( \widehat{\D}(z; s, t) \) of digraphs satisfies
  \begin{equation}\label{eq: GGF of digraphs marking components and source-like components}
    \widehat{\D}(z; s, t)
     =
    \dfrac
      {\Robin{2} \big(e^{(s-1)t \cdot \SCC(z)}\big)}
      {\Robin{2} \big(e^{-t \cdot \SCC(z)}\big)}
     =
    \dfrac
      {\Robin{2} \big( \E(z; (s-1)t) \big)}
      {\Robin{2} \big( \E(z; -t) \big)}
    \, .
  \end{equation}
  In particular, we have the following expression for the bivariate Graphic GF \( \widehat{\D}(z; t) \) of digraphs:
  \begin{equation}\label{eq: GGF of digraphs marking components}
    \widehat{\D}(z; t)
     =
    \dfrac
      {1}
      {\Robin{2} \big(e^{-t \cdot \SCC(z)}\big)}
     =
    \dfrac{1}{\Robin{2} \big( \E(z; -t) \big)}
    \, .
  \end{equation}
\end{corollary}
\begin{proof}
  Let us apply \cref{proposition: digraphs with given scc} to the family \( \mathcal A = \mathcal S_t \)
  consisting of all strongly connected digraphs taken with a weight \( t \) each.
  The Exponential GF of this family is \( t \cdot \SCC(z) \),
  while \( \widehat{D}_{\mathcal S_t}(z; s) \) is the Graphic GF of weighted digraphs, 
  where a digraph with \( k \) strongly connected components has the weight~\( t^k \) 
  and \( s \) marks the number of source-like components.
  Together with \eqref{eq:bivariate:SSD=exp(SCD)}, this gives us~\eqref{eq: GGF of digraphs marking components and source-like components},
  and putting \( s = 1 \), we obtain~\eqref{eq: GGF of digraphs marking components}.
\end{proof}

\begin{proposition}\label{proposition: complete multivariate Graphic GF for digraphs}
  If \( u \), \( v \) and \( y \) mark, respectively, the numbers of purely source-like, purely sink-like and isolated components,
  while \( t \) marks the total number of strongly connected components,
  then the multivariate Exponential GF \( \D(z; u, v, y, t) \) of digraphs is given by
  \begin{equation}\label{eq: complete multivariate Graphic GF for digraphs}
    \D(z; u, v, y, t)
    =
    e^{(y-u-v+1) t \cdot \SCC(z)} \cdot
    \Robin{2}^{-1}
    \left(
      \dfrac{
        \Robin{2} \big(e^{(u-1)t \cdot \SCC(z)}\big)
        \cdot
        \Robin{2} \big(e^{(v-1)t \cdot \SCC(z)}\big)
      }
      {\Robin{2} \big(e^{-t \cdot \SCC(z)}\big)}
    \right)
    \, .
  \end{equation}
\end{proposition}
\begin{proof}
  Consider the family of all digraphs with a distinguished subset of purely source-like, purely sink-like and isolated components.
  Let each purely source-like component be marked with \( \hat{u} \),
  each purely sink-like component be marked with \( \hat{v} \)
  and each isolated component be marked with either \( \hat{u} \), \( \hat{v} \) or \( \hat{y} \).
  If, additionally, \( t \) marks the total number of connected components,
  then the Graphic GF of such digraphs is
  \( \widehat{\D}(z; 1+\hat{u}, 1+\hat{v}, 1+\hat{u}+\hat{v}+\hat{y}, t) \).

  The constructed family can be decomposed into the labeled product of the class of isolated components marked by \( \hat{y} \) and another digraph family.
  The latter, in its turn, is the arrow product of the following three digraph families:
  the class of source-like components marked by \( \hat{u} \) (some of which may be isolated),
  the class of arbitrary digraphs,
  and the class of sink-like components marked by \( \hat{v} \) (again, some of them may be isolated).
  For example, \cref{fig:sink:source:decomposition} shows the following components of a digraph from the family, from left to right:
  the one marked with \( \hat{u} \) (violet),
  the non-distinguished one (blue),
  the one marked with \( \hat{v} \) (green),
  and the one marked with \( \hat{y} \) (orange).
  Note that each component is marked by \( t \) as well.

\begin{figure}[hbt!]
    \begin{center}
        \begin{tikzpicture}[>=stealth',thick, scale = 0.9]
\draw
%
%
node[arnVioletGrande](a) at (0,0) {$\circlearrowright$}
node[arnVioletGrande](b) at (0,-1) {$\circlearrowright$}
%
%
node[arnBleuGrande](c) at (2,0) {$\circlearrowright$}
node[arnBleuGrande](d) at (2,-1) {$\circlearrowright$}
node[arnBleuGrande](e) at (3.5,0) {$\circlearrowright$}
node[arnBleuGrande](f) at (3.5,-1) {$\circlearrowright$}
%
%
node[arnVertGrande](g) at (5.5,-0.5) {$\circlearrowright$}
%
%
node[arnOrangeGrande](i) at (7.5,0) {$\circlearrowright$}
node[arnOrangeGrande](h) at (7.5,-1) {$\circlearrowright$}
%
;

\node[rectangle,dashed,draw, fit=(a)(b),
      rounded corners = 3mm,inner sep = 7pt, bviolet, very thick] {};
\node[rectangle,dashed,draw, fit=(c)(d)(e)(f),
      rounded corners = 3mm,inner sep = 7pt, bblue, very thick] {};
\node[rectangle,dashed,draw, fit=(g),
      rounded corners = 3mm,inner sep = 9pt, bgreen, very thick] {};
\node[rectangle,dashed,draw, fit=(h)(i),
      rounded corners = 3mm,inner sep = 7pt, borange, very thick] {};

\path (a) edge [blackred,thick, bend right=-5,
    decoration={markings,mark=at position 0.7 with
    {\arrow[ultra thick,blackred, rotate=0]{>}}}, postaction={decorate}
    ] node {} (c);

\path (a) edge [blackred,thick, bend right=-50,
    decoration={markings,mark=at position 0.8 with
    {\arrow[ultra thick,blackred, rotate=0.1]{>}}}, postaction={decorate}
    ] node {} (g);

\path (c) edge [blackred,thick, bend right=-3,
    decoration={markings,mark=at position 0.99 with
    {\arrow[ultra thick,blackred, rotate=0]{>}}}, postaction={decorate}
    ] node {} (e);

\path (c) edge [blackred,thick, bend right=-3,
    decoration={markings,mark=at position 0.99 with
    {\arrow[ultra thick,blackred, rotate=0]{>}}}, postaction={decorate}
    ] node {} (f);

\path (d) edge [blackred,thick, bend right=3,
    decoration={markings,mark=at position 0.99 with
    {\arrow[ultra thick,blackred, rotate=0]{>}}}, postaction={decorate}
    ] node {} (f);

\path (f) edge [blackred,thick, bend right=3,
    decoration={markings,mark=at position 0.7 with
    {\arrow[ultra thick,blackred, rotate=0]{>}}}, postaction={decorate}
    ] node {} (g);
    
\draw (0,1) node {$\hat{u}$};
\draw (5.5,1) node {$\hat{v}$};
\draw (7.5,1) node {$\hat{y}$};
\end{tikzpicture}
    \end{center}
    \caption{\label{fig:sink:source:decomposition}Decomposition with marked source-like,
    sink-like and isolated components.}
\end{figure}
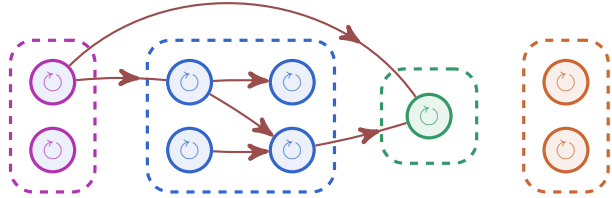

  The above decomposition can be turned into a functional equation with several intermediate conversions between Exponential and Graphic GFs:
  \[
    \D(z; 1+\hat{u}, 1+\hat{v}, 1+\hat{u}+\hat{v}+\hat{y}, t)
    =
    e^{\hat{y}t \cdot \SCC(z)} \cdot
    \Robin{2}^{-1}
    \left(
      \Robin{2} \big(e^{\hat{u}t \cdot \SCC(z)}\big)
      \cdot
      \widehat \D(z; t)
      \cdot
      \Robin{2} \big(e^{\hat{v}t \cdot \SCC(z)}\big)
    \right)
    \, ,
  \]
  where
  \( \widehat \D(z;t) = \widehat \D(z; 1, 1, 1, t) \),
  is the Graphic GF for digraphs with marking variable \( t \) for the total number of strongly connected components.
  The above equation can be solved by putting
  \( \hat{u} = u-1 \), \( \hat{v} = v-1 \), \( \hat{y} = y - u - v + 1 \)
  and substituting \( \widehat \D(z;t) \) from~\cref{corollary: GGF of digraphs marking components},
  which completes the proof.
\end{proof}

\begin{remark}
  Following de Panafieu and Dovgal~\cite{de2019symbolic}, relations \eqref{eq: GGF of digraphs marking components and source-like components}-\eqref{eq: complete multivariate Graphic GF for digraphs} can be rewritten with the help of the exponential Hadamard product in the following way:
  \[
    \widehat{\D}(z; s, t) =
    \dfrac
      {e^{(s-1)t \cdot \SCC(z)} \odot_z \Set(z)}
      {e^{-t \cdot \SCC(z)} \odot_z \Set(z)}
    \, , \qquad\qquad
    \widehat{\D}(z; t) =
    \dfrac
      {1}
      {e^{-t \cdot \SCC(z)} \odot_z \Set(z)}
  \]
  and
  \begin{multline*}
    \D(z; u, v, y, t) =\\
    e^{(y-u-v+1)t \cdot \SCC(z)}
      \left(
        \left[
          \dfrac{
            \big(e^{(u-1)t \cdot \SCC(z)} \odot_z \Set(z)\big)
            \big(e^{(v-1)t \cdot \SCC(z)} \odot_z \Set(z)\big)
          }{
            e^{-t \cdot \SCC(z)} \odot_z \Set(z)
          }
        \right]
        \odot_z \G(z)
      \right)
    \, .
  \end{multline*}
\end{remark}

\subsubsection{Enumeration of acyclic, strongly connected and semi-strong digraphs}

  Here, we recall exact enumeration results for directed acyclic graphs, strongly connected digraphs and semi-strong digraphs.
  They appeared for the first time in the paper of Robinson~\cite{robinson1973counting},
  and later were rewritten in terms of the exponential Hadamard product by de Panafieu and Dovgal~\cite{de2019symbolic}.

\begin{proposition}[{\cite[Corollary 1]{robinson1973counting}}]
\label{proposition: graphic gf of dag}
  The Graphic GF \( \widehat{\DAG}(z) \) of acyclic digraphs is given by
  \[
    \widehat{\DAG}(z)
    =
    \dfrac{1}{\Robin{2} (e^{-z})}
    \, .
  \]
\end{proposition}
\begin{proof}
  Apply Proposition~\ref{proposition: digraphs with given scc} to the family \( \mathcal A \) consisting of a single vertex and put \( s = 1 \).
\end{proof}

\begin{proposition}[{\cite[Corollary 3.5]{de2019symbolic}}]
\label{proposition: exponential gf of scc}
  The Exponential GF \( \SCC(z) \) of strongly connected digraphs is given by
  \[
    \SCC(z) = - \log \left(
      \G(z) \odot \dfrac{1}{\G(z)}
    \right)
    \, .
  \]
\end{proposition}

\begin{corollary}
\label{corollary: exponential gf of ssd and scc}
  The Exponential GFs \( \E(z) \) and \( \SCC(z) \) of semi-strong digraphs and strongly connected digraphs, respectively, are given by
  \begin{equation}\label{eq: semi-strong as sequences}
    \E(z) = \dfrac{1}{1 - \I(z) \odot \G(z)}
    \qquad \text{and} \qquad
    \SCC(z) = \log \dfrac{1}{1 - \I(z) \odot \G(z)}
    \, .
  \end{equation}
\end{corollary}
\begin{proof}
  It is sufficient to apply Proposition~\ref{proposition: exponential gf of scc} and formulae~(\ref{eq:SSD=exp(SCD)}-\ref{eq:G=T}).
\end{proof}

\begin{notation}
\label{notation: dag, scd, ssd}
  We designate by \( \dgg_n \), \( \scc_n \) and \( \ssd_n \), respectively,
  the numbers of acyclic, strongly connected and semi-strong digraphs on \( n \) vertices, so that
  \[
    \widehat{\DAG}(z)
     =
    \sum\limits_{n=0}^\infty \dgg_n
     \dfrac{z^n}{2^{\binom{n}{2}} n!}
    \, ,\qquad
    \SCC(z)
     =
    \sum\limits_{n=0}^\infty \scc_n
     \dfrac{z^n}{n!}
    \, ,\qquad
    \E(z)
     =
    \sum\limits_{n=0}^\infty \ssd_n
     \dfrac{z^n}{n!}
    \, .
  \]
\end{notation}

\begin{remark}
\label{remark: meaning of G^(-1)}
  While \cref{proposition: exponential gf of scc} also holds in the case of marked edges,
  a proper generalization of \cref{corollary: exponential gf of ssd and scc} is more subtle.
  Behind the proof of \cref{corollary: exponential gf of ssd and scc}, we can see an additional combinatorial meaning of the function~\( \big(\G(z)\big)^{-1} \), namely,
  \begin{equation}\label{eq:1/G=1-IT}
    \dfrac{1}{\G(z)} = 1 - \I(z)
  \end{equation}
  (the proof of~\eqref{eq:1/G=1-IT} follows immediately from relations~\eqref{eq:T=1/(1-IT)} and~\eqref{eq:G=T}).
  While formula~\eqref{eq:G=exp(CG)} is directly generalized to
  \[
    \G(z,w) = e^{\CG(z,w)}
  \]
  by marking edges, the way to generalize relation~\eqref{eq:T=1/(1-IT)} is hidden.
  The natural analogue for tournaments would be marking \emph{descents},
  that is, directed edges \( (s,t) \) whose labels satisfy \( s  < t \).
  However, as it was shown in~\cite{archer2020counting}, this approach fails, since relation~\eqref{eq:T=1/(1-IT)} remains true only
  if we replace Exponential GFs \( \T(z) \) and \( \I(z) \) with the so-called \emph{Eulerian GFs}.
  
  Nevertheless, it turns out that a required generalization is possible~\cite{monteilANTISEQ}.
  The key idea is to pass to another tournament analogue of the Erd\H{o}s-R\'enyi model \( \mathbb G(n,p) \) for random graphs discussed in \cref{remark:introducing Erdos-Renyi}.
  In order to do so, for \( p\in[1/2,1] \) and \( q = 1-p \), we introduce a \emph{tournament with ties} of size \( n \) as a directed graph with the set of vertices \( [n] \).
  Each pair of vertices \( x \) and \( y \) in this graph, independently of other pairs,
  is joined by the edge \( (x,y) \), the edge \( (y,x) \), and by both of them
  with probability \( 1-p \), \( 1-p \), and \( 2p-1 \), respectively.
  Putting the weight of such a tournament to be
  \[
    (w-1)^m = \left(\dfrac{2p-1}{1-p}\right)^m
    \, ,
  \]
  where \( m \) is the total number of \emph{ties} in the tournament (that is, pairs of vertices joined by two directed edges) and \( w = \frac{p}{1-p} \),
  we have the total weight of all tournaments with ties of size \( n \) equal to \( (1-p)^{-\binom{n}{2}} \),
  and
  \[
    \T(z;w) = \dfrac{1}{1 - \I(z;w)}
    \, .
  \]
  However, it is important to emphasize that it is not the variable \( w \) that counts ties now, but \( u = w-1 \).
\end{remark}

\subsection{2-SAT formulae}
\label{section:2sat}

  In this section, we shortly describe the symbolic method for enumeration of 2-SAT formulae and implication digraphs recently developed in~\cite{dovgal2021exact}.
  We start by recalling the model and related notions, such as implication digraphs and contradictory components.
  Next, we discuss the Implication GFs employed for enumeration purposes and their connections to other types of generating functions.
  Finally, we give exact expressions for generating functions of satisfiable formulae, contradictory strongly connected implication digraphs and implication digraphs with marking variables for contradictory and ordinary strongly connected components.

\subsubsection{Definitions and basic properties}

  A $k$-\emph{conjunctive normal form} ($k$-CNF) formula with \( n \) Boolean variables
  \( \{ x_1, \ldots, x_n \} \)
  and \( m \) clauses is a~conjunction of the form
  \[
	\bigwedge\limits_{i=1}^m (c_{i1} \lor \ldots \lor c_{ik}),
  \]
  where each of the \emph{literals} \( c_{ij} \) belongs to the set
  \( \{ x_1, \ldots, x_n, \n x_1, \ldots, \n x_n \} \).
  These formulae are also called \( k \)-SAT formulae.
  A formula is \emph{satisfiable} if it takes a \textsf{True} value under at least one variable assignment.

  In the case \( k = 2 \), a 2-CNF can be mapped to a so-called \emph{implication digraph}.
  The vertices of this digraph are the literals \( \{ x_1, \ldots, x_n, \n x_1, \ldots, \n x_n \} \).
  Each clause \( u \lor v \) corresponds to two edges \( \n u \to v \) and \( \n v \to u \) in the implication digraph.
  The meaning of each edge is a logical implication under satisfiability.
  We assume that there is no clause of types \( (x \lor x) \) and \( (x \lor \n x) \) in our model.
  Also, we suppose that each clause in a 2-CNF can occur at most once.
  As a consequence, the corresponding implication digraph has no loops and no multiple edges.

  It is now a well-known property of the 2-SAT problem (see~\cite{aspvall1979linear}) that
  a formula is not satisfiable if and only if there is a \emph{contradictory variable},
  \ie a pair of literals \( x \) and \( \n x \) such that there exists a directed path from \( x \) to \( \n x \) and from \( \n x \) to \( x \).
  It can be shown that if there is at least one contradictory variable inside a strongly connected component of the implication digraph,
  then all the variables included into this component are contradictory.
  We say that a \emph{component} is \emph{contradictory} if it contains at least one contradictory variable, and \emph{ordinary} otherwise.

\subsubsection{Implication generating function}

  Let \( (a_n)_{n=0}^\infty \) be the counting sequence of a class \( \mathcal A \) of 2-SAT formulae,
  \ie \( a_n \) denotes the number of Boolean 2-SAT formulae from \( \mathcal A \) with \( n \) variables.
  The \emph{Implication GF} of \( \mathcal A \)
is
  \begin{equation}\label{eq:ImplicationGF}
    \ddot A(z)
     :=
    \sum_{n = 0}^\infty a_n \dfrac{z^n}{2^{n^2} n!}
    \, .
  \end{equation}
  For example, the Implication GF of all 2-SAT formulae is
  \begin{equation}\label{eq:CNF:IGF}
    \CNF(z)
     =
    \sum_{n = 0}^\infty 2^{n(n-1)} \dfrac{z^n}{2^n n!}
     =
    D(z/2)
    \, .
  \end{equation}
  Clearly, the Exponential GF and Implication GF of the same class can be expressed in terms of each other:
  \begin{equation}\label{eq:EGF via IGF}
    A(z)
     = 
    \Robin{2}^{-2} \Big(\ddot A(2z) \Big)
     = 
    \ddot A(z) \odot \D(2z)
  \end{equation}
  and
  \begin{equation}\label{eq:IGF via EGF}
    \ddot A(z)
     = 
    \Robin{2}^2 \Big( A\left(z/2\right) \Big)
     = 
    A(z) \odot \dSet(z)
    \, ,
  \end{equation}
  where, in notations of de Panafieu, Dovgal and Ravelomanana \cite{dovgal2021exact},
  \[
    \dSet(z) := \sum_{n=0}^\infty
      \dfrac{z^n}{2^{n^2} n!}
  \]
  is the Implication GF of digraphs without vertices.

  Similarly to Exponential GFs and Graphic GF, the design of Implication GFs comes from enumerative needs.
  We will not discuss this issue here.
  For more detailed information, we refer the reader to~\cite{dovgal2021exact}, in particular to Proposition 3.11.

  If we introduce a marking variable \( w \) for the total number of clauses, 
  then the bivariate Implication GF takes the following form:
  \[
    \ddot A(z, w) := \sum_{n=0}^\infty a_n(w)
      \dfrac{z^n}{(1 + w)^{n(n-1)}2^n n!}
    \, .
  \]
  In particular,
  \[
    \dSet(z,w) := \sum_{n=0}^\infty
      \dfrac{z^n}{(1 + w)^{n(n-1)}2^n n!}
    \, ,
  \]
  which is used for the bivariate analogue of~\eqref{eq:IGF via EGF}.

\subsubsection{Enumeration of 2-SAT}

\begin{proposition}[{\cite[Proposition 4.5]{dovgal2021exact}}]
\label{proposition:cnf:types}
  Let \( \mathcal S \) and \( \mathcal C \) be two families of strongly connected digraphs whose Exponential GFs are \( S(z) \) and \( C(z) \), respectively.
  Then the Implication GF \( \CNF_{\mathcal S, \mathcal C}(z) \) of implication digraphs
  whose ordinary and contradictory strongly connected components belong to families \( \mathcal S \) and \( \mathcal C \), respectively,
  is given by
  \[
    \CNF_{\mathcal S, \mathcal C}(z)
     =
    \dfrac
      {\Robin{2}^2 \big( e^{C(z/2)-S(z)/2} \big)}
      {\Robin{2} \big( e^{-S(z)} \big)}
    \, .
  \]
\end{proposition}

\begin{corollary}[{\cite[Theorem 4.6]{dovgal2021exact}}]
\label{proposition:igf:sat}
  The Implication GF \( \SAT(z) \) of satisfiable 2-CNFs is given by
  \begin{equation}\label{eq:igf:sat}
    \SAT(z) = \G(z) \cdot
      \Robin{2}^2
      \left(
        e^{-\tfrac12 \SCC(z)}
      \right)
    \, .
  \end{equation}
\end{corollary}
\begin{proof}
  Apply \cref{proposition:cnf:types} with \( \mathcal C = \emptyset \) and \( \mathcal S \) consisting of all strongly connected digraphs
  and take into account that 
  \(
    \Robin{2} \big( e^{-\SCC(z)} \big) = G^{-1}(z)
  \)
  by \cref{proposition: exponential gf of scc}. 
\end{proof}

\begin{corollary}[{\cite[Theorem 4.8]{dovgal2021exact}}]
\label{proposition:egf:cscc}
  The Exponential GF \( \CSCC(z) \) of contradictory strongly connected implication digraphs is given by
  \begin{equation}\label{eq:egf:cscc}
    \CSCC(z) = \dfrac{1}{2} \SCC(2z) +
    \log \left(
      \Robin{2}^{-2}
       \Big( \D(z) \big( 1 - \I(2z) \big) \Big)
    \right)
    \, .
  \end{equation}
\end{corollary}
\begin{proof}
  Apply \cref{proposition:cnf:types} with \( \mathcal C \) and \( \mathcal S \) consisting of all contradictory strongly connected and all strongly connected digraphs, respectively,
  and use \cref{corollary: exponential gf of ssd and scc} and relation~\eqref{eq:CNF:IGF}.
\end{proof}

\begin{corollary}
\label{corollary:cnf:types}
  If \( s \) marks the number of contradictory strongly connected components
  and \( t \) marks the number of pairs of ordinary strongly connected components
  in the corresponding implication digraph,
  then the multivariate Implication GF \( \CNF(z; s, t) \) of 2-CNFs is given by
  \begin{equation}\label{eq:cnf:types}
    \CNF(z; s, t)
     =
    \dfrac
      {\Robin{2}^2 \big( e^{s \cdot \CSCC(z/2) - t/2 \cdot \SCC(z)} \big)}
      {\Robin{2} \big( e^{-t \cdot \SCC(z)} \big)}
    \, .
  \end{equation}
\end{corollary}
\begin{proof}
  The statement is obtained by equipping the Exponential GFs of the strongly connected components from \cref{proposition:cnf:types} with marking variables for their weights.
  In other words, we put \( C(z) = s \cdot \CSCC(z) \) and \( S(z) = t \cdot \SCC(z) \).
\end{proof}

\begin{remark}
  Initially, in~\cite{dovgal2021exact}, relations \eqref{eq:igf:sat}-\eqref{eq:cnf:types} were stated in terms of the exponential Hadamard product, respectively, as
  \[
    \SAT(z) = \G(z) \cdot
      \left(
        e^{-\tfrac12 \SCC(z)} \odot \dSet(2z)
      \right)
    \, ,
  \]
  \[
    \CSCC(z) = \dfrac{1}{2} \SCC(2z) +
    \log \left(
      \D(z) \odot \dfrac{\D(z)}{\G(2z)}
    \right)
  \]  
  and
  \[
   \CNF(z; s, t)
    =
    \dfrac
        {e^{s \cdot \CSCC(z) - t/2 \cdot \SCC(2z)} \odot_z \dSet(z)}
        {e^{-t \cdot \SCC(z)} \odot_z \Set(z)}
    \, .
  \]
\end{remark}

%
%

\section{Asymptotic transfers for graphically divergent series}\label{section:transfers}

  This section is devoted to the study of graphically divergent series and the corresponding Coefficient~GFs.
  We show that the graphically divergent series admit a ring structure, and that the asymptotic transfers are consistent with ring operations and compositions with analytic functions,
  as well as with several other operations such as differentiation and integration.
  Moreover, the concept of asymptotic transfer can be extended to the case of marking variables.

\subsection{Proofs of the main asymptotic transfer properties}
\label{sec: transfers-proofs}

  The goal of this section is to provide proofs of \cref{lemma: operation transfers} and~\cref{lemma: Bender's transfer}.
  In the case of compositions with analytic functions, a common tool for establishing asymptotic expansions is Bender's theorem~\cite{bender1975asymptotic} (see \cref{theorem: Bender's}).
  To simplify the presentation, we introduce \emph{gargantuan sequences}, which are implicitly employed in its statement as one of its hypotheses.
  We show that all counting sequences under consideration are gargantuan, and therefore, Bender's theorem is applicable.
  The concept of gargantuan sequences is also useful for the product of two series, although in this case the proof is rather straightforward.

\begin{definition}
\label{def: gargantuan sequence}
  We will call a sequence \( (a_n) \) \emph{gargantuan}, if for any positive integer $R$ the following two conditions hold, as \( n\to\infty \):
 \[
   \mbox{ (i) }\quad
     \dfrac{a_{n-1}}{a_n}\to0,\,\,
   \qquad\qquad\qquad
   \mbox{ (ii) }\quad
     \sum\limits_{k=R}^{n-R}|a_ka_{n-k}| =
     \BigO\big(a_{n-R}\big).
 \]
\end{definition}

\begin{lemma}
\label{lemma:log-convexity}
  Let \( \alpha \in \mathbb R_{> 1} \),
  \( \beta \in \mathbb Z_{>0} \) and
  \( m, \ell \in \mathbb Z_{\geq 0} \)
  be fixed numbers. Suppose that
  \[
    c_n \sim \dfrac{\alpha^{\beta \binom{n}{2}} \alpha^{-mn} n^\ell}{n!},
  \]
  as \( n \to \infty \).
  Then the sequence \( (c_n) \) is gargantuan.
\end{lemma}
\begin{proof}
  Let \( d_n \) denote
  \( \alpha^{\beta \binom{n}{2}-mn} n^\ell/n! \).
  Since \( c_n \sim d_n \) as \( n \to \infty \),
  we have \( c_n \leq 2d_n \) for large enough~\( n \).
  Furthermore, \( d_n = o(d_{n+1}) \) as \( n \to \infty \),
  and hence, \( c_n = o(c_{n+1}) \) as \( n \to \infty \).
  Thus, condition (i) of \cref{def: gargantuan sequence} holds.
  In particular, \( d_n \leq d_{n+1} \) for large enough \( n \).
  
  In order to verify condition (ii), we prove that the maximum of the product \( d_kd_{n-k} \) over \( k\in[R,n-R] \)
  is attained on the boundary of the interval \( [R,n-R] \), as \( n \) is sufficiently large.
  For this aim, consider the function
  \[
    d \colon \mathbb R_{>0} \to \mathbb R,\qquad
    d(x) = \dfrac{\alpha^{\frac{\beta x(x-1)}{2}-mx} x^\ell}{\Gamma(x+1)},
  \]
  coinciding with \( (d_n)_{n=0}^{\infty} \) at positive integers.
  It is sufficient to show that \( \log d(x) \) is a convex function for large enough argument \( x \).
  Indeed, according to \cite[(6.4.12)]{abramowitz1988handbook},
  \[
    \dfrac{\partial^2}{\partial x^2}(\log \Gamma(x)) \sim \frac{1}{x} \, ,
    \quad \text{and therefore,} \quad
    \dfrac{\partial^2}{\partial x^2} \log d(x) = \beta \log \alpha
    + \BigO\left( \dfrac{1}{x} \right)
    \, ,
  \]
  as \( x \to \infty \).
  The latter quantity is positive for large \( x \), since \( \alpha > 1 \) and \( \beta > 0 \).

  Now, let \( N \) be an integer, large enough, but fixed, such that
  the inequalities \( c_k \leq 2 d_k \) and \( d_k \leq d_{k+1} \) hold for all \( k \geq N \).
  Assume further that \( N \geq R + 1 \).
  Then, as \( n \to \infty \), we have
  \begin{align*}
    \sum_{k=R}^{n-R} |c_k c_{n-k}|
    &=
    2\sum_{k=R}^{N-1} |c_k c_{n-k}|
    + \sum_{k=N}^{n-N} |c_k c_{n-k}|
    \\
    &\leq
    4(N-R) \max_{k\in [R,N]} |c_k| \cdot d_{n-R}
    + 4 \sum_{k=N}^{n-N} d_k d_{n-k}
    \\
    &\leq
    \BigO(d_{n-R}) + 4(n-2N+1) d_N d_{n-N}
    \\
    &= \BigO(d_{n-R})
    = \BigO(c_{n-R})\, .
  \end{align*}
\end{proof}

\begin{lemma}
\label{lemma: technical lemma for product}
  If a sequence \( (a_n) \) is gargantuan
  and a sequence \( (b_n) \) satisfies
  \( b_n = \BigO(a_n) \), as \( n \to \infty \),
  then, for any positive integer \( R \),
  \[
    \sum_{k=R}^{n-R} |b_k a_{n-k}| = \BigO(a_{n-R}),
  \]
  as \( n \to \infty \).
\end{lemma}
\begin{proof}
  The proof of this statement is direct.
  Since \( b_n = \BigO(a_n) \),
  there exist constants \( C \in \mathbb R_{>0} \) and \( N \in \mathbb Z_{>0} \)
  such that, for \( k \geq N \),
  we have \( |b_k| \leq C |a_k| \).
  Hence, the sum can be split into two parts:
  \begin{align*}
    \sum_{k=R}^{n-R} |b_k a_{n-k}| &\leq
    \sum_{k=R}^{N-1} |b_k a_{n-k}|
    + C\sum_{k=N}^{n-R} |a_k a_{n-k}|
    \\
    &\leq |b_R a_{n-R}|
    + \max_{k \in [R, N-1]} |b_k|
    \cdot (N-R-1) \cdot o(a_{n-R})
    + C\sum_{k=R}^{n-R} |a_k a_{n-k}|
    \\
    &= \BigO(a_{n-R}) \, .
  \end{align*}
\end{proof}

\begin{proof}[Proof of \cref{lemma: operation transfers}]
  Let \( A,B\in\Ring{\alpha}{\beta} \).
  In this case, formula~\eqref{eq: sum-transfer} holds, since
  \[
    a_n+b_n \approx \alpha^{\beta \binom{n}{2}}
    	\left[
      \sum_{m\geq\M} \alpha^{-mn}
      \sum_{\ell=0}^{\infty}
      	n^{\underline{\ell}}
      	(a_{m,\ell}^\circ + b_{m,\ell}^\circ)
    	\right]
    	\, ,
  \]
  where \( M \) is the minimum of the constants corresponding to expansions of \( a_n \) and \( b_n \).
  
  To get~\eqref{eq: prod-transfer}, note that,
  according to \cref{lemma:log-convexity},
  the sequences \( (a_n/n!) \) and \( (b_n/n!) \) are gargantuan.
  Fixing a positive integer \( R \), we have
  \[
    n![z^n] A(z) B(z) =
    \sum_{k=0}^{R-1} \binom{n}{k}
      \big(a_k b_{n-k} + b_k a_{n-k}\big) +         	\sum_{k=R}^{n-R} \binom{n}{k} b_k a_{n-k}
    	\, .
  \]
  Due to \cref{lemma: technical lemma for product}, the second sum is negligible.
  Thus, it is sufficient to rewrite the first sum in asymptotic form~\eqref{eq:expansion}
  and verify that its coefficients coincide with those of the right-hand side of~\eqref{eq: prod-transfer}.
  Given a fixed integer \( k \),
  the asymptotics of the shifted sequence \( (a_{n-k}) \), as \( n\to\infty \), is
  \begin{align*}
    \binom{n}{k} b_k a_{n-k}
    &\approx
    \dfrac{b_k}{k!}
    \alpha^{\beta \binom{n-k}{2}}
      \sum_{m\geq M} \alpha^{-m(n-k)}
        \sum_{\ell \geq 0}
          n^{\underline{k}}
          (n-k)^{\underline{\ell}}
          a_{m,\ell}^\circ
    \\&\approx
    \dfrac{b_k}{k!}
    \alpha^{\beta \binom{n}{2}+\beta \binom{k}{2} + \beta k}
    \sum_{m \geq M} \alpha^{-mn-\beta kn}
      \sum_{\ell \geq 0}
        n^{\underline{\ell+k}}
        a_{m,\ell}^\circ \alpha^{mk}
    \\&\approx
    \dfrac{b_k}{k!}
    \alpha^{\beta \binom{n}{2}+\beta \binom{k}{2} + \beta k}
    \sum_{m \geq M + \beta k} \alpha^{-mn}
      \sum_{\ell \geq k}
        n^{\underline{\ell}}
        a_{m-\beta k,\ell-k}^\circ
        \alpha^{(m-\beta k)k}
    \\&\approx
    \alpha^{\beta \binom{n}{2}}
    \sum_{m \geq M + \beta k} \alpha^{-mn}
      \sum_{\ell \geq k}
        n^{\underline{\ell}}
        a_{m-\beta k,\ell-k}^\circ
        \dfrac{b_k \alpha^{km}}{k! \alpha^{\beta \binom{k}{2}}}
    \, .
  \end{align*}
  At the same time, by expanding a product of the form
  \(
    B \big(
      \alpha^{\frac{\beta+1}{2}}z^\beta w
    \big)
     \cdot
    (\Convert{\alpha}{\beta} A)(z, w)
  \), we obtain:
  \begin{align*}
	&
    \alpha^{\frac{1}{\beta} \binom{m}{2}}
    [z^m w^\ell]
    \left(
      \sum_{k=0}^{\infty}
        b_k
        \dfrac{ \big(
          \alpha^{\frac{\beta+1}{2}} z^\beta w
        \big)^k }{k!}
    \right)
    \left(
      \sum_{r=M}^{\infty}
        \sum_{s=0}^{\infty}
          a_{r,s}^\circ
          \dfrac{z^r}{\alpha^{\frac{1}{\beta} \binom{r}{2}}} w^s
    \right)
    \\ &=
    \sum_{k \geq 0}
      \dfrac{b_k}{k!}
      a^\circ_{m-\beta k, \ell-k}
      \alpha^{
        \frac{(\beta+1)k}{2}
        + \frac{1}{\beta}\binom{m}{2}
        - \frac{1}{\beta}\binom{m-\beta k}{2}
        }
    \\ &=
    \sum_{k \geq 0}
      \dfrac{b_k \alpha^{km}}{k! \alpha^{\beta \binom{k}{2}}}
      a^\circ_{m-\beta k, \ell-k}
    \, .
  \end{align*}
  That is why, comparing this expression with the previous one
  summed up over \( k \) such that \( k \leq \ell \) and \( M+\beta k \leq m\),
  we conclude that formula~\eqref{eq: prod-transfer} is valid.
\end{proof}

  To prove \cref{lemma: Bender's transfer}, we use Bender's theorem~\cite{bender1975asymptotic}
  that commonly serves to provide asymptotic expansions for divergent formal power series.
  We cite here an adaptation of his theorem originally presented in a more general form.

\begin{proposition}[{\cite[Theorem 2]{bender1975asymptotic}}]\label{theorem: Bender's}
  Consider a formal power series
  \[
    A(z) = \sum\limits_{n=1}^{\infty}a_nz^n
  \]
  and a function $F(x)$, which is analytic in some neighborhood of origin.
  Define
  \[
    B(z) = \sum\limits_{n=0}^{\infty}b_nz^n =
    F\big(A(z)\big)
    \qquad\qquad
    \mbox{and}
    \qquad\qquad
     C(z) = \sum\limits_{n=0}^{\infty}c_n z^n =
     \left.     
       \dfrac{\partial}{\partial x} F(x)
     \right|_{x=A(z)}
    \, .
  \]
  Assume that the sequence $(a_n)$ is gargantuan and $a_n\ne 0$ for any positive integer $n>0$.
  Then
  \[
    b_n \approx \sum\limits_{k\geq0}c_ka_{n-k}
  \]
  and the sequence $(b_n)$ is gargantuan.
\end{proposition}

\begin{proof}[Proof of \cref{lemma: Bender's transfer}]
  Let \( A\in\Ring{\alpha}{\beta} \) and \( F \) be a~function analytic in a neighbourhood of the origin.
  According to \cref{theorem: Bender's} 
  (which is applicable due to \cref{lemma:log-convexity}),
  \begin{equation}
   	\label{eq: Bender's transfer proof}
    n![z^n] (F\circ A)(z) \approx
    \sum_{k \geq 0} \binom{n}{k} \eta_k a_{n-k}
    	\, ,
  \end{equation}
  where \( \eta_k \) are the coefficients of the \( H(z) = \partial_x F(x)|_{x=A(z)} \),
  \[
    H(z) = \sum\limits_{k=0}^{\infty}
    \eta_k \dfrac{z^k}{k!}
    \, .
  \]
  Now we follow the scheme of the proof of \cref{lemma: operation transfers}.
  Namely, we rewrite~\eqref{eq: Bender's transfer proof} in asymptotic form~\eqref{eq:expansion}
  and verify that its coefficients are equivalent to those of
  \(
    H \big(
      \alpha^{\frac{\beta+1}{2}} z^\beta w
    \big) \cdot
    (\Convert{\alpha}{\beta} A)(z,w).
  \)
  Indeed, for fixed \( k\in\mathbb Z_{>0} \) and \( n\to\infty \), we have
  \[
    \binom{n}{k} \eta_k a_{n-k} \approx
    \alpha^{\beta \binom{n}{2}}
    \sum_{m \geq M + \beta k} \alpha^{-mn}
      \sum_{\ell \geq k}
        n^{\underline{\ell}}
        a_{m-\beta k,\ell-k}^\circ
        \dfrac{\eta_k\alpha^{km}}{k!\alpha^{\beta \binom{k}{2}}}
  \]
  and
  \[
    \alpha^{\frac{1}{\beta} \binom{m}{2}}
    [z^m w^\ell]
    H \big(
      \alpha^{\frac{\beta+1}{2}}z^\beta w
    \big)
    (\Convert{\alpha}{\beta} A)(z, w) =
    \sum_{k \geq 0}
      \dfrac{\eta_k \alpha^{km}}{k! \alpha^{\beta \binom{k}{2}}}
      a^\circ_{m-\beta k, \ell-k}
    \, .
  \]
  Hence, formula~\eqref{eq: Bender-transfer} is valid.
\end{proof}
\subsection{Other transfer properties}
\label{sec: other properties}

  In this section, we discuss four more properties of the asymptotic transfer \( \Convert{\alpha}{\beta} \).
  First, we express explicitly the behavior of \( \Convert{\alpha}{\beta} \) with respect to powers of a series \( A \in \Ring{\alpha}{\beta} \).
  Next, we study its relations with linear change of variable \( z \mapsto \alpha z \).
  Finally, we discuss how the asymptotic transfer \( \Convert{\alpha}{\beta} \) can be applied together with the operations of differentiation and integration.

\begin{lemma}
\label{lemma: power transfer}
  If \( \alpha \in \mathbb R_{>1} \), \( \beta \in \mathbb Z_{>0} \) and \( A \in \Ring{\alpha}{\beta} \),
  then \( A^m \in \Ring{\alpha}{\beta} \) for each \( m\in\mathbb Z_{\geq0} \) with
  \begin{equation}\label{eq: power-transfer}
    \big(
	  \Convert{\alpha}{\beta} A^m
    \big)(z, w)
      =
    m \cdot A^{m-1} \big(
      \alpha^{\frac{\beta+1}{2}} z^\beta w
    \big) \cdot
    (\Convert{\alpha}{\beta} A)(z,w)
    \, .
  \end{equation}
  If, additionally, \( a_0 = 1 \), then \eqref{eq: power-transfer} holds for any \( m\in\mathbb Q \).
\end{lemma}
\begin{proof}
  For non-negative integer \( m \), this follows from~\eqref{eq: prod-transfer} by induction.
  In the other cases, we apply \cref{lemma: Bender's transfer} to the series \( \big(A(z) - 1\big) \) and the function
  \(
    F(x) = (1+x)^m.
  \)
\end{proof}

\begin{lemma}
\label{lemma: transfer shift}
  Let \( \alpha \in \mathbb R_{>1} \), \( \beta \in \mathbb Z_{>0} \),
  and \( A, B \in \Ring{\alpha}{\beta} \).
  If \( B(z) = A(\alpha^d z) \) for some \( d \in \mathbb Z \), then
  \[
	(\Convert{\alpha}{\beta} B)(z,w) = 
	\dfrac{
	  (\Convert{\alpha}{\beta} A)(\alpha^{d/\beta}z,w)
	}
	{
	  \alpha^{\frac{1}{\beta} \binom{d+1}{2}} z^d
	}
    \, .
  \]
\end{lemma}
\begin{proof}
  If~\eqref{eq:expansion} holds for the coefficients \( (a_n)_{n=0}^\infty \) of \( A(z) \),
  then the coefficients of \( B(z) \) satisfy
  \[
    b_n := \alpha^{dn} a_n \approx
     \alpha^{\beta \binom{n}{2}}
    	\left[
      \sum_{m\geq\M-d} \alpha^{-mn}
      \sum_{\ell=0}^{\infty}
	  	n^{\underline{\ell}} \,
    		a_{m+d,\ell}^\circ
    	\right]
	\, .
  \]
  Hence, using \cref{definition: asymptotics} and the relation
  \(
    \binom{m-d}{2} =
    \binom{m}{2} - md + \binom{d+1}{2},
  \)
  we obtain
  \[
    (\Convert{\alpha}{\beta} B)(z, w) =
    \dfrac{z^{-d}}{\alpha^{\frac{1}{\beta} \binom{d+1}{2}}}
    \sum_{m = \M}^\infty
    \sum_{\ell=0}^\infty
      a_{m,\ell}^\circ
      \dfrac{(\alpha^d z)^m}{\alpha^{\frac{1}{\beta} \binom{m}{2}}}
      w^\ell
    \, ,
  \]
  which implies the statement of the lemma.
\end{proof}

\begin{corollary}
\label{corollary: transfer shift}
  Let \( \alpha \in \mathbb R_{>1} \), \( \beta \in \mathbb Z_{>0} \),
  and \( A, B, C \in \Ring{\alpha}{\beta} \).
  If \( B(z) = A(\alpha z) \) and \( C(z) = A(z / \alpha) \), then
  \[
	(\Convert{\alpha}{\beta} B)(z,w) = 
	\alpha^{-1/\beta} z^{-1} \cdot
    (\Convert{\alpha}{\beta} A)(\alpha^{1/\beta}z,w)
	\qquad\mbox{and}\qquad
	(\Convert{\alpha}{\beta} C)(z,w) = z \cdot 
	(\Convert{\alpha}{\beta} A)(\alpha^{-1/\beta}z,w)
    \, .
  \]
\end{corollary}
\begin{proof}
  Apply \cref{lemma: transfer shift} for \( d = 1 \) and \( d = -1 \), respectively.
\end{proof}

\begin{proposition}
\label{lemma: derivative transfer}
  If \( \alpha \in \mathbb R_{>1} \), \( \beta \in \mathbb Z_{>0} \) and \( A \in \Ring{\alpha}{\beta} \), then
  \[
	(\Convert{\alpha}{\beta} A')(z,w) = 
	\dfrac{
	  (\Convert{\alpha}{\beta} A)(z,w) +
	  \frac{\partial}{\partial w}(\Convert{\alpha}{\beta} A)(z,w)
	}
	{
	  \alpha^{\frac{\beta+1}{2}} z^\beta
	}
	\, .
  \]
  and
  \[
	\left(\Convert{\alpha}{\beta} \int A\right)(z,w) = 
    \alpha^{\frac{\beta+1}{2}} z^\beta
	\sum\limits_{k=0}^{\infty}
	  (-1)^k \frac{\partial^k}{\partial w^k}\big(\Convert{\alpha}{\beta} A\big)(z,w)
	\, .
  \]
\end{proposition}
\begin{proof}
  Similarly to the proof of \cref{lemma: transfer shift}, this follows from \cref{definition: asymptotics} with the help of direct calculations.
  Given coefficients \( (a_n)_{n=0}^\infty \) satisfying~\eqref{eq:expansion},
  the idea is to express asymptotics expansions of \( a_{n+1} \) and \( a_{n-1} \) in the same form as the initial sequence.
  That can be done due to relations
  \[
    (n+1)^{\underline{\ell}} =
    n^{\underline{\ell}} + \ell n^{\underline{\ell-1}}
  \]
  and
  \[
    (n-1)^{\underline{\ell}} =
    \sum\limits_{k=0}^\ell
  	  (-1)^{\ell-k}
  	  \ell^{\underline{\ell-k}}
  	  n^{\underline{k}}
  	\, .
  \]
  Since the calculations are straightforward, we allow ourselves to omit the details.
\end{proof}

\subsection{Relations between different rings of graphically divergent series}
\label{sec: transfers-relations}

  The goal of this section is to study relations between rings \( \Ring{\alpha}{\beta} \) for a fixed parameter \( \alpha \) and different values of \( \beta \).
  First, we show that there is a natural inclusion
  \[
    \Ring{\alpha}{1} \subset
    \Ring{\alpha}{2} \subset
    \Ring{\alpha}{3} \subset
    \Ring{\alpha}{4} \subset
    \ldots
  \]
  and that each ring \( \Ring{\alpha}{\beta} \) in this row belongs to the ``kernel'' of the next one,
  meaning that asymptotic coefficients of the elements of \( \Ring{\alpha}{\beta} \) are zeroes with respect to \( \Ring{\alpha}{\beta+1} \).
  To ``compare'' elements of \( \Ring{\alpha}{\beta} \) and \( \Ring{\alpha}{\beta+1} \),
  we introduce a linear operator \( \Robin{\alpha} \) that changes the growth rate of a formal power series, leaving the asymptotic coefficients the same.
  The operator \( \Robin{\alpha} \) comes together with another family of operators, denoted by \( \Change{\alpha}{\beta_1}{\beta_2} \),
  that change the type of a Coefficient GF, leaving its coefficients unchanged.
  The connections between these operators and asymptotic transfers are reflected by the commutative diagram described in \cref{lemma: commutative conversion}.

\begin{lemma}
\label{lemma: ring inclusion}
  If \( \alpha \in \mathbb R_{>1} \)
  and \( \beta_1, \beta_2 \in \mathbb Z_{>0} \),
  such that \( \beta_1 < \beta_2 \),
  then
  \[
	\Ring{\alpha}{\beta_1} \subset \Ring{\alpha}{\beta_2}
    \, .
  \]
  Moreover, for any \( A \in \Ring{\alpha}{\beta_1} \),
  we have
  \[
	(\Convert{\alpha}{\beta_2} A)(z,w) = 0
    \, .
  \]
\end{lemma}
\begin{proof}
  Since, for some integers \( m \) and \( \ell \) and constant $C$,
  \[
    n! [z^n] \sim C \alpha^{\beta_1 \binom{n}{2}} \alpha^{-mn}
    n^{\underline{\ell}} \, ,
  \]
  all the coefficients of expansion~\eqref{eq:expansion} are zeroes if
  \( \alpha^{\beta_2 \binom{n}{2}} \) is chosen as the main term.
\end{proof}

\begin{remark}
\label{remark: conversion kernel}
  Power series whose coefficients grow exponentially or factorially
  also take part of the ring \( \Ring{\alpha}{\beta} \) for any \( \alpha \in \mathbb R_{>1} \) and \( \beta \in \mathbb Z_{>0} \).
  So are power series with non-zero radius of convergence.
  Indeed, if \( A \) is a series of one of the mentioned kind,
  then its expansion coefficients \( a^{\circ}_{m, \ell} \) are all zeroes.
  In particular, \( (\Convert{\alpha}{\beta} A)(z,w) = 0 \).
\end{remark}

\begin{definition}
\label{definition: Robinson operator}
  Let \( \alpha \in \mathbb R_{>1} \) and \( \beta \in \mathbb Z_{>1} \).
  The linear operator
  \( \Robin{\alpha} \colon
  \Ring{\alpha}{\beta} \to
  \Ring{\alpha}{\beta-1} \)
  is defined by
  \begin{equation}\label{eq:Robinson operator}
    \Robin{\alpha}
    \left(
      \sum_{n=0}^\infty a_n \dfrac{z^n}{n!}
    \right)
    :=
    \sum_{n=0}^{\infty}
      \dfrac{a_n}{\alpha^{\binom{n}{2}}} 
      \dfrac{z^n}{n!}
    \, ,
  \end{equation}
  where \( (a_n)_{n=0}^\infty \) satisfy~\eqref{eq:EGF F} and~\eqref{eq:expansion}.
\end{definition}

\begin{remark}
\label{remark: Robinson operator via Hadamard product}
  Taking into account \cref{remark: conversion kernel},
  we can consider \( \Robin{\alpha} \) as an operator
  \( \Ring{\alpha}{\beta} \to \Ring{\alpha}{\beta} \)
  for any positive integer \( \beta \), including \( \beta = 1 \).
  This allows us to make sense of the power \( \Robin{\alpha}^m A \) for any \( A \in \Ring{\alpha}{\beta} \) and \( m \in \mathbb Z \).
  The operator \( \Robin{2} \) that we have seen in~\eqref{eq:Robinson's operator} is the particular case of the above operator \( \Robin{\alpha} \) with \( \alpha = 2 \).
  As well as for \( \Robin{2} \), the action of \( \Robin{\alpha} \) can be expressed in terms of the exponential Hadamard product.
  To this end, in relations~\eqref{eq:GF conversion} we need to replace Exponential GFs \( \G(z) \) and \( \Set(z) \) by their bivariate analogues~\eqref{eq:marking edges in graphs} and~\eqref{eq:marking edges in graphs - 2} taken at \( w=\alpha-1 \):
  \[
    \Robin{\alpha}A(z) = A(z) \odot \Set(z,\alpha-1)
	\qquad \text{and} \qquad
    \Robin{\alpha}^{-1}A(z) = A(z) \odot \G(z,\alpha-1) \, .
  \]
  More generally, for any non-zero integer $m$, we have
  \begin{equation}\label{eq:general GF conversion}
    \Robin{\alpha}^{m}A(z) = A(z) \odot \G(z,\alpha^{-m}-1) \, .
  \end{equation}
  This observation is particularly useful for numerical calculations.
\end{remark}

\begin{lemma}
\label{lemma: QDA = 0}
  If \( \alpha \in \mathbb R_{>1} \),
  \( \beta \in \mathbb Z_{>0} \)
  and \( A \in \Ring{\alpha}{\beta} \),
  then
  \[
	\big(
	  \Convert{\alpha}{\beta} (\Robin{\alpha} A)
	\big)(z,w) = 0
    \, .
  \]
\end{lemma}
\begin{proof}
  This follows directly from \cref{lemma: ring inclusion} and \cref{remark: conversion kernel}.
\end{proof}

\begin{definition}
\label{definition: Coefficient operator}
  Let \( \alpha \in \mathbb R_{>1} \) and \( \beta_1,\beta_2 \in \mathbb Z_{>0} \).
  The linear operator
  \( \Change{\alpha}{\beta_1}{\beta_2} \colon
  \Coeff{\alpha}{\beta_1} \to
  \Coeff{\alpha}{\beta_2} \)
  is the mapping that transfers
  a Coefficient GF of type \( (\alpha,\beta_1) \)
  to the Coefficient GF of type \( (\alpha,\beta_2) \) with the same coefficients.
  In other words,
  \begin{equation}\label{eq:Coefficient operator}
    \Change{\alpha}{\beta_1}{\beta_2}
    \left(
      \sum_{m=\M}^{\infty}
      \sum_{\ell=0}^\infty
    		a_{m,\ell}^\circ
    		\dfrac{z^m}{\alpha^{\frac{1}{\beta_1}\binom{m}{2}}}
    		w^\ell
    \right)
    :=
    \sum_{m=\M}^{\infty}
    \sum_{\ell=0}^\infty
      a_{m,\ell}^\circ
      \dfrac{z^m}{\alpha^{\frac{1}{\beta_2}\binom{m}{2}}}
      w^\ell \, ,
  \end{equation}
  where \( \M\in\mathbb{Z} \) and the support of the sequence \( (a^\circ_{m,\ell})_{\ell=0}^\infty \) is finite for any \( m \in \mathbb{Z}_{m\geq\M} \).
\end{definition}

\begin{remark}
  We have already mentioned that the sets \( \Coeff{\alpha}{\beta} \) represent different choices of bases in the same linear space \( \mathbb R[w][[z]] \).
  In its turn, the operator \( \Change{\alpha}{\beta_1}{\beta_2} \) represents a change of basis in this space.
  This corresponds to the change of the Coefficient GF associated with a series,
  when we pass from \( A \in \Ring{\alpha}{\beta_1} \)
  to \( B = \Robin{\alpha}^{\beta_1-\beta_2}A \in \Ring{\alpha}{\beta_2} \)
  whose coefficients are defined by
  \( b_n = \alpha^{\binom{n}{2}(\beta_2-\beta_1)} a_n \).
  As well as~\(\Robin{\alpha} \), the operator \( \Change{\alpha}{\beta_1}{\beta_2} \) can be expressed in terms of the exponential Hadamard product:
  \[
    \Change{\alpha}{\beta_1}{\beta_2}
      A^\circ(z, w) 
    = A^\circ(z, w)
      \odot_z
      \G(z, \alpha^{\frac{1}{\beta_1}-\frac{1}{\beta_2}} - 1) \, .
  \]
  For calculations, the following identity, understood in terms of the Hadamard product, could be useful:
  \[
    \Change{\alpha}{\beta_1}{\beta_2} =
    \Robin{\alpha}^{\frac{1}{\beta_2}-\frac{1}{\beta_1}}.
  \]
\end{remark}

\begin{lemma}
\label{lemma: commutative conversion}
  If \( \alpha \in \mathbb R_{>1} \)
  and \( \beta_1, \beta_2 \in \mathbb Z_{>0} \),
  then, for any \( A\in\Ring{\alpha}{\beta_1} \),
  \[
    \big(
	  \Convert{\alpha}{\beta_2}
	  (\Robin{\alpha}^{\beta_1-\beta_2}A)     
	\big)(z, w)
	=
	\Change{\alpha}{\beta_1}{\beta_2}     
	\Big(
		(\Convert{\alpha}{\beta_1}A)(z, w)
    	\Big)
    \, .
  \]
  In other words, the following diagram is commutative.
  \[
  \begin{CD}
	\Ring{\alpha}{\beta_1}
	@>\Convert{\alpha}{\beta_1}>>
	\Coeff{\alpha}{\beta_1} \\
	@V\Robin{\alpha}^{\beta_1-\beta_2}VV
	@VV\Change{\alpha}{\beta_1}{\beta_2}V\\
	\Ring{\alpha}{\beta_2}
	@>\Convert{\alpha}{\beta_2}>>
	\Coeff{\alpha}{\beta_2} \\
  \end{CD}
  \]
\end{lemma}
\begin{proof}
  This follows directly from \cref{definition: asymptotics}, \cref{definition: Robinson operator} and~\cref{definition: Coefficient operator}.
\end{proof}

\subsection{Transfers and marking variables}
\label{sec: transfers-variables}

  The theory developed in the previous sections can be naturally extended for the case of marking variables.
  The aim of this section is to convince the reader that the results we have seen above are still valid for this extension.
  For simplicity, we consider only one marking variable \( u \).
  The reader will see that the statements we formulate are fulfilled in the case of several variables as well.
  
  Given \( \alpha \in \mathbb R_{>1} \) and \( \beta \in \mathbb Z_{>0} \),
  let \( \Ring{\alpha}{\beta}(u) \) be the set of formal power series \( A = A(z;u) \) of form~\eqref{eq:EGF F}
  whose coefficients \( a_n = a_n(u) \) satisfy~\eqref{eq:expansion} with
  \[
    a^\circ_{m,\ell} = a_{m,\ell}^\circ(u) =
    \sum_{k=0}^{\infty}
      a_{m,\ell;k}^\circ u^k
    \, .
  \]
  Here we suppose that the support of the two-dimensional array \( (a^\circ_{m,\ell;k})_{\ell,k=0}^\infty \) is finite for each \( m \in \mathbb Z_{\geq\M} \).
  In particular, \( a_{m,\ell}^\circ(u) \) are polynomials in \( u \), and relation~\eqref{eq:expansion} can be rewritten as
  \[
	a_n(u) \approx \alpha^{\beta \binom{n}{2}}
	\left[
      \sum_{m\geq\M} \alpha^{-mn}
      	\sum_{\ell=0}^{L_m}
	  	n^{\underline{\ell}} \,
          \sum_{k=0}^{K_m}
           a_{m,\ell;k}^\circ u^k
	\right]
  \]
  for some constants \( L_m \) and \( K_m \).
  In this case, similarly to \cref{definition: asymptotics},
  the Coefficient GF of type~\( (\alpha, \beta) \) associated with \( A(z;u) \) is
  the formal power series \( A^\circ = A^\circ (z,w;u) \) defined by~\eqref{eq:CoefficientGF},
  and the set of Coefficient GFs is denoted by \( \Coeff{\alpha}{\beta}(u) \).
  The operator
  \(
    \Convert{\alpha}{\beta}\colon
    \Ring{\alpha}{\beta}(u) \to
    \Coeff{\alpha}{\beta}(u)
  \)
  is defined as before,
  so that \( \Convert{\alpha}{\beta}A = A^\circ \).

\begin{lemma}
  \label{lemma:[u^k] and Q commute}
  If \( \alpha \in \mathbb R_{>1} \), \( \beta \in \mathbb Z_{>0} \) and \( A \in \Ring{\alpha}{\beta}(u) \), then for any \( \kappa \in \mathbb Z_{\geq 0} \)
  \[
    [u^\kappa] \, (\Convert{\alpha}{\beta}A)(z,w;u)
    =
    \big(
      \Convert{\alpha}{\beta} \, [u^\kappa]A
    \big)(z,w)
    \, .
  \]
\end{lemma}
\begin{proof}
  Straightforward calculations show that both expressions are equal to
  \[
    \sum_{m=\M}^{\infty}
    \sum_{\ell=0}^{\infty}
      a_{m,\ell;\kappa}^\circ
        \dfrac{z^m}{\alpha^{\frac{1}{\beta} \binom{m}{2}}}
        w^\ell \, .
  \]
\end{proof}

\begin{notation}
  For a fixed \( \kappa \in \mathbb Z_{\geq 0} \), denote
  \[
    A_\kappa(z) := [u^\kappa] \, A(z;u) \, .
  \]
  It follows from the above that if \( A(z;u) \in \Ring{\alpha}{\beta}(u) \), then \( A_\kappa(z) \in \Ring{\alpha}{\beta} \).
\end{notation}

\begin{proposition}
\label{lemma: operation transfers and marking variables}
  For any fixed \( \alpha\in \mathbb R_{>1} \) and \( \beta \in \mathbb Z_{>0} \), the set \( \Ring{\alpha}{\beta}(u) \) form a ring.
  For each pair \( A,B\in\Ring{\alpha}{\beta}(u) \), the operations of addition and multiplication satisfy
  \begin{equation}\label{eq: sum-transfer and marking variables}
    \big(
      \Convert{\alpha}{\beta} (A + B)
    	\big)(z, w; u)
    =
    (\Convert{\alpha}{\beta} A)(z,w;u)
    +
    (\Convert{\alpha}{\beta} B)(z,w;u)
  \end{equation}
  and
  \begin{equation}\label{eq: prod-transfer and marking variables}
    \big(
	  \Convert{\alpha}{\beta} (A \cdot B)
    \big)(z, w; u)
    =
    A \big(
      \alpha^{\frac{\beta+1}{2}} z^\beta w; u
    \big) \cdot
    (\Convert{\alpha}{\beta} B)(z,w;u)
    +
    B \big(
      \alpha^{\frac{\beta+1}{2}} z^\beta w; u
    \big) \cdot
    (\Convert{\alpha}{\beta} A)(z,w;u)
    \, .
  \end{equation}
\end{proposition}
\begin{proof}
  Relation~\eqref{eq: sum-transfer and marking variables} comes directly from the definitions.
  In order to verify formula~\eqref{eq: prod-transfer and marking variables}, we prove that the corresponding coefficients in \( u \) are the same.
  Indeed, according to \cref{lemma:[u^k] and Q commute} and \cref{lemma: operation transfers}, we have
  \begin{align*}
    [u^\kappa] \, (\Convert{\alpha}{\beta}AB)(z,w;u)
    &=
    \big(
      \Convert{\alpha}{\beta} \, [u^\kappa]AB
    \big)(z,w)
    \\
    &=
    \left(
      \Convert{\alpha}{\beta} 
      \sum\limits_{s=0}^{\kappa}A_s B_{\kappa-s}
    \right)(z,w)
    \\
    &=
    \sum\limits_{s=0}^{\kappa}
    \left(
      A_s \big(
        \alpha^{\frac{\beta+1}{2}} z^\beta w
      \big) \cdot
      (\Convert{\alpha}{\beta} B_{\kappa-s})(z,w)
      +
      B_{\kappa-s} \big(
        \alpha^{\frac{\beta+1}{2}} z^\beta w
      \big) \cdot
      (\Convert{\alpha}{\beta} A_s)(z,w)
    \right)
    \, .
  \end{align*}
  On the other hand, the same tools give us
  \begin{align*}
    [u^\kappa]
    \left(
    A \big(
      \alpha^{\frac{\beta+1}{2}} z^\beta w; u
    \big) \cdot
    (\Convert{\alpha}{\beta} B)(z,w;u)
    \right)
    &=
    \sum\limits_{s=0}^{\kappa}
    A_s \big(
      \alpha^{\frac{\beta+1}{2}} z^\beta w
    \big) \cdot
    [u^{\kappa-s}] \, 
    \big(
      \Convert{\alpha}{\beta} B
    \big)(z,w;u)
    \\
    &=
    \sum\limits_{s=0}^{\kappa}
      A_s \big(
        \alpha^{\frac{\beta+1}{2}} z^\beta w
      \big) \cdot
      (\Convert{\alpha}{\beta} B_{\kappa-s})(z,w)
  \end{align*}
  for the first summand of the right-hand side of~\eqref{eq: prod-transfer and marking variables},
  and a similar relation holds for the second summand.
  Comparing the obtained expressions, we conclude that relation~\eqref{eq: prod-transfer and marking variables} holds.
\end{proof}

\begin{proposition}
\label{lemma: Bender's transfer and marking variables}
  Let \( \alpha \in \mathbb R_{>1} \), \( \beta \in \mathbb Z_{>0} \) and \( A \in \Ring{\alpha}{\beta} \) with \( a_0=0 \).
  If \( F(x;u) \) is a~function analytic in a~neighbourhood of the origin,
  \( G(z;u) = F\big( A(z); u \big) \),
  and \( H(z;u) = \partial_x F(x;u)|_{x=A(z)} \),
  then \( G \in \Ring{\alpha}{\beta}(u) \) and
  \begin{equation}\label{eq: Bender-transfer and marking variables}
    \big(
	  \Convert{\alpha}{\beta} G
    \big)(z, w ;u)
    =
    H \big(
      \alpha^{\frac{\beta+1}{2}} z^\beta w;u
    \big) \cdot
    (\Convert{\alpha}{\beta} A)(z,w)
    \, ,
  \end{equation}
\end{proposition}
\begin{proof}
  Similarly to the proof of \cref{lemma: operation transfers and marking variables}, we verify that the coefficients in \( u \) are the same for both expressions.
  For this purpose, for any \( \kappa \in \mathbb Z_{\geq 0} \), let us introduce
  \[
    F_\kappa(x) := [u^\kappa] \, F(x;u)
    \qquad \mbox{and} \qquad
    H_\kappa(z) := [u^\kappa] \, H(z;u)
    \, .
  \]
  Since
  \(
    [u^\kappa] G (z; u)
     =
    [u^\kappa] F \big( A(z); u \big)
     =
    F_\kappa \big( A(z) \big)
  \)
  and \( H_\kappa(z) = \partial_x F_\kappa(x)|_{x=A(z)} \), due to \cref{lemma:[u^k] and Q commute} and \cref{lemma: Bender's transfer} we have
  \begin{align*}
    [u^\kappa]
    \big(
	  \Convert{\alpha}{\beta} G
    \big)(z, w ;u)
    &=
    \big(
      \Convert{\alpha}{\beta} (F_\kappa\circ A)
    \big)(z,w)
    \\
    &=
    H_\kappa \big(
      \alpha^{\frac{\beta+1}{2}} z^\beta w
    \big) \cdot
    (\Convert{\alpha}{\beta} A)(z,w)
    \\
    &=
    [u^\kappa]
    H \big(
      \alpha^{\frac{\beta+1}{2}} z^\beta w;u
    \big) \cdot
    (\Convert{\alpha}{\beta} A)(z,w)
    \, .
  \end{align*}
\end{proof}

\begin{proposition}
\label{lemma: power transfer and marking variables}
  If \( \alpha \in \mathbb R_{>1} \), \( \beta \in \mathbb Z_{>0} \) and \( A \in \Ring{\alpha}{\beta}(u) \),
  then \( A^m \in \Ring{\alpha}{\beta}(u) \) for each \( m\in\mathbb Z_{\geq0} \) with
  \begin{equation}\label{eq: power-transfer and marking variables}
    \big(
	  \Convert{\alpha}{\beta} A^m
    \big)(z, w; u)
     =
    m \cdot A^{m-1} \big(
      \alpha^{\frac{\beta+1}{2}} z^\beta w; u
    \big)
     \cdot
    (\Convert{\alpha}{\beta} A)(z, w; u)
    \, .
  \end{equation}
  If, additionally, \( [u^0]A(z;u) = 1 \), then \eqref{eq: power-transfer} holds for any \( m\in\mathbb Z \).
\end{proposition}
\begin{proof}
  The essential part of the proof concerns the case when \( m = -1 \),
  since the rest follows from~\cref{lemma: operation transfers and marking variables}.
  Let us check that, for each \( \kappa \in \mathbb Z_{>0} \),
  extracting \( \kappa \)th coefficients from both sides of~\eqref{eq: power-transfer and marking variables} leads to the same result.
  To do this, first notice that \( A_+(z;u) = A(z;u) - 1 \) is divisible by \( u \).
  Hence, \( \big(A_+(z;u)\big)^k \) is divisible by \( u^{\kappa+1} \) for any \( k > \kappa \) and
  \[
    [u^\kappa]\big(A(z;u)\big)^{-1}
     =
    \sum\limits_{s=0}^\kappa (-1)^s \,
      [u^\kappa]\big(A_+(z;u)\big)^s
    \, .
  \]
  As a consequence, \cref{lemma:[u^k] and Q commute} and \cref{lemma: operation transfers} imply that
  \[
    [u^\kappa]\big(
	  \Convert{\alpha}{\beta} A^{-1}
    \big)(z, w; u)
     =
    [u^\kappa]
    \left[
      \sum\limits_{s=1}^\kappa (-1)^s s \cdot
      A_+^{s-1} \big(
        \alpha^{\frac{\beta+1}{2}} z^\beta w; u
      \big)
       \cdot
      (\Convert{\alpha}{\beta} A_+)(z, w; u)
    \right]
    \, .
  \]
  On the other hand,
  \begin{multline*}
    -[u^\kappa]
    \Bigg[
      A^{-2} \big(
        \alpha^{\frac{\beta+1}{2}} z^\beta w; u
      \big) \cdot
      (\Convert{\alpha}{\beta} A)(z, w; u)
    \Bigg]
     = \\
    -[u^\kappa]
    \left[
      \left(
        \sum\limits_{s=0}^\infty (-1)^s 
        A_+^s \big(
          \alpha^{\frac{\beta+1}{2}} z^\beta w; u
        \big)          
      \right)^2
      (\Convert{\alpha}{\beta} A)(z, w; u)
    \right]
    \, ,
  \end{multline*}
  and, taking into account that
  \(
    (\Convert{\alpha}{\beta} A)(z, w; u)
     =
    (\Convert{\alpha}{\beta} A_+)(z, w; u)
  \)
  is divisible by \( u \), this gives us the same expression. 
\end{proof}

  Similarly to what has been done in \cref{sec: transfers-relations}, we define operators
  \(
    \Robin{\alpha} \colon
    \Ring{\alpha}{\beta}(u) \to
    \Ring{\alpha}{\beta-1}(u)
  \)
  and
  \(
    \Change{\alpha}{\beta_1}{\beta_2} \colon
    \Coeff{\alpha}{\beta_1}(u) \to
    \Coeff{\alpha}{\beta_2}(u)
    \, ,
  \)
  so that they satisfy relations~\eqref{eq:Robinson operator} and~\eqref{eq:Coefficient operator}, respectively.
  In this case, the following generalization of \cref{lemma: commutative conversion} holds.

\begin{lemma}
\label{lemma: generalized commutative conversion}
  If \( \alpha \in \mathbb R_{>1} \)
  and \( \beta_1, \beta_2 \in \mathbb Z_{>0} \),
  then, for any \( A\in\Ring{\alpha}{\beta_1}(u) \),
  \[
    \big(
	  \Convert{\alpha}{\beta_2}
	  (\Robin{\alpha}^{\beta_1-\beta_2}A)     
	\big)(z, w; u)
	=
	\big(
		\Change{\alpha}{\beta_1}{\beta_2}     
		(\Convert{\alpha}{\beta_1}A)
    	\big)(z, w; u)
    \, .
  \]
\end{lemma}
\begin{proof}
  This follows directly from definitions.
\end{proof}

\section{Digraphs}
\label{section:digraphs}

\subsection{Asymptotics for graphs and tournaments}
\label{sec:graphs-tournaments}

  In this section, as a first example of the application of our method, we discuss asymptotics of undirected graphs and tournaments.
  Both asymptotics are known:
  they were established for the first time in 1970 by Wright~\cite{wright1970asymptotic, wright1970the}
  and were combinatorially interpreted in 2021 by Monteil and Nurligareev~\cite{monteil2021asymptotics}.
  Here we, first, revisit these results in terms of Coefficient GFs by applying asymptotic transfer,
  and second, employ the obtained results to get more general asymptotics.
  Namely, we obtain the Coefficient GFs of graphs and tournaments with a marking variables for the number of connected graphs and irreducible tournaments, respectively.
  The latter result appears to be novel in its general form, although essential information related to the asymptotics under consideration (including the interpretation of the coefficients) can be obtained in another way; see~\cite{monteilSEQ} and~\cite{monteilANTISEQ}.

\begin{lemma}
\label{example: graphs and tournaments}
  The Exponential GFs \( \G(z) \), \( \T(z) \), and \( \D(z) \) of undirected graphs, tournaments, and directed graphs, respectively,
  belong to the rings \( \Ring{2}{1} \), \( \Ring{2}{1} \), and \( \Ring{2}{2} \).
  Their Coefficient GFs of type \( (2,1) \), \( (2,1) \), and \( (2,2) \), respectively, satisfy
  \[
	(\Convert{2}{1}\, \G)(z,w) =
	(\Convert{2}{1}\, \T)(z,w) = 1
	\qquad \mbox{and} \qquad
	(\Convert{2}{2}\, \D)(z,w) = 1.
  \]
\end{lemma}
\begin{proof}
  This follows immediately from \cref{definition: asymptotics}, since
  \[
	\G(z) = \T(z) = \sum_{n=0}^{\infty} 2^{\binom{n}{2}} \dfrac{z^n}{n!}
	\qquad \mbox{and} \qquad
	\D(z) = \sum_{n=0}^{\infty} 2^{2\binom{n}{2}} \dfrac{z^n}{n!}
	\, .
  \]
\end{proof}

\begin{theorem}
\label{theorem: connected graphs asymptotics}
  The Exponential GF \( \CG(z) \) of connected graphs belongs to the ring \( \Ring{2}{1} \) and its Coefficient GF of type \( (2,1) \) satisfies
  \begin{equation}\label{eq: connected graphs asymptotics}
	(\Convert{2}{1}\, \CG) = 1 - \I(2zw).
  \end{equation}
\end{theorem}
\begin{proof}
  As we have seen in \cref{example: graphs and tournaments},
  the Exponential GF \( \G(z) \) of graphs belongs to \( \Ring{2}{1} \) with
  \[
	(\Convert{2}{1}\, \G)(z,w) = 1.
  \]
  Since the Exponential GF of connected graphs satisfy the exponential formula
  \[
	\CG(z) = \log\big(\G(z)\big),
  \]
  we can apply \cref{lemma: Bender's transfer} to \( A(z) = \G(z) - 1 \) and \( F(x) = \log(1 + x)\) with \( \alpha = 2 \) and \( \beta = 1 \).
  Taking into account formulae~\eqref{eq:T=1/(1-IT)} and~\eqref{eq:G=T}, we have 
  \[
	H(z) = \dfrac{1}{\G(z)} = 1 - \I(z),
  \]
  which implies target relation~\eqref{eq: connected graphs asymptotics}.
\end{proof}

\begin{corollary}
\label{cor: graphs with m components asymptotics}
  The bivariate Exponential GF \( G(z; t) \) of graphs
  with a variable \( t \) that marks the number of connected components
  belongs to the ring \( \Ring{2}{1}(t) \) and
  its Coefficient GF of type \( (2,1) \) satisfies
  \[
	(\Convert{2}{1}\, \G)(z,w;t)
	 =
	t \cdot \G(2zw;t) \cdot \big( 1 - \I(2zw) \big)
	\, .
  \]
  In particular, for any \( m \in \mathbb Z_{\geq0} \), the asymptotics of graphs with \( (m+1) \) connected components is given by
  \begin{equation}\label{eq: graphs with fixed number of connected components}
	[t^{m+1}](\Convert{2}{1}\, \G)(z,w;t)
	 =
    \dfrac{1}{m!} \CG^m(2zw)
	 \cdot
	\big( 1 - \I(2zw) \big)
	\, .  
  \end{equation}
\end{corollary}
\begin{proof}
  Taking into account relation~\eqref{eq:G(t)=exp(tCG)}, it is sufficient to apply \cref{lemma: Bender's transfer and marking variables} to \( A(z) = \CG(z) \) and \( F(x;t) =e^{tx} \) with \( \alpha = 2 \) and \( \beta = 1 \).
\end{proof}

\begin{remark}
\label{remark: graphs references}
  Relation~\eqref{eq: connected graphs asymptotics} corresponds to the asymptotic expansion
  \[
	\cg_n \approx 2^{\binom{n}{2}}
	\left(1 -
	  \sum\limits_{k\geq1}
	  \ir_k\binom{n}{k}
		\dfrac{2^{\binom{k+1}{2}}}{2^{kn}}
	  \right)
  \]
  proved in~\cite{monteil2021asymptotics}, where \( \cg_n \) and \( \ir_n \) are the numbers of connected graphs and irreducible tournaments of size $n$, respectively.
  Formula~\eqref{eq: graphs with fixed number of connected components} reflects the asymptotics from~\cite[Theorem 7.3.1]{nurligareev2022irreducibility},
  see also~\cite{monteilANTISEQ}.
\end{remark}

\begin{theorem}
\label{theorem: irreducible tournaments asymptotics}
  The Exponential GF \( \I(z) \) of irreducible tournaments belongs to the ring \( \Ring{2}{1} \) and its Coefficient GF of type \( (2,1) \) satisfies
  \begin{equation}\label{eq: irreducible tournaments asymptotics}
	(\Convert{2}{1}\, \I)(z,w) = \big(1 - \I(2zw)\big)^2.
  \end{equation}
\end{theorem}
\begin{proof}
  Due to \cref{example: graphs and tournaments},
  the Exponential GF of tournaments belongs to \( \Ring{2}{1} \) and
  \[
	(\Convert{2}{1}\, \T)(z,w) = 1.
  \]
  According to~\eqref{eq:T=1/(1-IT)}, the Exponential GF of irreducible tournaments satisfy
  \[
	\I(z) = 1 -	\dfrac{1}{\T(z)}.
  \]
  Hence, to get relation~\eqref{eq: irreducible tournaments asymptotics},
  it is sufficient to apply \cref{lemma: Bender's transfer} to \( A(z) = \T(z) - 1 \) and \( F(x) = 1 - (1+x)^{-1}\) with \( \alpha = 2 \) and \( \beta = 1 \).
\end{proof}

\begin{corollary}
\label{cor: tournament with m irreducible parts}
  The bivariate Exponential GF \( T(z; t) \) of tournaments
  with a variable \( t \) that marks the number of irreducible parts
  belongs to the ring \( \Ring{2}{1}(t) \) and
  its Coefficient GF of type \( (2,1) \) satisfies
  \[
	(\Convert{2}{1}\, \T)(z,w;t)
	 =
	t \cdot \left(
	  \T(2zw;t) \cdot \big( 1 - \I(2zw) \big)
	\right)^2
	\, .
  \]
  In particular, for any \( m \in \mathbb Z_{\geq0} \), the asymptotics of tournaments with \( (m+1) \) irreducible parts is given by
  \begin{equation}\label{eq: tournaments with fixed number of irreducible parts}
	[t^{m+1}](\Convert{2}{1}\, \T)(z,w;t)
	 =
    (m+1) \cdot \I^m(2zw)
	 \cdot
	\big( 1 - \I(2zw) \big)^2
	\, .
  \end{equation}
\end{corollary}
\begin{proof}
  Due to the second of relations~\eqref{eq:bivariate:SSD=exp(SCD)}, it is sufficient to apply \cref{lemma: Bender's transfer and marking variables} to \( A(z) = \I(z) \) and \( F(x;t) =(1-tx)^{-1} \) with \( \alpha = 2 \) and \( \beta = 1 \).
\end{proof}

\begin{remark}
\label{remark: tournaments references}
  Formula~\eqref{eq: irreducible tournaments asymptotics} reflects the asymptotic expansion
  \[
	\ir_n \approx 2^{\binom{n}{2}}
	\left(1 -
	  \sum\limits_{k\geq1}
	  \big(2\ir_k-\ir_k^{(2)}\big)
	    \binom{n}{k}
		  \dfrac{2^{\binom{k+1}{2}}}{2^{kn}}
	\right)
  \]
  proved in~\cite{monteil2021asymptotics},
  where \( \ir_n \) and \( \ir_n^{(2)} \) are the number of irreducible tournaments and tournaments consisting of two irreducible parts (both of size $n$).
  Formula~\eqref{eq: tournaments with fixed number of irreducible parts} is consistent with the asymptotics established in~\cite{monteilSEQ} and~\cite[Theorem 5.3.1]{nurligareev2022irreducibility},
  see also~\cite{monteilANTISEQ}.
\end{remark}

\begin{remark}
\label{remark: asymptotics of G(n p)}
  Taking into account \cref{remark:introducing Erdos-Renyi}, we can establish the asymptotic behavior of graphs within the Erd\H{o}s-R\'enyi model.
  Indeed, denoting \( \alpha = (1 - p)^{-1} \), we get the Exponential GF of graphs expressed as
  \[
    \G(z) = \sum\limits_{n=0}^\infty
     \alpha^{\binom{n}{2}} \dfrac{z^n}{n!}
   \, .
  \]
  Relations \eqref{eq:G=exp(CG)} and \eqref{eq:G(t)=exp(tCG)} remain valid, which lead us to
  \[
	(\Convert{\alpha}{1}\CG)(z,w)
	 = 
	\dfrac{1}{\G(\alpha zw)}
	 = 
	e^{-\CG(\alpha zw)}
  \]
  and
  \[
	(\Convert{\alpha}{1}\G)(z,w;t)
	 =
	t \cdot \dfrac{\G(\alpha zw;t)}{\G(\alpha zw)}
	 =
	t \cdot e^{(t-1) \cdot \CG(\alpha zw)}
	\, .
  \]
  However, there is no combinatorial interpretation in terms of irreducible tournaments anymore, see \cref{remark: meaning of G^(-1)}.
\end{remark}

\subsection{Asymptotics for strongly connected digraphs}

  This section is devoted to the asymptotic behavior of strongly connected digraphs.
  A classical enumeration result related to this combinatorial class was first obtained by Wright~\cite{wright1971number} in~1971 and reproved by Bender~\cite{bender1975asymptotic} several years later
  (see also the papers of Liskovets~\cite{liskovets1969, liskovets1970number}).
  Here, we establish the corresponding Coefficient GF and the exact form of its coefficients in terms of semi-strong digraphs and tournaments.
  This allows us to rewrite the previously known asymptotics in a compact form
  and provide a combinatorial meaning to the involved coefficients.  

\begin{theorem}\label{theorem:asymptotics:scc}
  The Exponential GF \( \SCC(z) \) of strongly connected digraphs belongs to the ring~\( \Ring{2}{2} \)
  and its Coefficient GF of type \( (2,2) \) satisfies
  \begin{equation}\label{equation:scc Graphic GF}
    (\Convert{2}{2}\, \SCC)(z,w) =
    \E(2^{3/2}z^2w) \cdot
      \Change{2}{1}{2} \big(1 - \I(2zw)\big)^2
    \, .
  \end{equation}
\end{theorem}
\begin{proof}
  Recall that, according to \cref{corollary: exponential gf of ssd and scc}, the Exponential GF of strongly connected graphs satisfy
  \[
    \SCC(z) = \log \dfrac{1}{1 - \I(z) \odot \G(z)}
    \, .
  \]
  As we have seen in \cref{theorem: irreducible tournaments asymptotics}, the Exponential GF of irreducible tournaments belongs to \( \Ring{2}{1} \) with
  \[
	(\Convert{2}{1}\, \I)(z,w) = \big(1 - \I(2zw)\big)^2.
  \]
  Hence, \( \big( \Robin{2}^{-1} \I \big)(z) = \I(z) \odot \G(z) \) belongs to \( \Ring{2}{2} \), and to obtain
  \(
    (\Convert{2}{2}\, \SCC)(z,w)
  \)
  we can apply \cref{lemma: Bender's transfer} to
  \[
    A(z) = \I(z) \odot \G(z)
    \qquad \text{and} \qquad
	F(x) = -\log(1-x)
  \]
  with \( \alpha = \beta = 2 \). Taking into account relations~\eqref{eq: semi-strong as sequences}, in the case in hand we have 
  \[
    \left.\dfrac{\partial F}{\partial x}
    \right|_{x=A(z)} =
    \dfrac{1}{1 - \I(z) \odot \G(z)} = \E(z).
  \]
  To finish the proof, we use \cref{lemma: commutative conversion} and relation~\eqref{eq: irreducible tournaments asymptotics}:
  \[
    \big(
      \Convert{2}{2}(\Robin{2}^{-1}\I)
    \big)(z,w)
    =
    \big(
      \Change{2}{1}{2} (\Convert{2}{1} \, \I)
    \big)(z,w)
    =
    \Change{2}{1}{2} \big(1 - \I(2zw)\big)^2
    \, .
  \]
\end{proof}

\begin{corollary}\label{corollary:scc probability}
  The probability \( p_n \) that a uniform random digraph with \( n \) vertices is strongly connected satisfies
  \begin{equation}\label{equation:scc probability}
    p_n \approx
    \sum_{m\geq0} \dfrac{1}{2^{nm}}
      \sum_{\ell=\lceil m/2 \rceil}^{m}
        n^{\underline{\ell}}\,
        \scc^\circ_{m,\ell}
    \, ,
  \end{equation}
  where
  \[
    \scc^\circ_{m,\ell}
    =
    2^{m(m+1)/2 + \ell(\ell-m)}
    \dfrac{\ssd_{m-\ell}}{(m-\ell)!}
    \dfrac{\mathbf 1_{m=2\ell} - 2\ir_{2\ell-m} + \ir^{(2)}_{2\ell-m}}{(2\ell-m)!}
    \, ,
  \]
  and \( \ssd_k \), \( \ir_k \) and \( \ir_k^{(2)} \) denote
  the numbers of semi-strong digraphs, irreducible tournaments and tournaments with two irreducible components, respectively
  (all of them are of size \( k \)).
\end{corollary}
\begin{proof}
  By definition, we have the relation
  \(
    p_n = \scc_n/2^{2\binom{n}{2}}.
  \)
  According to \cref{theorem:asymptotics:scc}, the Exponential GF \( \SCC(z) \) belongs to \( \Ring{2}{2} \).
  Hence, it follows from~\eqref{eq:expansion} and~\eqref{equation:scc Graphic GF} that
  \[
    p_n \approx
    \sum_{m \geq 0} \dfrac{1}{2^{nm}}
      \sum_{\ell=0}^{\infty}
        n^{\underline{\ell}}\,
        \scc^\circ_{m,\ell}
    \, .
  \]
  To establish the limits of summation, let us denote, for any \( n,k\in\mathbb{Z}_{\geq0} \),
  \begin{equation}\label{equation:b_n}
    b_{n}
    :=
    n! [z^n] \big(1 - \I(z)\big)^2
    =
    \mathbf 1_{n=0} - 2\ir_n + \ir_n^{(2)}
  \end{equation}
  and
  \[
    b_{n,k}
    :=
    n! [z^n w^k] \big(1 - \I(zw)\big)^2
    =
    b_n \cdot \mathbf 1_{n=k}
    \, .
  \]
  Since the coefficients \( b_{n,k} \) are non-zero when \( n = k \) only,
  by direct calculations we obtain
  \begin{align*}
    \scc_{m,\ell}^\circ &=
    2^{\frac{1}{2} \binom{m}{2}}
      [z^m w^\ell] (\Convert{2}{2}\, \SCC)(z,w)
    \\& =
    \sum_{k=0}^{\infty}
      2^{3k/2} \dfrac{\ssd_k}{k!}
      \cdot
      2^{\frac{1}{2} \binom{m-2k}{2}}
      [z^{m-2k} w^{\ell-k}]
        \big(1 - \I(2zw)\big)^2
    \\&=
    \sum_{k=0}^{\infty}
      2^{m(m+1)/2 + k(k-m)}
      \dfrac{\ssd_k}{k!}
      \dfrac{b_{m-2k,\ell-k}}{(m-2k)!}
    \\&=
      2^{m(m+1)/2 + \ell(\ell-m)}
      \dfrac{\ssd_{m-\ell}}{(m-\ell)!}
      \dfrac{b_{2\ell-m}}{(2\ell-m)!}
    \, ,
  \end{align*}
  which is non-zero for \( \ell \leq m \leq 2\ell \) only. 
\end{proof}

\begin{remark}
  Relation~\eqref{equation:scc Graphic GF} corresponds to the asymptotic expansion of the form
  \[
    \scc_n \approx 2^{n^2-n}
      \sum_{m\geq0} \dfrac{w_m(n)}{2^{nm}}
  \]
  investigated by Wright~\cite{wright1971number} who established a recursive method of computing the polynomials \( w_m(n) \).
  The latter asymptotics was also studied by Bender~\cite{bender1975asymptotic} who proposed a direct way of computing these polynomials based on \cref{theorem: Bender's}.
  For numerical values of \( \scc_{m,\ell}^\circ \) and \( w_m(n) \), as well as for more details, see \cref{Appendix: strongly connected digraphs}.
\end{remark}

\subsection{Fixed number of strongly connected components}

  This section is devoted to the asymptotics of digraphs with a marking variable for the number of strongly connected components.
  We start with establishing the Coefficient GF of semi-strong digraphs.
  Next, we proceed to the Coefficient GF of all digraphs.
  Finally, we provide the leading term of the probability that a random digraph has a fixed number of strongly connected components and
  indicate the combinatorial meaning of this term, which involves directed acyclic graphs.
  All the results discussed here appear to be novel.

\begin{theorem}
\label{theorem:cgf:ssd}
  The bivariate Exponential GF \( \E(z; t) \) of semi-strong digraphs with the marking variable \( t \) for the number of strongly connected components
  belongs to the ring \( \Ring{2}{2}(t) \) and
  the corresponding Coefficient GF of type \( (2,2) \) satisfies
  \begin{equation}\label{eq:cgf:ssd}
    (\Convert{2}{2}\, \E)(z, w; t) = t \cdot
    \E(2^{3/2}z^2w; 1+t) \cdot
      \Change{2}{1}{2} \big(1 - \I(2zw)\big)^2
    \, .  
  \end{equation}
  In particular, for any \( m \in \mathbb Z_{\geq0} \), the asymptotics of semi-strong digraphs with \( (m+1) \) strongly connected components is given by
  \begin{equation}\label{eq:[t^k]cgf:ssd}
	[t^{m+1}](\Convert{2}{2}\, \E)(z,w;t)
	 =
    \dfrac{1}{m!} \SCC^m(2^{3/2}z^2w)
	 \cdot
	(\Convert{2}{2}\, \SCC)(z,w)
	\, .
  \end{equation}
\end{theorem}
\begin{proof}
  Taking into account relation~\eqref{eq:bivariate:SSD=exp(SCD)},
  it is sufficient to apply \cref{lemma: Bender's transfer and marking variables} to \( A(z) = \SCC(z) \) and \( F(x;t) = e^{tx} \) with \( \alpha = \beta = 2 \).
  To complete the proof, we use \cref{theorem:asymptotics:scc}.
\end{proof}

\begin{theorem}
\label{theorem:cgf:several:scc}
  The Graphic GF \( \widehat{\D}(z; t) \) of digraphs with the marking variable \( t \) for the number of strongly connected components
  belongs to the ring \( \Ring{2}{1}(t) \) and
  the corresponding Coefficient GF of type \( (2,1) \) is given by
  \begin{equation}
  \label{eq:cgf:several:scc}
    (\Convert{2}{1}\, \widehat{\D})(z, w; t)
      =
    \left(
      \widehat{\D}(2zw; t)
    \right)^2
      \cdot
    \Change{2}{2}{1}
    \Big(
      (\Convert{2}{2}\, \E)(z, w; -t)
    \Big)
    \, .  
  \end{equation}
\end{theorem}
\begin{proof}
  It is sufficient to apply \cref{lemma: power transfer and marking variables}, \cref{lemma: generalized commutative conversion} and \cref{lemma: Bender's transfer and marking variables}
  to the Graphic GF \( \widehat{\D}(z; t) \) of digraphs written in the form
  \[
    \widehat{\D}(z; t) =
    \dfrac
      {1}
      {\Robin{2} \big(e^{-t\cdot\SCC(z)}\big)}
    \, .
  \]
\end{proof}

\begin{corollary}
\label{corollary:dominant:m-fixed}
  The probability \( p_{n,m+1} \) that a uniform random digraph on \( n \) vertices has \( (m+1) \) strongly connected components
  satisfies the following asymptotic behavior, as \( n \to \infty \):
  \begin{equation}\label{eq:dominant:m-fixed}
    p_{n,m+1}
    \sim
    \binom{n}{m}
    \dfrac{2^m \dgg_m^{(\vec{2})}}{2^{mn}}
    \, ,
  \end{equation}
  where \( \dgg_m^{(\vec{2})} \) are the coefficients of the Graphic GF \( \widehat{\DAG}^2(z) \).
\end{corollary}
\begin{proof}
  In order to obtain the dominant term of the probability \( p_{n,m+1} \),
  we need to trace the term with the smallest power of \( z \) in
  \(
    [t^{m+1}] (\Convert{2}{1}\, \widehat{\D})(z, w; t)
    \, .
  \)
  In order to do that, rewrite~\eqref{eq:cgf:several:scc} as
  \begin{equation}
  \label{eq:cgf:several:scc-2}
    (\Convert{2}{1}\, \widehat{\D})(z, w; t)
     =
    t
     \cdot
    \dfrac
    {
      \Change{2}{2}{1}
      \left(
        e^{(1-t) \cdot \SCC(2^{3/2}z^2 w)}
          \cdot
        \Change{2}{1}{2}
        \big( 1 - \I(2zw) \big)^2
      \right)
    }
    {
      \left(
        \Robin{2}
        \big( e^{-t \cdot \SCC(2zw)} \big)
      \right)^2
    }
    \, .  
 \end{equation}
  Note that the (only) smallest strongly connected digraph has one vertex,
  hence, \( \SCC(z) \) starts with~\( z \).
  As a consequence, for every \( t^m \),
  the smallest exponent in \( z \) in the numerator of~\eqref{eq:cgf:several:scc-2} is equal to \( 2m \).
  At the same time, the smallest exponent in \( z \) in the denominator is \( m \),
  which is smaller for every \( m > 0 \).
  Since the numerator starts with \( 1 \), the whole fraction can be simplified in the following way:
  \begin{align*}
    [t^{m+1}] (\Convert{2}{1}\, \widehat{\D})(z, w; t)
    &=
    [t^m]
    \dfrac{1}{\big( e^{-2tzw} \odot_z \Set(z) \big)^2}
    = 2^m z^{m} w^{m}
    [t^m] \left(\dfrac{1}{\Set(-t)} \right)^2
    \, .
  \end{align*}
  \cref{proposition: graphic gf of dag} implies that
  \[
    \left( \dfrac{1}{\Set(-z)} \right)^2
    =
    \left( \dfrac{1}{\Robin{2} (e^{-z})} \right)^2
    =
    \widehat{\DAG}^2(z)
    \, .
  \]
  Thus, the expansion coefficient at \( 2^{-mn} n^{\underline{m}} \) is equal to
  \[
    2^{\binom{m}{2}} [z^{m} w^{m}]
    [t^{m+1}] (\Convert{2}{1}\, \widehat{\D})(z, w; t)
    =
    2^{m} \dfrac{\dgg_m^{(\vec{2})}}{m!} \, .
  \]
  The proof is completed by noting that \( \binom{n}{m} = n^{\underline{m}}/m! \).
\end{proof}

\begin{remark}
\label{remark:dominant:m-fixed}
  There is also a direct combinatorial way to establish asymptotics~\eqref{eq:dominant:m-fixed},
  which is based on the structural analysis of involved digraphs.
  This method works for more general scenarios as well,
  even when generating functions are not available,
  but obtaining secondary terms in this case may become too tedious.
  The key idea is that the main contribution into the asymptotics of digraphs with \( (m+1) \) components
  is given by digraphs whose \( m \) components contain one vertex each,
  and the remaining component contains \( (n-m) \) vertices.
  We will call the one-vertex components the \emph{small} ones,
  and the component with \( (n-m) \) vertices the \emph{large} one.
  As \( n \to \infty \), with high probability each of the small components is connected to the large component by at least one edge.
  Note that, for any small component, the connections can only be in one direction,
  otherwise the small component would be merged into the large one.

  This observation allows us to repartition the \( m \) one-vertex components into the group of \( k \) components having the edges \emph{towards} the large component,
  and the group of \( (m-k) \) components that have the edges \emph{from} the large component (see \cref{fig:two:dags}).
  \begin{figure}[hbt!]
    \begin{center}
      \begin{tikzpicture}[>=stealth',thick, scale = 0.9]
%
%
%
\draw
node[arnBleuGrande](c) at (3,0) {$\circlearrowright$}
node[arnBleuPetit](l1) at (1,1) {}
node[arnBleuPetit](l2) at (0,-1) {}
node[arnBleuPetit](r1) at (5,1) {}
node[arnBleuPetit](r2) at (7, 0) {}
node[arnBleuPetit](r3) at (6,-1) {}
;

%
%
%
\node[rectangle,dashed,draw, fit=(l1)(l2),
      rounded corners = 3mm,inner sep = 7pt, bviolet, very thick] {};
\node[rectangle,dashed,draw, fit=(r1)(r2)(r3),
      rounded corners = 3mm,inner sep = 7pt, bviolet, very thick] {};

%
%
%
\draw node at (0.5,-2) {$k$};
\draw node at (6,-2) {$m-k$};
\draw node at (3,-1) {$n-m$};

%
%
%
\path (l1) edge [blackred,thick, bend right=-5,
    decoration={markings,mark=at position 0.7 with
    {\arrow[ultra thick,blackred, rotate=0]{>}}}, postaction={decorate}
    ] node {} (c);
\path (l2) edge [blackred,thick, bend right=-5,
    decoration={markings,mark=at position 0.7 with
    {\arrow[ultra thick,blackred, rotate=0]{>}}}, postaction={decorate}
    ] node {} (c);

\path (c) edge [blackred,thick, bend right=-5,
    decoration={markings,mark=at position 0.7 with
    {\arrow[ultra thick,blackred, rotate=0]{>}}}, postaction={decorate}
    ] node {} (r1);
\path (c) edge [blackred,thick, bend right=-5,
    decoration={markings,mark=at position 0.7 with
    {\arrow[ultra thick,blackred, rotate=0]{>}}}, postaction={decorate}
    ] node {} (r2);
\path (c) edge [blackred,thick, bend right=-5,
    decoration={markings,mark=at position 0.7 with
    {\arrow[ultra thick,blackred, rotate=0]{>}}}, postaction={decorate}
    ] node {} (r3);

%
%
%
\path (l1) edge [blackred,thick, bend right=-5,
    decoration={markings,mark=at position 0.7 with
    {\arrow[ultra thick,blackred, rotate=0]{>}}}, postaction={decorate}
    ] node {} (l2);
\path (r1) edge [blackred,thick, bend right=-5,
    decoration={markings,mark=at position 0.7 with
    {\arrow[ultra thick,blackred, rotate=0]{>}}}, postaction={decorate}
    ] node {} (r2);
\path (l1) edge [blackred,thick, bend left=60,
    decoration={markings,mark=at position 0.7 with
    {\arrow[ultra thick,blackred, rotate=0]{>}}}, postaction={decorate}
    ] node {} (r2);
\end{tikzpicture}
    \end{center}
    \caption{\label{fig:two:dags}}
  \end{figure}
  The vertices within each of the above two groups form a directed acyclic graph structure,
  with \( k \) and \( (m-k) \) vertices, respectively.
  Furthermore, there can be additional edges from the first group to the second one.
  Hence, the number of arrangements of \( m \) vertices beyond the large component is
  \[
    \sum\limits_{k=0}^m
    \binom{m}{k} 2^{k(m-k)}
    \dgg_k\dgg_{m-k}
     =
    \dgg_n^{(\vec{2})}
  \]
  and the total number of digraphs on \( n \) vertices defined in such a way is asymptotically equal to
  \[
    \binom{n}{m} 2^{(n-m)(n-m-1)} 2^{m(n-m)}
    \dgg_m^{(\vec{2})}
    \, .
  \]
  Thus, the dominant term of the probability that a random digraph contains \( (m+1) \) strongly connected components is
  \[
    \dfrac{1}{2^{n(n-1)}}
    \binom{n}{m} 2^{(n-m)(n-m-1)} 2^{m(n-m)}
    \dgg_m^{(\vec{2})}
    =
    \binom{n}{m}
    \dfrac{ 2^{m} \dgg_m^{(\vec{2})} } {2^{mn}}
    \, .
  \]
\end{remark}

\subsection{Different kinds of strongly connected components}

  In this section, we provide complete asymptotic expansions for two more multivariate versions of digraphs.
  The advantage of the corresponding Coefficient GFs lies in possessing information related to the asymptotic behavior of digraphs with an arbitrary number of components of arbitrary types.
  At the same time, the presented new results are a straightforward application of the asymptotic transfer to the relations discussed in \cref{subsection: enumeration of digraphs with marking variables}. 
  That is why we omit tedious details of their proofs.

\begin{theorem}
\label{theorem:cgf:trivariate:digraphs}
  The Graphic GF \( \widehat{\D}(z; s, t) \) of digraphs with marking variables \( s \) and \( t \) for the numbers of source-like and all strongly connected components, respectively,
  belongs to the ring \( \Ring{2}{1}(s,t) \)
  and the corresponding Coefficient GF of type \( (2,1) \) is given by
  \begin{multline*}
    (\Convert{2}{1}\, \widehat{\D})(z, w; s, t)
     = \\
    \widehat{\D}(2zw; t)
    \left[
      \Change{2}{2}{1}
       \Big(
         (\Convert{2}{2}\, \E) \big( z,w; (s-1)t \big)
       \Big)
       +
      \widehat{\D}(2zw; s, t)
       \cdot
      \Change{2}{2}{1}
       \Big( (\Convert{2}{2}\, \E)(z,w; -t) \Big)
    \right]
    \, .
  \end{multline*}
\end{theorem}
\begin{proof}
  It is sufficient to apply \cref{lemma: operation transfers and marking variables} to relation~\eqref{eq: GGF of digraphs marking components and source-like components} in the form
  \[
    \widehat{\D}(z; s, t)
     =
    \Robin{2} \Big( \E \big( z; (s-1)t \big) \Big)
     \cdot
    \widehat{\D}(z; t)
    \, .
  \]
  The first summand of the result comes directly from the chain rule.
  To get the second summand, we additionally apply \cref{theorem:cgf:several:scc} and use relation~\eqref{eq: GGF of digraphs marking components and source-like components} again, but in the opposite direction.
\end{proof}

\begin{theorem}
\label{theorem:cgf:multivariate:digraphs}
  The Exponential GF \( \D(z; u, v, y, t) \) of digraphs
  with marking variables  \( u, v, y \) and \( t \) for the numbers of purely source-like, purely sink-like, isolated and all strongly connected components, respectively,
  belongs to the ring \( \Ring{2}{2}(u,v,y,t) \)
  and the corresponding Coefficient GF of type \( (2,2) \) is given by
  \begin{equation}\label{eq:cgf:multivariate:digraphs}
    (\Convert{2}{2}\, \D)(z, w; u, v, y, t) =
    \D^\circ_1
     +
    \D^\circ_{20} \cdot
    \Change{2}{1}{2}
    \Big(
      \D^\circ_{21} + \D^\circ_{22} + \D^\circ_{23}
    \Big)
    \, ,
  \end{equation}
  where
  \begin{align*}
    \D_1^\circ(z,w; u,v,y,t) &=
      (y-u-v+1)t
       \cdot
      \D(2^{3/2} z^2w; u,v,y,t)
       \cdot
      (\Convert{2}{2}\, \SCC)(z,w)
    \, ; \\
    \D_{20}^\circ(z,w; u,v,y,t) &=
      \E\big( 2^{3/2}z^2w; (y-u-v+1)t \big)
    \, ;  \\
    \D_{21}^\circ(z,w; u,v,y,t) &=
      \widehat{\D}(2zw; u,t)
       \cdot
      \Change{2}{2}{1}
       \Big(
         (\Convert{2}{2}\, \E) \big( z,w; (v-1)t \big) 
       \Big)
    \, ; \\
    \D_{22}^\circ(z,w; u,v,y,t) &=
      \widehat{\D}(2zw; v,t)
       \cdot
      \Change{2}{2}{1}
       \Big(
         (\Convert{2}{2}\, \E) \big( z,w; (u-1)t \big) 
       \Big)
    \, ; \\
    \D_{23}^\circ(z,w; u,v,y,t) &=
      \widehat{\D}(2zw; u,t)
       \cdot
      \widehat{\D}(2zw; v,t)
       \cdot
      \Change{2}{2}{1}
       \Big( (\Convert{2}{2}\, \E)(z,w; -t) \Big)
    \, .
  \end{align*}
\end{theorem}
\begin{proof}
  The proof is a straightforward application of \cref{lemma: operation transfers and marking variables} and \cref{lemma: generalized commutative conversion} to relation~\eqref{eq: complete multivariate Graphic GF for digraphs}
  with further simplifications done by \cref{corollary: GGF of digraphs marking components}, \cref{theorem:cgf:ssd} and \cref{theorem:cgf:several:scc}.
\end{proof}

\begin{remark}
  Similarly to \cref{corollary:dominant:m-fixed}, we can prove that
  the probability that a uniform random digraph on \( n \) vertices has \( (m+1) \) strongly connected components,
  from which \( i > 1 \) are purely source-like, \( j > 1 \) are purely sink-like and \( \ell \geq 0 \) are isolated, is asymptotically
  \[
    2^{-n(m+\ell)}
    \binom{n}{m} \binom{m}{\ell}
    2^{m(\ell+1)}
    \sum_k \binom{m-\ell}{k}
    2^{k(m-\ell-k)}
    \dgg_{k,i}
    \dgg_{m-\ell-k,j}
    \, ,
  \]
  where \( \dgg_{m,j} \) denotes the number of directed acyclic graphs with \( m \) vertices,
  from which \( j \) are source-like.
  The idea is to trace the term with the smallest exponent with respect to \( z \) in
  \(
    [t^{m+1} u^i v^j y^\ell]
    (\Convert{2}{2}\, \D)(z, w; u, v, y, t)
  \).
  The easiest way is to start with extracting \( [y^\ell] \) in the right-hand side of expression~\eqref{eq:cgf:multivariate:digraphs},
  and then extract other variables.
  At some point, extraction of the dominant term can be obtained by expressing all parts in terms of \( \SCC \) and  replacing all \( \SCC(z) \) with \( z \).

  Again, one can obtain the dominant term in the asymptotics combinatorially, by considering the structure of involved digraphs.
  Indeed, if a digraph has \( (m+1) \) strongly connected components,
  then, with high probability, all the components except one consist of a single vertex.
  Moreover, the large component containing the rest \( (n-m) \) vertices cannot be isolated, source-like or sink-like,
  with an exception of the two cases:
  either when all the components are isolated,
  or when there is only one source-like or sink-like component. 
  Once the large component is marked, the digraph splits into several parts, 
  and computations similar to those carried out in \cref{remark:dominant:m-fixed} yield the above result.
\end{remark}

\begin{remark}
\label{remark: asymptotics of D(n p)}
  Similarly to the undirected case discussed in \cref{remark: asymptotics of G(n p)}, we can establish asymptotics of digraphs within the \( \mathbb D(n,p) \) model.
  Again, we put \( \alpha = (1 - p)^{-1} \), so that the Exponential GF of digraphs is equal to
  \[
    \D(z) = \sum\limits_{n=0}^\infty
     \alpha^{2\binom{n}{2}} \dfrac{z^n}{n!}
   \, .
  \]
  As we have observed in \cref{remark:introducing Erdos-Renyi}, all enumerative results, with the exception of \cref{corollary: exponential gf of ssd and scc} and relation~\eqref{eq:1/G=1-IT}, remain the same.
  Therefore, to obtain the correct statements that generalize the ones we have seen in \cref{section:digraphs}, we need to replace 2 by \( \alpha \) and \( 1 - \I(z) \) by \( \G^{-1}(z) \).
  In particular, the asymptotics of strongly connected digraphs is described by
  \[
    (\Convert{\alpha}{2}\SCC)(z,w) =
    \E(\alpha^{3/2}z^2w) \cdot
      \Change{\alpha}{1}{2} \Big(\G^{-2}(\alpha zw)\Big)
    \, ,
  \]
  the generalizations of formulae \eqref{eq:cgf:ssd} and \eqref{eq:cgf:several:scc} are, respectively,
  \[
    (\Convert{\alpha}{2}\E)(z, w; t) = t \cdot
    \E(\alpha^{3/2}z^2w; 1+t) \cdot
      \Change{\alpha}{1}{2} \Big(\G^{-2}(\alpha zw)\Big)
  \]
  and
  \[
    (\Convert{\alpha}{1} \widehat{\D})(z, w; t)
      =
    \left(
      \widehat{\D}(\alpha z w; t)
    \right)^2
      \cdot
    \Change{\alpha}{2}{1}
      \Big( (\Convert{\alpha}{2}\E)(z, w; -t) \Big)
    \, ,  
  \]
  and so on.
\end{remark}

\section{2-SAT formulae}
\label{section:2-SAT}

\subsection{Asymptotics for satisfiable 2-CNFs}

  In this section, we provide the complete asymptotic expansion of satisfiable 2-CNFs,
  both in the form of the Coefficient GF and as a series.
  Additionally, we give a combinatorial interpretation of the leading term of the asymptotics of the number of satisfiable 2-CNF formulae.
  All the results presented in this section are novel.

\begin{theorem}
\label{theorem:cgf:sat}
  The Implication GF \( \SAT(z) \) of satisfiable 2-CNF formulae belongs to the ring \( \Ring{2}{1} \) and its Coefficient GF of type \( (2,1) \) is given by
  \[
    (\Convert{2}{1}\, \SAT)(z, w)
     = 
    \SAT(2zw) \big(1 - \I(2zw) \big)
    \, .
  \]
\end{theorem}
\begin{proof}
  This follows directly from \cref{lemma: operation transfers} applied to relation~\eqref{eq:igf:sat},
  \[
    \SAT(z) = \G(z) \cdot
      \Robin{2}^2
      \left(
        e^{-\tfrac12 \SCC(z)}
      \right)
    \, .
  \]
  Indeed, as we have seen in \cref{example: graphs and tournaments}, \( \G(z) \) belongs to \( \Ring{2}{1} \) and \( (\Convert{2}{1}\, \G)(z,w) = 1 \).
  On the other hand, according to \cref{theorem:cgf:ssd}, the Exponential GF \( \E(z) \) belongs to the ring \( \Ring{2}{2} \),
  and hence, by \cref{lemma: power transfer},
  so does \( e^{-\tfrac12 \SCC(z)} = \big( \E(z) \big)^{-1/2} \).
  As a consequence, due to \cref{lemma: QDA = 0},
  \[
    \Robin{2}^2
    \left(
      e^{-\tfrac12 \SCC(z)}
    \right)
    \xmapsto{\;\;\Convert{2}{1}\;\;} 0
    \, .
  \]
  Thus, 
  \[
    (\Convert{2}{1}\, \SAT)(z, w)
     = 
    (\Convert{2}{1}\, \G)(z,w) \cdot
      \Robin{2}^2
      \left(
        e^{-\tfrac12 \SCC(2zw)}
      \right)
     = 
    \dfrac{\SAT(z)}{\G(2zw)}
  \]
  and the relation \( \big(\G(2zw)\big)^{-1} = 1 - \I(2zw) \) completes the proof.
\end{proof}

\begin{corollary}
\label{corollary:asymptotics:sat_n}
  The number \( \sat_n \) of satisfiable 2-CNF formulae with \( n \) Boolean variables satisfies
  \[
    \sat_n
     \approx
    2^{3 \binom{n}{2} + n}
    \left(
      1
       - 
      \sum_{m\geq1}
       s_m^\circ
        \binom{n}{m} 
         \dfrac{2^{\binom{m+1}{2}}}{2^{mn}}
    \right)
    \, ,
  \]
  where
  \[
    s_m^\circ
     =
    \left(
      \sum_{k=0}^{m-1} \binom{m}{k}
       \dfrac{\sat_k \ir_{m-k}}{2^{k^2}}
    \right)
     -
    \dfrac{\sat_m}{2^{m^2}}
    \, ,
  \]
  and \( \ir_k \) denotes the number of irreducible tournaments with \( k \) vertices.
\end{corollary}
\begin{proof}
  This follows from \cref{theorem:cgf:sat} with the help of the definitions of the Implication GF and the Coefficient GFs by extracting the required coefficient.
\end{proof}

\begin{remark}
\label{remark:asymptotics:sat}
  The fact that \( \sat_n \sim 2^{3\binom{n}{2}+n} \) has a simple combinatorial explanation.
  With high probability, the implication digraph of a typical satisfiable formula
  consists of one pair of ordinary strongly connected components,
  supplied with additional edges going from one of these components to another.
  There are \( 2^n \) ways to choose which literals go into each of the components.
  Furthermore, there are \( 2^{n(n-1)} \) ways to choose directed edges within the first of these components
  (the edges of the second component are uniquely defined by this choice).
  Finally, there are \( 2^{\binom{n}{2}} \) ways to draw directed edges between the components,
  \ie to choose pairs of edges \( x \to \n y \) and \( y \to \n x \) such that vertices \( x, y \) belong to the first component, and \( \n x, \n y \) belong to the second one.
\end{remark}

\subsection{Further asymptotics}

  This section contains two more new asymptotic results.
  First, we establish the Coefficient GF of contradictory strongly connected implication digraphs.
  The second result is more general.
  Namely, we obtain the Coefficient GF of implication digraphs
  with marking variables for the numbers of contradictory and ordinary strongly connected components.
  We also briefly discuss the combinatorial interpretation of its leading term.

\begin{theorem}
\label{theorem:asymptotics:cscc}
  The Exponential GF \( \CSCC(z) \) of contradictory strongly connected implication digraphs belongs to the ring \( \Ring{2}{4} \),
  and its Coefficient GF of type \( (2,4) \) is given by
  \[
    (\Convert{2}{4}\, \CSCC)(z, w)
     =
    \exp \left(
      \dfrac{1}{2} \SCC(2^{7/2} z^4 w)
       -
      \CSCC(2^{5/2} z^4 w)
    \right)
     \cdot
    \Change{2}{2}{4}
     \Big( 1 - \I(2^{5/2}z^2w) \Big)
    \, .
  \]
\end{theorem}
\begin{proof}
  The proof is rather straightforward.
  Recall that, according to \cref{proposition:egf:cscc},
  \[
    \CSCC(z) = \dfrac{1}{2} \SCC(2z) +
    \log \left(
      \Robin{2}^{-2}
       \Big( \D(z) \big( 1 - \I(2z) \big) \Big)
    \right)
    \, .
  \]  
  First of all, \( \SCC(2z) \in \Ring{2}{2} \) and \( \big( 1 - \I(2z) \big) \in \Ring{2}{1} \) by \cref{theorem:asymptotics:scc} and \cref{theorem: irreducible tournaments asymptotics}, respectively.
  Hence, due to \cref{lemma: ring inclusion},
  \[
    \SCC(2z) \xmapsto{\;\;\Convert{2}{4}\;\;} 0
    \qquad\mbox{and}\qquad
    \big( 1 - \I(2z) \big)\xmapsto{\;\;\Convert{2}{2}\;\;} 0
    \, .
  \]
  This implies that
  \(
    A(z) = \left(
      \Robin{2}^{-2}
       \Big( \D(z) \big( 1 - \I(2z) \big) \Big) - 1
    \right) \in \Ring{2}{4}
  \)  
  and, according to \cref{lemma: commutative conversion} and \cref{lemma: operation transfers},
  \(
    (\Convert{2}{4} A)(z, w) =
    \Change{2}{2}{4}
     \Big( 1 - \I(2^{5/2}z^2w) \Big)
  \).
  Now, applying \cref{lemma: Bender's transfer} to \( A(z) \) and \( F(x) = \log(x+1) \) with \( \alpha = 2 \) and \( \beta = 4 \),
  we have
  \[
    H(z) = \dfrac{1}{A(z) + 1} =
    \exp \left(
      \dfrac{1}{2} \SCC(2z) - \CSCC(z)
    \right)
  \]
  and
  \(
    (\Convert{2}{4}\, \CSCC)(z, w)
     =
    H(2^{5/2}z^4w) \cdot (\Convert{2}{4} A)(z, w)
  \).
\end{proof}

\begin{theorem}
\label{theorem:cgf:cnf:types}
  The Implication GF \( \CNF(z; s, t) \) of 2-CNF formulae with variables \( s \) and \( t \) that mark, respectively,
  the numbers of contradictory strongly connected components and
  pairs of ordinary strongly connected components in the corresponding implication digraph,
  belongs to the ring \( \Ring{2}{2}(s,t) \) and the corresponding Coefficient GF of type \( (2,2) \) is given by
  \begin{multline*}
    (\Convert{2}{2}\, \CNF)(z, w; s, t)
     =
    s \cdot \widehat{\D}(2^{3/2} z^2 w; t)
     \cdot \\
    \Change{2}{4}{2}
    \left[
      z \cdot
      \exp \left(
        \big(s-1\big) \cdot \CSCC(2^{3/2} z^4 w)
         +
        \dfrac{(1-t)}{2} \cdot \SCC(2^{5/2} z^4 w)
      \right)
       \cdot
      \Change{2}{2}{4}
       \Big( 1 - \I(4z^2w) \Big)
    \right]
    \, .
  \end{multline*}
\end{theorem}
\begin{proof}
  The main idea of the proof is to apply \cref{lemma: operation transfers and marking variables} and \cref{lemma: generalized commutative conversion} to relation~\eqref{eq:cnf:types} written in the form
  \[
    \CNF(z; s, t)
     =
    \widehat{\D}(z; t) \cdot
    \Robin{2}^2 \left(
      e^{s \cdot \CSCC(z/2) - t/2 \cdot \SCC(z)}
    \right)
    \, .
  \]
  Due to \cref{lemma: ring inclusion},
  \[
    \widehat{\D}(z; t) \xmapsto{\;\;\Convert{2}{2}\;\;} 0
    \qquad\mbox{and}\qquad
    e^{t/2 \cdot \SCC(z)} \xmapsto{\;\;\Convert{2}{4}\;\;} 0
    \, .
  \]
  Therefore, the essential part of the proof comes from \cref{lemma: Bender's transfer and marking variables} applied to \( A(z) = \CSCC(z/2) \) and \( F(x;s) = e^{sx} \) with \( \alpha = 2 \) and \( \beta = 4 \):
  \[
    (\Convert{2}{2}\, \CNF)(z, w; s, t)
     =
    s \cdot \widehat{\D}(2^{3/2} z^2 w; t) \cdot
    \Change{2}{4}{2}
    \left(
      e^{
        s \cdot \CSCC(2^{3/2} z^4 w)
         -
        t/2 \cdot \SCC(2^{5/2} z^4 w)
      }
       \cdot
      (\Convert{2}{4}A)(z,w)
    \right)
    \, .
  \]
  To complete the proof, we use \cref{corollary: transfer shift}, so that
  \(
    (\Convert{2}{4}A)(z,w)
     =
    z \cdot (\Convert{2}{4}\, \CSCC)\big( 2^{-1/4}z, w \big)
  \),
  and finally, \cref{theorem:asymptotics:cscc}.
\end{proof}

\begin{remark}
\label{remark: structure of 2-SAT with m cscc}
  By analysing the leading term of the Coefficient GF in \cref{theorem:cgf:cnf:types},
  we can obtain the structure of 2-CNF formulae with given constraints.
  Thus, the structure of a typical 2-CNF implication digraph with \( (m+1) \) contradictory components is as follows.
  One large component contains almost all the variables,
  and the remaining \( m \) components contain \( 2 \) Boolean variables each.
  Note that this implication digraph, with high probability, consists of isolated contradictory components only.

  Next, if the numbers of contradictory components and ordinary components are fixed,
  then, with high probability, all the ordinary components contain one vertex each,
  and they are connected by an edge with the large contradictory component.
  More precisely, one-node components are partitioned into two copies of directed acyclic graphs:
  one of them points towards the large contradictory component,
  and the other has edges directed from that component.
  There can be, in addition, an arbitrary subset of directed edges from the first of these directed acyclic graphs towards the small 2-variable contradictory components
  (as well as the complementary subset of edges directed from small contradictory components toward the second directed acyclic graph).
\end{remark}

\begin{remark}
\label{remark:2-sat:Erdos-Renyi}
  Similarly to the Erd\H{o}s-R\'enyi model,
  let us define \( \mathbb F(n, p) \) to be the random 2-SAT model with \( n \) Boolean variables,
  so that each of the \( n(n-1) \) possible clauses appears independently with a fixed positive probability \( p \).
  In this case, the probability that a randomly generated
2-CNF formula from \( \mathbb F(n,p) \) belongs to a family \( \mathcal F \) whose Implication GF is \( \ddot F(z,w) \) is equal to
  \[
    \mathbb P_{\mathcal F}(n, p)
    = (1 - p)^{n(n-1)} 2^n n! [z^n] \ddot F \left(
        z, \dfrac{p}{1 - p}
    \right)
  \]
  (\cf \cite[Proposition 3.3]{dovgal2021exact}, compare with \cref{remark:introducing Erdos-Renyi}).
  It is natural, as we have it done for undirected and directed graphs, to introduce \( \alpha = (1 - p)^{-1} \).
  However, this leads to
  \[
    \CNF(z)
     =
    \sum_{n = 0}^\infty \alpha^{n(n-1)} \dfrac{z^n}{2^n n!}
  \]
  and the complete asymptotic expansion of the counting sequence corresponding to a 2-SAT family is not homogeneous anymore.
  Now, different powers of $2$ and $\alpha$ are interlacing, which complicates the order of terms in the full asymptotic expansion.
  Indeed, if
  \[
    f_n \approx
    \alpha^{\beta \binom{n}{2}}
      \sum_{m \geq \M} \alpha^{-mn}
      \sum_{\ell=0}^\infty
        n^{\underline{\ell}} f^\circ_{m,\ell}
    \, ,
  \]
  then, dividing the argument by \( 2 \), we have
  \[
    f_n 2^{-n} \approx
    \alpha^{\beta \binom{n}{2}}
      \sum_{m \geq \M} \alpha^{-mn} 2^{-n}
      \sum_{\ell=0}^\infty
        n^{\underline{\ell}} f^\circ_{m,\ell}
    \, .
  \]
  When the corresponding generating function is multiplied by another function or participates in a~functional composition, the order of the first few dominant terms of the expansion depends on the value of \( \alpha \).
  Furthermore, the sequence of expansion coefficients can no longer be captured by a~conventional Coefficient GF, unless when \( \log_2 \alpha \) is rational.
  All these observations show that the presented method is not applicable to the \( \mathbb F(n, p) \) model in full generality.
\end{remark}

\section{Discussion and open problems}
\label{section:discussion}

  In this section, we discuss open problems related to the asymptotic transfer method presented in this paper and to its applications.
  We have seen that this method works well for various dense graph families and 2-SAT formulae.
  Also, it can be used to study the Erd\H{o}s-R\'enyi model where edges of a random graph are drawn independently with a constant probability \( p \in (0,1) \), see \cref{remark: asymptotics of G(n p)} and \cref{remark: asymptotics of D(n p)}.
  Apparently, as \( p \to 0 \), the limit where the method ceases to be applicable is at \( p = \Theta(\log n / n) \),
  since the divergence of a quadratic term \( (1 - p)^{-\beta \binom{n}{2}} \) should be faster than the divergence of a factorial.
  At this threshold, the terms of the asymptotic expansion may start having a comparable order.
  This phenomenon has a heuristic combinatorial explanation:
  in random Erd\H{o}s-R\'enyi graphs, \( p \sim \log n/ n \) indicates the connectivity threshold where all but one component consists of a single vertex.
  A~similar phenomenon must take place for
random digraphs and 2-CNF formulae.
  Our method could be potentially applied to the case where other components have finite sizes by summing all the
contributions and using a delicate generating function argument.
  Below the threshold \( p = \Theta(\log n / n) \), a different approach is clearly required.
  The development of the mentioned technique is one of the directions for a future research.
  Note that, in the case of 2-SAT, this might be tricky because \( \alpha \) and \( 2 \) are on average algebraically independent, see~\cref{remark:2-sat:Erdos-Renyi}.

  It is of interest whether our method can be extended to the case of a functional composition of two graphically divergent generating functions, as it happens in the case of factorially divergent series (see the Borinsky's paper~\cite{borinsky2018generating}).
  The positive answer to this question would potentially unlock refined asymptotic enumeration of \emph{2-vertex-connected graphs} (also known as \emph{nonseparable graphs} or \emph{blocks}, \ie connected graphs without cutpoints) and \emph{2-edge-connected graphs} (connected graphs without bridges),
  whose respective Exponential GFs \( B(x) \) and \( H(x) \) satisfy functional equations
  \[
    \CG'(x) = e^{B'(x \CG'(x))}
    \qquad \mbox{and} \qquad
    \CG'(x) = H'(x e^{x\CG'(x)})
    \, .
  \]
  Another possible way to get this enumeration would be to restate Lagrange inversion of appropriate generating series in terms of Coefficient GFs
  (see the paper~\cite{bender1984asymptotic} of Bender and Richmond who established the first of these asymptotics by developing the method for treating inverses and functional compositions with analytic functions).

  Next questions are related to the \( k \)-SAT problem.
  As we have seen in \cref{remark:asymptotics:sat}, the fact that the number of satisfiable 2-SAT formulae grows asymptotically as \( 2^{n + 3 \binom{n}{2}} \) has a simple combinatorial interpretation in terms of implication digraphs.
  Would there be an equivalently simple heuristic explanation for the asymptotic number of satisfiable \( k \)-CNF formulae on \( n \) Boolean variables?
  Furthermore, the Coefficient GF of the satisfiable 2-SAT formulae has a remarkably simple form
  \(
    (\Convert{2}{1}\, \SAT)(z, w)
     = 
    \SAT(2zw) \big( 1 - \I(2zw) \big)
  \).
  Could there be an equally simple expression for a sort of Coefficient GF for \( k \)-SAT with \( k > 2 \), even if their exact enumeration is elusive?

  It is worth mentioning that, for problems like \( k \)-SAT, the logarithm of the total number of objects is growing faster than a quadratic function.
  This may suggest that the required analogue of the Coefficient GF must have more than two dimensions, which leads to another question.
  Namely, could the asymptotic transfer method be meaningfully generalized to higher dimensions as well?
  The latter would be useful, for instance, to count families of hypergraphs and directed hypergraphs.

  We conclude our review of open problems with a particular question related to the enumeration of digraphs.
  Curiously, the statement of \cref{corollary: exponential gf of ssd and scc} suggest that there might be a combinatorial explanation of expressions~\eqref{eq: semi-strong as sequences}.
  The first of them could follow from the fact that the family of semi-strong digraphs would be in a natural one-to-one correspondence with sequences of irreducible tournaments decorated with an arbitrary subset of \( \binom{n}{2} \) edges of additional color.
  The corresponding counting sequence \( (\eta_n)_{n=1}^\infty \) first appeared in the paper of Wright~\cite{wright1971number}
  who obtained the following recurrence for the counting sequence \( (\scc_n) \) of strongly connected digraphs:
  \[
    \scc_n = \eta_n + \sum_{t = 1}^{n - 1} \binom{n-1}{t-1} \scc_t \eta_{n-t} \, .
  \]
  Liskovets later discovered in~\cite{liskovets1975} that
  \( \eta_n 2^{-\binom{n}{2}} \) indeed enumerates irreducible tournaments and even extended this enumeration result to the case of unlabelled structures.
  Recently, Archer, Gessel, Graves and Liang~\cite{archer2020counting}, among other
results, revealed some fine enumerative properties of the combinatorial class corresponding to \( \eta_n \).
  They also noted that there might be a natural bijection between strong digraphs and cycles of irreducible decorated tournaments, but could not identify such a bijection
  (which would correspond to the second expression of~\eqref{eq: semi-strong as sequences}).
  Unfortunately, despite our attempts, we have not been able to find any of such bijections either and, to our best knowledge, they still remain an open problem.

\paragraph*{Acknowledgements.}
  Sergey Dovgal was supported by the EIPHI Graduate School (contract ANR-17-EURE-0002), FEDER and Région Bourgogne Franche-Comté.

  Khaydar Nurligareev was supported by the project PICS ANR-22-CE48-0002, as well as the ANR-FWF project PAnDAG ANR-23-CE48-0014-01, both funded by l’Agence Nationale de la Recherche.

\appendix
\section{Numerical values of Coefficient GFs}
\label{section:appendix}

\subsection{Connected graphs}
\label{Appendix: connected graphs}

  According to \cref{theorem: connected graphs asymptotics}, the Coefficient GF of type \( (2,1) \) of connected graphs satisfies
  \[
	(\Convert{2}{1}\, \CG)(z,w) = 1 - \I(2zw).
  \]
  As a consequence, the corresponding asymptotic coefficients \( \cg^\circ_{m,\ell} \) are of form
  \[
    \cg^\circ_{m,\ell}
    =
    \mathbf 1_{m = \ell = 0}
    -
    \mathbf 1_{m = \ell > 0}
    \cdot \ir_m \cdot
	\dfrac{2^{\binom{m+1}{2}}}{m!}
	\, .
  \]
  The sequence \( (\ir_m)_{m=1}^{\infty} \) counts irreducible tournaments and is given by \href{https://oeis.org/A054946}{A054946} from the OEIS:
  \[
    (\ir_m)
     =
    1,\,
    0,\,
    2,\,
    24,\,
    544,\,
    22\,320,\,
    1\,677\,488,\,
    236\,522\,496,\,
    64\,026\,088\,576,\,
    33\,832\,910\,196\,480,\,
    \ldots
  \]
  Thus, the sequence \( \big(\cg^\circ_{m,m}\big)_{m=0}^{\infty} \) starts by
  \[
    \big(\cg^\circ_{m,m}\big)
     = 
    1,\,
    -2,\,
    0,\,
    -\dfrac{64}{3},\,
    -1024,\,
    -\dfrac{2\,228\,224}{15},\,
    -65\,011\,712,\,
    -\dfrac{28\,143\,578\,513\,408}{315},\,
    \ldots
  \]

\subsection{Graphs, counting connected components}
\label{Appendix: counting components of graphs}
  According to \cref{cor: graphs with m components asymptotics}, the Coefficient GF of type \( (2,1) \) of graphs, enriched with a marking variable \( t \) for counting connected components, satisfies
  \[
	(\Convert{2}{1}\, \G)(z,w;t)
	 =
	t \cdot \G(2zw;t) \cdot \big( 1 - \I(2zw) \big)
	\, .
  \]
  This series starts with the following terms:
  \[
	(\Convert{2}{1}\, \G)(z,w;t)
	 =
	t + 2(t^2-t)zw + 2(t^3-t^2)z^2w^2 + \dfrac{4}{3}(t^4+t^2-2t)z^3w^3 + \ldots
  \]
  The corresponding asymptotic coefficients \( \g^\circ_{m,\ell}(t) \) can be written as
  \[
    \g^\circ_{m,\ell}(t)
    =
    \mathbf 1_{m = \ell = 0}
    -
    \mathbf 1_{m = \ell > 0} \cdot
    \left(\bar{\g}_m^{(m+1)} t^{m+1} + \ldots + \bar{\g}_m^{(1)} t\right)
	\cdot \dfrac{2^{\binom{m+1}{2}}}{m!}
	\, ,
  \]
  where \( \bar{\g}_m^{(k)} \) are integers listed in \cref{tab:graphs}.

\begin{table}[ht!]
  \caption{\label{tab:graphs} Values of the coefficients \( \bar{\g}_m^{(k)} \) for \( m,k \leq 9 \).}
 \(
  \begin{array}{c|cccccccccc}
   m & 0 & 1 & 2 & 3 & 4 & 5 & 6 & 7 & 8 & 9 \\
   \hline
  \big(\bar{\g}_m^{(1)}\big) &  1 & -1 & 0 & -2 & -24 & -544 & -22\,320 & -1\,677\,488 & -236\,522\,496 & -64\,026\,088\,576 \\
  \big(\bar{\g}_m^{(2)}\big) & 0 & 1 & -1 & 1 & 14 & 398 & 18\,552 & 1\,505\,644 & 222\,306\,448 & 61\,826\,469\,776 \\ 
  \big(\bar{\g}_m^{(3)}\big) & 0 & 0 & 1 & 0 & 7 & 115 & 3\,238 & 156\,576 & 13\,457\,052 & 2\,131\,689\,876 \\ 
  \big(\bar{\g}_m^{(4)}\big) & 0 & 0 & 0 & 1 & 2 & 25 & 455 & 13\,783 & 711\,788 & 65\,405\,368 \\ 
  \big(\bar{\g}_m^{(5)}\big) & 0 & 0 & 0 & 0 & 1 & 5 & 65 & 1\,330 & 43\,673 & 2\,400\,363 \\ 
  \big(\bar{\g}_m^{(6)}\big) & 0 & 0 & 0 & 0 & 0 & 1 & 9 & 140 & 3\,248 & 115\,689 \\
  \big(\bar{\g}_m^{(7)}\big) & 0 & 0 & 0 & 0 & 0 & 0 & 1 & 14 & 266 & 7\,014 \\
  \big(\bar{\g}_m^{(8)}\big) & 0 & 0 & 0 & 0 & 0 & 0 & 0 & 1 & 20 & 462 \\
  \big(\bar{\g}_m^{(9)}\big) & 0 & 0 & 0 & 0 & 0 & 0 & 0 & 0 & 1 & 27 \\
  \end{array}
 \)
\end{table}
  
  For a fixed positive integer \( k \), the sequence  \( \big(\bar{\g}_m^{(k)}\big)_{m=0}^{\infty} \) appears in the asymptotics of graphs with \( k \) connected components,
  which is given by \( [t^k] (\Convert{2}{1}\, \G)(z,w;t) \), see relation~\eqref{eq: graphs with fixed number of connected components}.
  In particular, the case \( k = 1 \) corresponds to the asymptotics of connected graphs discussed in \cref{Appendix: connected graphs},
  and
  \[
    \bar{\g}_m^{(1)} = -\ir_m
    \, .
  \]

\subsection{Irreducible tournaments}
\label{Appendix: irreducible tournaments}

  According to \cref{theorem: irreducible tournaments asymptotics}, the Coefficient GF of irreducible tournaments of type \( (2,1) \) satisfies
  \[
	(\Convert{2}{1}\, \I)(z,w) = \big(1 - \I(2zw)\big)^2.
  \]
  This gives us the asymptotic coefficients \( \ir^\circ_{m,\ell} \) that turn out to be
  \[
    \ir^\circ_{m,\ell}
    =
    \mathbf 1_{m = \ell = 0}
    -
    \mathbf 1_{m = \ell > 0}
    \cdot \big(2\ir_m - \ir_m^{(2)}\big) \cdot
	\dfrac{2^{\binom{m+1}{2}}}{m!}
	\, ,
  \]
  where the counting sequence \( (\ir_m)_{m=1}^{\infty} \) of irreducible tournaments is described in \cref{Appendix: connected graphs},
  and the sequence \( \big(\ir_m^{(2)}\big)_{m=1}^{\infty} \) of tournaments with exactly two irreducible parts is given by
  \[
    \big(\ir_m^{(2)}\big)
     = 
    0,\,
    2,\,
    0,\,
    16,\,
    240,\,
    6\,608,\,
    315\,840,\,
    27\,001\,984,\,
    4\,268\,194\,560,\,
    1\,281\,626\,527\,232,\,
    \ldots
  \]
  Thus, we have the following starting values of \( \big(\ir^\circ_{m,m}\big)_{m=0}^{\infty} \):
  \[
    \big(\ir^\circ_{m,m}\big)
     =
    1,\,
    -4,\,
    8,\,
    -\dfrac{128}{3},\,
    -\dfrac{4\,096}{3},
    -\dfrac{3\,473\,408}{15},\,
    -\dfrac{4\,984\,930\,304}{45},\,
    -\dfrac{50\,988\,241\,125\,376}{315},\,
    \ldots
  \]

\subsection{Tournaments, counting irreducible parts}
\label{Appendix: counting parts of tournaments}
  According to \cref{cor: tournament with m irreducible parts}, the Coefficient GF of type \( (2,1) \) of tournaments, enriched with a~marking variable \( t \) for counting irreducible parts, satisfies
  \[
	(\Convert{2}{1}\, \T)(z,w;t)
	 =
	t \cdot \left(
	  \T(2zw;t) \cdot \big( 1 - \I(2zw) \big)
	\right)^2
	\, .
  \]
  The first several terms of this series are
  \[
	(\Convert{2}{1}\, \T)(z,w;t)
	 =
	t + 4(t^2-t)zw + 4(3t^3-4t^2+t)z^2w^2 + \dfrac{16}{3}(6t^4-9t^3+4t^2-t)z^3w^3 + \ldots
  \]
  The corresponding asymptotic coefficients \( \tu^\circ_{m,\ell}(t) \) can be written as
  \[
    \tu^\circ_{m,\ell}(t)
    =
    \mathbf 1_{m = \ell = 0}
    -
    \mathbf 1_{m = \ell > 0} \cdot
    \left(\bar{\tu}_m^{(m+1)} t^{m+1} + \ldots + \bar{\tu}_m^{(1)} t\right)
	\cdot \dfrac{2^{\binom{m+1}{2}}}{m!}
	\, ,
  \]
  where \( \bar{\tu}_m^{(k)} \) are integers listed in \cref{tab:tournaments}.

\begin{table}[ht!]
  \caption{\label{tab:tournaments} Values of the coefficients \( \bar{\tu}_m^{(k)} \) for \( m,k \leq 9 \).}
  \small
 \(
  \begin{array}{c|cccccccccc}
   m & 0 & 1 & 2 & 3 & 4 & 5 & 6 & 7 & 8 & 9 \\
   \hline
  \big(\bar{\tu}_m^{(1)}\big) & 1 & -2 & 2 & -4 & -32 & -848 & -38\,032 & -3\,039\,136 & -446\,043\,008 & -123\,783\,982\,592 \\ 
  \big(\bar{\tu}_m^{(2)}\big) & 0 & 2 & -8 & 16 & -16 & 368 & 22\,528 & 2\,232\,064 & 372\,697\,856 &  111\,712\,858\,112 \\
  \big(\bar{\tu}_m^{(3)}\big) & 0 & 0 & 6 & -36 & 120 & 0 & 9\,744 & 586\,656 & 60\,297\,600 & 10\,743\,552\,000 \\
  \big(\bar{\tu}_m^{(4)}\big) & 0 & 0 & 0 & 24 & -192 & 960 & 960 & 153\,216 & 10\,063\,872 & 1\,129\,843\,200 \\
  \big(\bar{\tu}_m^{(5)}\big) & 0 & 0 & 0 & 0 & 120 & -1\,200 & 8\,400 & 16\,800 & 2\,177\,280 & 156\,844\,800 \\ 
  \big(\bar{\tu}_m^{(6)}\big) & 0 & 0 & 0 & 0 & 0 & 720 & -8\,640 & 80\,640 & 241\,920 & 30\,723\,840 \\
  \big(\bar{\tu}_m^{(7)}\big) & 0 & 0 & 0 & 0 & 0 & 0 & 5\,040 & -70\,560 & 846\,720 & 3\,386\,880 \\
  \big(\bar{\tu}_m^{(8)}\big) & 0 & 0 & 0 & 0 & 0 & 0 & 0 & 40\,320 & -645\,120 & 9\,676\,800 \\
  \big(\bar{\tu}_m^{(9)}\big) & 0 & 0 & 0 & 0 & 0 & 0 & 0 & 0 & 362\,880 & -6\,531\,840 \\
  \end{array}
 \)
\end{table}
  
  For a fixed positive integer \( k \), the sequence  \( \big(\bar{\tu}_m^{(k)}\big)_{m=0}^{\infty} \) appears in the asymptotics of tournaments with \( k \) irreducible parts,
  which is given by \( [t^k] (\Convert{2}{1}\, \T)(z,w;t) \), see relation~\eqref{eq: tournaments with fixed number of irreducible parts}.
  In particular,
  \[
    \bar{\tu}_m^{(k)} =
    k \cdot \left(
      \ir_m^{(k-1)} - 2\ir_m^{(k)} + \ir_m^{(k+1)}
    \right)
    \, ,
  \]
  where \( \ir_m^{(k)} \) is the number of tournaments of size \( m \) with \( k \) irreducible parts,
  supplemented by the convention \( \ir_m^{(0)} = \mathbf 1_{m = 0} \) that leads us to the asymptotics of irreducible tournaments discussed in \cref{Appendix: irreducible tournaments}.

\subsection{The Erd\H{o}s-R\'enyi model}
\label{Appendix: Erdos-Renyi}
  As it was mentioned in \cref{remark: asymptotics of G(n p)},
  the Coefficient GF of type \( (\alpha,1) \) of graphs within the Erd\H{o}s-R\'enyi model \(\mathbb{G}(n,p) \) satisfies
  \[
	(\Convert{\alpha}{1}\G)(z,w;t)
	 =
	t \cdot \dfrac{\G(\alpha zw;t)}{\G(\alpha zw)}
	 =
	t \cdot e^{(t-1) \cdot \CG(\alpha zw)}
	\, ,
  \]
  where \( \alpha = (1 - p)^{-1} \)
  and \( t \) is the marking variable for connected components.
  The corresponding asymptotic coefficients \( \g^\circ_{m,\ell}(\alpha,t) \) can be written as
  \[
    \g^\circ_{m,\ell}(\alpha,t)
    =
    \mathbf 1_{m = \ell = 0}
    -
    \mathbf 1_{m = \ell > 0} \cdot
    \left(\bar{\g}_m^{(m+1)}(\alpha)\, t^{m+1} + \ldots + \bar{\g}_m^{(1)}(\alpha)\, t\right)
	\cdot \dfrac{\alpha^{\binom{m+1}{2}}}{m!}
	\, ,
  \]
  where \( \bar{\g}_m^{(k)}(\alpha) \) are polynomials in \( \alpha \) listed in \cref{tab:Erdos-Renyi}.

\begin{table}[ht!]
  \caption{\label{tab:Erdos-Renyi} Values of the coefficients \( \bar{\g}_m^{(k)}(\alpha) \) for \( m,k \leq 4 \).}
 \(
  \begin{array}{c|cccccc}
   m & 0 & 1 & 2 & 3 & 4 \\ 
   \hline
  \big(\bar{\g}_m^{(1)}(\alpha)\big) & 1 & -1 & -\alpha + 2 & -\alpha^3 + 6\alpha - 6 & -\alpha^6 + 8\alpha^3 + 6\alpha^2 - 36\alpha + 24 \\ 
  \big(\bar{\g}_m^{(2)}(\alpha)\big) & 0 & 1 & \alpha - 3 & \alpha^3 - 9\alpha + 11 & \alpha^6 - 12\alpha^3 - 9\alpha^2 + 66\alpha - 50 \\ 
  \big(\bar{\g}_m^{(3)}(\alpha)\big) & 0 & 0 & 1 & 3\alpha - 6 & 4\alpha^3 + 3\alpha^2 - 36\alpha + 35 \\ 
  \big(\bar{\g}_m^{(4)}(\alpha)\big) & 0 & 0 & 0 & 1 & 6\alpha - 10 \\ 
  \end{array}
 \)
\end{table}

  We observe that the column sums in \cref{tab:Erdos-Renyi} equal zero,
  except for the column that corresponds to \( m = 0 \).
  Clearly, this can be explained by the fact that the total weight of all graphs within the Erd\H{o}s-R\'enyi model is \( \alpha^{\binom{n}{2}} \).
  Another observation is that the case where \( p = 1/2 \), \ie \( \alpha = 2 \), corresponds to the situation discussed in \cref{Appendix: counting components of graphs}.
  In other words, substituting the value \( \alpha = 2 \) into \cref{tab:Erdos-Renyi}, we obtain \cref{tab:graphs}.

 For a fixed non-negative integer \( k \), the sequence  \( \big(\bar{\g}_m^{(k+1)}(\alpha)\big)_{m=0}^{\infty} \) appears in the asymptotics of graphs with exactly \( (k+1) \) connected components,
  which is given by the relation
  \[
    [t^{k+1}] (\Convert{\alpha}{1}\, \G)(z,w;t)
	 =
    \dfrac{1}{k!}
	 \cdot
	\dfrac{\CG^k(\alpha zw)}{\G(\alpha zw)}
    \, .
  \]
  In particular, the case \( k = 0 \) corresponds to the asymptotics of connected graphs:
  \[
	(\Convert{\alpha}{1}\CG)(z,w)
	 = 
	\dfrac{1}{\G(\alpha zw)}
	 = 
	e^{-\CG(\alpha zw)}
	\, .
  \]
  The last relation can be interpreted in the following way:
  the probability \( p_n \) that a random graph of size \( n \) within the Erd\H{o}s-R\'enyi model is connected satisfies
  \[
	p_n \approx
	1 - \sum\limits_{m\geq1}
	 P_m(\alpha)
	  \binom{n}{m}
	   \dfrac{\alpha^{\binom{m+1}{2}}}{\alpha^{mn}}
	\, ,
  \]
  where \( P_m(\alpha) = - \bar{\g}_m^{(1)}(\alpha)\).
  The first six polynomials \( P_m(\alpha) \) are
  \begin{align*}
    P_1(\alpha) &= 1, \\
    P_2(\alpha) &= \alpha - 2, \\
    P_3(\alpha) &= \alpha^3 - 6\alpha + 6, \\
    P_4(\alpha) &= \alpha^6 - 8\alpha^3 - 6\alpha^2 + 36\alpha - 24, \\
    P_5(\alpha) &= \alpha^{10} - 10\alpha^6 - 20\alpha^4 + 60\alpha^3 + 90\alpha^2 - 240\alpha + 120, \\
    P_6(\alpha) &= \alpha^{15} - 12\alpha^{10} - 30\alpha^7 + 70\alpha^6 + 360\alpha^4 - 390\alpha^3 - 1080\alpha^2 + 1800\alpha - 720 \, . \\
  \end{align*}
  
\subsection{Strongly connected digraphs}
\label{Appendix: strongly connected digraphs}

  According to \cref{theorem:asymptotics:scc}, the Coefficient GF of type \( (2,2) \) of strongly connected digraphs satisfies
  \[
    (\Convert{2}{2}\, \SCC)(z,w)
     =
    \E(2^{3/2}z^2w)
     \cdot
    \Change{2}{1}{2} \big(1 - \I(2zw)\big)^2
    \, ,
  \]
  or, in terms of the exponential Hadamard product,
  \[
    (\Convert{2}{2}\, \SCC)(z,w)
     =
    \E(2^{3/2}z^2w)
     \cdot
    \left(
      \big(1 - \I(2zw)\big)^2
       \odot_z
      \G(z, \sqrt2 - 1)
    \right)
    \, .
  \]
  Due to \cref{corollary:scc probability}, the corresponding coefficients \( \scc^\circ_{m,\ell} \) are
  \[
    \scc^\circ_{m,\ell}
     =
    2^{m(m+1)/2 + \ell(\ell-m)}
    \dfrac{\ssd_{m-\ell}}{(m-\ell)!}
    \dfrac{\mathbf 1_{m=2\ell} - 2\ir_{2\ell-m} + \ir^{(2)}_{2\ell-m}}{(2\ell-m)!}
	\, .
  \]
  The sequence \( (\ssd_k)_{k=0}^{\infty} \) counts semi-strong digraphs and is given by \href{https://oeis.org/A054948}{A054948} from the OEIS:
 \[
  (\ssd_k)
   =
  1,\,
  1,\,
  2,\,
  22,\,
  1\,688,\,
  573\,496,\,
  738\,218\,192,\,
  3\,528\,260\,038\,192,\,
  63\,547\,436\,065\,854\,848,\,
  \ldots
 \]
 Together with the values of \( (\ir_k)_{k=1}^{\infty} \) and \( \big(\ir_k^{(2)}\big)_{k=1}^{\infty} \) indicated in \cref{Appendix: connected graphs} and \cref{Appendix: irreducible tournaments}, respectively,
 this gives us numerical values of \( \scc^\circ_{m,\ell} \) indicated in \cref{tab:scc}.
\begin{table}[ht!]
  \caption{\label{tab:scc} Values of the coefficients \( \scc^\circ_{m,\ell} \) for \( m,\ell \leq 7 \).}
 \(
  \begin{array}{c|cccccccc}
   \ell & 0 & 1 & 2 & 3 & 4 & 5 & 6 & 7 \\
   \hline
   \scc^\circ_{0,\ell} & 1 & 0 & 0 & 0 & 0 & 0 & 0 & 0 \\
   \scc^\circ_{1,\ell} & 0 & -4 & 0 & 0 & 0 & 0 & 0 & 0 \\
   \scc^\circ_{2,\ell} & 0 & 4 & 8 & 0 & 0 & 0 & 0 & 0 \\
   \scc^\circ_{3,\ell} & 0 & 0 & -32 & -\frac{128}{3} & 0 & 0 & 0 & 0 \\
   \scc^\circ_{4,\ell} & 0 & 0 & 64 & 128 & -\frac{4\,096}{3} & 0 & 0 & 0 \\
   \scc^\circ_{5,\ell} & 0 & 0 & 0 & -1\,024 & -\frac{4\,096}{3} & -\frac{3\,473\,408}{15} & 0 & 0 \\
   \scc^\circ_{6,\ell} & 0 & 0 & 0 & \frac{45\,056}{3} & 8\,192 & -\frac{262\,144}{3} & -\frac{4\,984\,930\,304}{45} & 0 \\
   \scc^\circ_{7,\ell} & 0 & 0 & 0 & 0 & -\frac{1\,441\,792}{3} & -\frac{524\,288}{3} & -\frac{444\,596\,224}{15} & -\frac{50\,988\,241\,125\,376}{315} \\
  \end{array}
 \)
\end{table}

  If we rewrite the above relation as
  \[
    \scc_n \approx 2^{2\binom{n}{2}}
      \sum_{m\geq0} \dfrac{w_m(n)}{2^{nm}}
    \, ,
	\qquad\mbox{where}\quad
    w_m(n) = \sum_{\ell=\lceil m/2 \rceil}^{m}
      n^{\underline{\ell}}\,
      \scc^\circ_{m,\ell}
    \, ,
  \] 
  then, for \( m\leq6 \), the polynomials \( w_m(n) \) have the following explicit form:
  \begin{align*}
    w_0(n) &= 1, \\
    w_1(n) &= -4n, \\
    w_2(n) &= 4n(2n-1), \\
    w_3(n) &= -\dfrac{32}{3} n(n-1)(4n-5), \\
    w_4(n) &= -\dfrac{64}{3} n(n-1)(64n^2-326n+393), \\
    w_5(n) &= -\dfrac{1\,024}{15} n(n-1)(n-2) (3\,392 n^2 - 23\,724n + 40\,659) \, , \\
    w_6(n) &= -\dfrac{4\,096}{45} n(n-1)(n-2) (1\,217\,024n^3 - 14\,603\,328n^2 + 57\,193\,318n - 73\,009\,815) \, .
  \end{align*}
  This corresponds to the asymptotic expansion of strongly connected digraphs established by Wright~\cite{wright1971number} and Bender~\cite{bender1975asymptotic}.
  In their papers, a different notations were used: the asymptotics was expressed via three sequences \( (\eta_k)_{k=0}^\infty \), \( (\gamma_k)_{k=0}^\infty \) and \( (\xi_k)_{k=0}^\infty \) determined by certain recurrences.
  It follows from \cref{theorem:asymptotics:scc} and \cref{corollary:scc probability} that these sequences satisfy the following relations:
  \[
    \eta_k = 2^{\binom{k}{2}}\ir_k \, ,
    \qquad
    \gamma_k = \dfrac{\ssd_k}{k!} \, ,
    \qquad
    \xi_k = \dfrac{b_k}{k!} \, ,
  \]
  where the sequence \( (b_k)_{k=0}^\infty \) is defined by~\eqref{equation:b_n}.

  Furthermore, Wright explicitly computed the polynomials of \( w_m(n) \) for  \( m\leq5 \).
  We use the occasion to fix a typo in his expression for \( w_5(n) \).
  The corrected value is indicated above, while Wright mistakenly omitted the last digit in the number \( 23\,724 \).

\subsection{Semi-strong digraphs, counting strongly connected components}
\label{Appendix: semi-strong digraphs}
  According to \cref{theorem:cgf:ssd}, the Coefficient GF of type \( (2,2) \) of semi-strong digraphs satisfies
  \[
    (\Convert{2}{2}\, \E)(z, w; t) = t \cdot
    \E(2^{3/2}z^2w; 1+t) \cdot
      \Change{2}{1}{2} \big(1 - \I(2zw)\big)^2
    \, ,  
  \]
  where \( t \) is the marking variable for strongly connected components.
  In terms of the exponential Hadamard product, we can rewrite this formula in the following way:
  \[
    (\Convert{2}{2}\, \E)(z, w; t)
     =
    t \cdot
    e^{(t+1) \cdot \SCC(2^{3/2}z^2w)}
     \cdot
    \left(
      \big(1 - \I(2zw)\big)^2
       \odot_z
      \G(z, \sqrt2 - 1)
    \right)
    \, .
  \]
  The corresponding coefficients \( \ssd^\circ_{m,\ell}(t) \) are polynomials in \( t \) listed in \cref{tab:ssd(t)}.
  
\begin{table}[ht!]
  \caption{\label{tab:ssd(t)} Values of the coefficients \( \ssd^\circ_{m,\ell}(t) \) for \( m,\ell \leq 5 \).}
 \(
  \begin{array}{c|cccccccc}
   \ell & 0 & 1 & 2 & 3 & 4 & 5 \\
   \hline
   \ssd^\circ_{0,\ell}(t) & t & 0 & 0 & 0 & 0 & 0 \\
   \ssd^\circ_{1,\ell}(t) & 0 & -4t & 0 & 0 & 0 & 0 \\
   \ssd^\circ_{2,\ell}(t) & 0 & 4(t^2+t) & 8t & 0 & 0 & 0 \\
   \ssd^\circ_{3,\ell}(t) & 0 & 0 & -32(t^2+t) & -\frac{128}{3}t & 0 & 0 \\
   \ssd^\circ_{4,\ell}(t) & 0 & 0 & 32(t^3+3t^2+2t) & 128(t^2+t) & -\frac{4\,096}{3}t & 0 \\
   \ssd^\circ_{5,\ell}(t) & 0 & 0 & 0 & -512(t^3+3t^2+2t) & -\frac{4\,096}{3}(t^2+t) & -\frac{3\,473\,408}{15}t \\
  \end{array}
 \)
\end{table}

  Putting \( t = 1 \), we obtain the asymptotic coefficients of semi-strong digraphs without reference to strongly connected components, see \cref{tab:ssd}.
  
\begin{table}[ht!]
  \caption{\label{tab:ssd} Values of the coefficients \( \ssd^\circ_{m,\ell} \) for \( m,\ell \leq 7 \).}
 \(
  \begin{array}{c|cccccccc}
   \ell & 0 & 1 & 2 & 3 & 4 & 5 & 6 & 7 \\
   \hline
   \ssd^\circ_{0,\ell} & 1 & 0 & 0 & 0 & 0 & 0 & 0 & 0 \\
   \ssd^\circ_{1,\ell} & 0 & -4 & 0 & 0 & 0 & 0 & 0 & 0 \\
   \ssd^\circ_{2,\ell} & 0 & 8 & 8 & 0 & 0 & 0 & 0 & 0 \\
   \ssd^\circ_{3,\ell} & 0 & 0 & -64 & -\frac{128}{3} & 0 & 0 & 0 & 0 \\
   \ssd^\circ_{4,\ell} & 0 & 0 & 192 & 256 & -\frac{4\,096}{3} & 0 & 0 & 0 \\
   \ssd^\circ_{5,\ell} & 0 & 0 & 0 & -3\,072 & -\frac{8\,192}{3} & -\frac{3\,473\,408}{15} & 0 & 0 \\
   \ssd^\circ_{6,\ell} & 0 & 0 & 0 & \frac{114\,688}{3} & 24\,576 & -\frac{524\,288}{3} & -\frac{4\,984\,930\,304}{45} & 0 \\
   \ssd^\circ_{7,\ell} & 0 & 0 & 0 & 0 & -\frac{3\,670\,016}{3} & -524\,288 & -\frac{889\,192\,448}{15} & -\frac{50\,988\,241\,125\,376}{315} \\
  \end{array}
 \)
\end{table}

  Another option is to extract \( k \)th coefficient of \( (\Convert{2}{2}\, \E)(z, w; t) \) in \( t \),
  which leads to the asymptotics of semi-strong digraphs with \( k \) strongly connected components (see relation~\eqref{eq:[t^k]cgf:ssd}).
  For \( k = 1 \), these digraphs are strongly connected, and their asymptotics was discussed in \cref{Appendix: semi-strong digraphs}, see \cref{tab:scc}.
  For the cases \( k = 2 \) and \( k = 3 \), we present the asymptotic coefficients in \cref{tab:2-scc_in_ssd} and \cref{tab:3-scc_in_ssd}, respectively.

\begin{table}[ht!]
  \caption{\label{tab:2-scc_in_ssd} Values of the coefficients \( [t^2]\ssd^\circ_{m,\ell}(t) \) for \( m,\ell \leq 8 \).}
 \(
  \begin{array}{c|ccccccccc}
   \ell & 0 & 1 & 2 & 3 & 4 & 5 & 6 & 7 & 8 \\
   \hline
   \verb"["t^2\verb"]"\ssd^\circ_{0,\ell}(t) & 0 & 0 & 0 & 0 & 0 & 0 & 0 & 0 & 0 \\
   \verb"["t^2\verb"]"\ssd^\circ_{1,\ell}(t) & 0 & 0 & 0 & 0 & 0 & 0 & 0 & 0 & 0 \\
   \verb"["t^2\verb"]"\ssd^\circ_{2,\ell}(t) & 0 & 4 & 0 & 0 & 0 & 0 & 0 & 0 & 0 \\
   \verb"["t^2\verb"]"\ssd^\circ_{3,\ell}(t) & 0 & 0 & -32 & 0 & 0 & 0 & 0 & 0 & 0 \\
   \verb"["t^2\verb"]"\ssd^\circ_{4,\ell}(t) & 0 & 0 & 96 & 128 & 0 & 0 & 0 & 0 & 0 \\
   \verb"["t^2\verb"]"\ssd^\circ_{5,\ell}(t) & 0 & 0 & 0 & -1\,536 & -\frac{4\,096}{3} & 0 & 0 & 0 & 0 \\
   \verb"["t^2\verb"]"\ssd^\circ_{6,\ell}(t) & 0 & 0 & 0 & 18\,432 & 12\,288 & -\frac{262\,144}{3} & 0 & 0 & 0 \\
   \verb"["t^2\verb"]"\ssd^\circ_{7,\ell}(t) & 0 & 0 & 0 & 0 & -589\,824 & 262\,144 & -\frac{444\,596\,224}{15} & 0 & 0 \\
   \verb"["t^2\verb"]"\ssd^\circ_{8,\ell}(t) & 0 & 0 & 0 & 0 & \frac{233\,046\,016}{3} & 9\,437\,184 & -33\,554\,432 & -\frac{1\,276\,142\,157\,824}{45} & 0 \\
  \end{array}
 \)
\end{table}
  
\begin{table}[ht!]
  \caption{\label{tab:3-scc_in_ssd} Values of the coefficients \( [t^3]\ssd^\circ_{m,\ell}(t) \) for \( m,\ell \leq 9 \).}
 \(
  \begin{array}{c|cccccccccc}
   \ell & 0 & 1 & 2 & 3 & 4 & 5 & 6 & 7 & 8 & 9 \\
   \hline
   \verb"["t^3\verb"]"\ssd^\circ_{0,\ell}(t) & 0 & 0 & 0 & 0 & 0 & 0 & 0 & 0 & 0 & 0 \\
   \verb"["t^3\verb"]"\ssd^\circ_{1,\ell}(t) & 0 & 0 & 0 & 0 & 0 & 0 & 0 & 0 & 0 & 0 \\
   \verb"["t^3\verb"]"\ssd^\circ_{2,\ell}(t) & 0 & 0 & 0 & 0 & 0 & 0 & 0 & 0 & 0 & 0 \\
   \verb"["t^3\verb"]"\ssd^\circ_{3,\ell}(t) & 0 & 0 & 0 & 0 & 0 & 0 & 0 & 0 & 0 & 0 \\
   \verb"["t^3\verb"]"\ssd^\circ_{4,\ell}(t) & 0 & 0 & 32 & 0 & 0 & 0 & 0 & 0 & 0 & 0 \\
   \verb"["t^3\verb"]"\ssd^\circ_{5,\ell}(t) & 0 & 0 & 0 & -512 & 0 & 0 & 0 & 0 & 0 & 0 \\
   \verb"["t^3\verb"]"\ssd^\circ_{6,\ell}(t) & 0 & 0 & 0 & 4\,096 & 4\,096 & 0 & 0 & 0 & 0 & 0 \\
   \verb"["t^3\verb"]"\ssd^\circ_{7,\ell}(t) & 0 & 0 & 0 & 0 & -131\,072 & -\frac{262\,144}{3} & 0 & 0 & 0 & 0 \\
   \verb"["t^3\verb"]"\ssd^\circ_{8,\ell}(t) & 0 & 0 & 0 & 0 & 4\,325\,376 & 2\,097\,152 & -\frac{33\,554\,432}{3} & 0 & 0 & 0 \\
   \verb"["t^3\verb"]"\ssd^\circ_{9,\ell}(t) & 0 & 0 & 0 & 0 & 0 & -276\,824\,064 & -\frac{268\,435\,456}{3} & -\frac{113\,816\,633\,344}{15} & 0 & 0 \\
  \end{array}
 \)
\end{table}

\subsection{Digraphs, counting strongly connected components, part 1}
\label{Appendix: digraphs with marking variable t}
  According to \cref{theorem:cgf:several:scc}, the Coefficient GF of type \( (2,1) \) of digraphs with the marking variable~\( t \) for the number of strongly connected components satisfies
  \[
    (\Convert{2}{1}\, \widehat{\D})(z, w; t)
      =
    \left(
      \widehat{\D}(2zw; t)
    \right)^2
      \cdot
    \Change{2}{2}{1}
    \Big(
      (\Convert{2}{2}\, \E)(z, w; -t)
    \Big)
    \, ,
  \]
  or, in terms of the exponential Hadamard product,
  \[
    (\Convert{2}{1} \widehat{\D})(z, w; t)
      =
    t
      \cdot
    \dfrac
    {
      \left(
        e^{(1-t) \cdot \SCC(2^{3/2}z^2 w)}
          \cdot
	    \Big(
    		  \big(1 - \I(2zw)\big)^2
       		\odot_z
      	  \G(z, \sqrt2 - 1)
    		\Big)
      \right)
    		\odot_z
    	  \G(z, 1/\sqrt2 - 1)
    }
    {
      \left(
        \Robin{2}
        \big( e^{-t \cdot \SCC(2zw)} \big)
      \right)^2
    }
    \, .  
  \]
  The corresponding coefficients \( \hat{\dd}^\circ_{m,\ell}(t) \) are polynomials in \( t \).
  For small values of the parameters \( m \) and \( \ell \), they are given in \cref{tab:hatd(t)}.
  
\begin{table}[ht!]
  \caption{\label{tab:hatd(t)} Values of the coefficients \( \hat{\dd}^\circ_{m,\ell}(t) \) for \( m,\ell \leq 3 \).}
 \(
  \begin{array}{c|cccc}
   \ell & 0 & 1 & 2 & 3 \\
   \hline
   \hat{\dd}^\circ_{0,\ell}(t) & t & 0 & 0 & 0 \\
   \hat{\dd}^\circ_{1,\ell}(t) & 0 & 4(t^2-t) & 0 & 0 \\
   \hat{\dd}^\circ_{2,\ell}(t) & 0 & -4(t^2-t) & 4(5t^3-7t^2+2t) & 0 \\
   \hat{\dd}^\circ_{3,\ell}(t) & 0 & 0 & -32(2t^3-3t^2+t) & 8\big(\frac{61}{3}t^4-29t^3+14t^2-\frac{16}{3}t\big) \\
  \end{array}
 \)
\end{table}

  Extracting \( k \)th coefficient of \( (\Convert{2}{1}\, \widehat{\D})(z, w; t) \) in \( t \),
  we obtain asymptotics of digraphs with \( k \) strongly connected components
  whose dominant term is described by \cref{corollary:dominant:m-fixed}.
  For \( k = 1 \), these digraphs are strongly connected, and their asymptotics was discussed in \cref{Appendix: semi-strong digraphs}, see \cref{tab:scc}.
  For the cases \( k = 2 \) and \( k = 3 \), we present the asymptotic coefficients in \cref{tab:2-scc_in_d} and \cref{tab:3-scc_in_d}, respectively.

\begin{table}[ht!]
  \caption{\label{tab:2-scc_in_d} Values of the coefficients \( [t^2]\hat{\dd}^\circ_{m,\ell}(t) \) for \( m,\ell \leq 7 \).}
 \(
  \begin{array}{c|cccccccc}
   \ell & 0 & 1 & 2 & 3 & 4 & 5 & 6 & 7 \\
   \hline
   \verb"["t^2\verb"]"\hat{\dd}^\circ_{0,\ell}(t) & 0 & 0 & 0 & 0 & 0 & 0 & 0 & 0 \\
   \verb"["t^2\verb"]"\hat{\dd}^\circ_{1,\ell}(t) & 0 & 4 & 0 & 0 & 0 & 0 & 0 & 0 \\
   \verb"["t^2\verb"]"\hat{\dd}^\circ_{2,\ell}(t) & 0 & -4 & -28 & 0 & 0 & 0 & 0 & 0 \\
   \verb"["t^2\verb"]"\hat{\dd}^\circ_{3,\ell}(t) & 0 & 0 & 96 & 112 & 0 & 0 & 0 & 0 \\
   \verb"["t^2\verb"]"\hat{\dd}^\circ_{4,\ell}(t) & 0 & 0 & -96 & -896 & -248 & 0 & 0 & 0 \\
   \verb"["t^2\verb"]"\hat{\dd}^\circ_{5,\ell}(t) & 0 & 0 & 0 & 5\,632 & \frac{40\,960}{3} & \frac{271\,808}{3} & 0 & 0 \\
   \verb"["t^2\verb"]"\hat{\dd}^\circ_{6,\ell}(t) & 0 & 0 & 0 & -18\,432 & -77\,824 & 1\,449\,984 & \frac{2\,895\,728\,576}{45} & 0 \\
   \verb"["t^2\verb"]"\hat{\dd}^\circ_{7,\ell}(t) & 0 & 0 & 0 & 0 & \frac{13\,303\,808}{3} & 10\,747\,904 & \frac{14\,709\,161\,984}{15} & \frac{5\,311\,318\,221\,568}{45} \\
  \end{array}
 \)
\end{table}

\begin{table}[ht!]
  \caption{\label{tab:3-scc_in_d} Values of the coefficients \( [t^3]\hat{\dd}^\circ_{m,\ell}(t) \) for \( m,\ell \leq 7 \).}
 \(
  \begin{array}{c|cccccccc}
   \ell & 0 & 1 & 2 & 3 & 4 & 5 & 6 & 7 \\
   \hline
   \verb"["t^3\verb"]"\hat{\dd}^\circ_{0,\ell}(t) & 0 & 0 & 0 & 0 & 0 & 0 & 0 & 0 \\
   \verb"["t^3\verb"]"\hat{\dd}^\circ_{1,\ell}(t) & 0 & 0 & 0 & 0 & 0 & 0 & 0 & 0 \\
   \verb"["t^3\verb"]"\hat{\dd}^\circ_{2,\ell}(t) & 0 & 0 & 20 & 0 & 0 & 0 & 0 & 0 \\
   \verb"["t^3\verb"]"\hat{\dd}^\circ_{3,\ell}(t) & 0 & 0 & -64 & -232 & 0 & 0 & 0 & 0 \\
   \verb"["t^3\verb"]"\hat{\dd}^\circ_{4,\ell}(t) & 0 & 0 & 32 & 2\,048 & 2\,140 & 0 & 0 & 0 \\
   \verb"["t^3\verb"]"\hat{\dd}^\circ_{5,\ell}(t) & 0 & 0 & 0 & -6\,656 & -30\,720 & 56\,720 & 0 & 0 \\
   \verb"["t^3\verb"]"\hat{\dd}^\circ_{6,\ell}(t) & 0 & 0 & 0 & 4\,096 & 430\,080 & \frac{913\,408}{3} & \frac{87\,938\,960}{3} & 0 \\
   \verb"["t^3\verb"]"\hat{\dd}^\circ_{7,\ell}(t) & 0 & 0 & 0 & 0 & -4\,849\,664 & -\frac{41\,156\,608}{3} & -347\,275\,264 & \frac{1\,476\,072\,351\,296}{45} \\
  \end{array}
 \)
\end{table}

\subsection{Digraphs, counting strongly connected components, part 2}
\label{Appendix: digraphs with marking variables t and s}
  According to \cref{theorem:cgf:trivariate:digraphs}, the Coefficient GF of type \( (2,1) \) of digraphs with the marking variables~\( s \) and \( t \) for the number of source-like and all strongly connected components, respectively, satisfies
  \begin{multline*}
    (\Convert{2}{1}\, \widehat{\D})(z, w; s, t)
     = \\
    \widehat{\D}(2zw; t)
    \left[
      \Change{2}{2}{1}
       \Big(
         (\Convert{2}{2}\, \E) \big( z,w; (s-1)t \big)
       \Big)
       +
      \widehat{\D}(2zw; s, t)
       \cdot
      \Change{2}{2}{1}
       \Big( (\Convert{2}{2}\, \E)(z,w; -t) \Big)
    \right]
    \, .
  \end{multline*}
  
  The corresponding coefficients \( \hat{\dd}^\circ_{m,\ell}(s,t) \) are polynomials in \( s \) and \( t \).
  For \( m,\ell \leq 3 \), the non-zero values of the coefficients \( \hat{\dd}^\circ_{m,\ell}(s,t) \) are the following:
  \begin{align*}
    \hat{\dd}^\circ_{0,0}(s,t) &= st, \\
    \hat{\dd}^\circ_{1,1}(s,t) &= 4st^2 - 4st, \\
    \hat{\dd}^\circ_{2,1}(s,t) &= 4(s^2-2s)t^2+4st, \\
    \hat{\dd}^\circ_{2,2}(s,t) &= 2(s^2+9s)t^3-28st^2+8st, \\
    \hat{\dd}^\circ_{3,2}(s,t) &= 32(s^2-3s)t^3 - 32(s^2-4s)t^2 - 32st, \\
    \hat{\dd}^\circ_{3,3}(s,t) &= \frac{4}{3}(s^3+21s^2+100s)t^4 - 4(7s^2+51s)t^3 + 112st^2 - \frac{128}{3}st \, . 
  \end{align*}

  Clearly, putting \( s = 1 \), on can get the asymptotics discussed in \cref{Appendix: digraphs with marking variable t}.
  On the other hand, putting \( t = 1 \), we obtain asymptotics of digraphs with respect to source-like components.
  We provide the values of \( \hat{\dd}^\circ_{m,\ell}(s,1) \) in \cref{tab:hatd(s)}.
  
\begin{table}[ht!]
  \caption{\label{tab:hatd(s)} Values of the coefficients \( \hat{\dd}^\circ_{m,\ell}(s,1) \) for \( m,\ell \leq 4 \).}
 \(
  \begin{array}{c|ccccc}
   \ell & 0 & 1 & 2 & 3 & 4 \\
   \hline
   \hat{\dd}^\circ_{0,\ell}(s,1) & s & 0 & 0 & 0 & 0 \\
   \hat{\dd}^\circ_{1,\ell}(s,1) & 0 & 0 & 0 & 0 & 0 \\
   \hat{\dd}^\circ_{2,\ell}(s,1) & 0 & 4(s^2-s) & 2(s^2-s) & 0 & 0 \\
   \hat{\dd}^\circ_{3,\ell}(s,1) & 0 & 0 & 0 & \frac{4}{3}(s^3-s) & 0 \\
   \hat{\dd}^\circ_{4,\ell}(s,1) & 0 & 0 & 32(s^3-s) & 128(s^2-s) & \frac{2}{3}s^4 + \frac{4}{3}s^3 + 42s^2 - 44s \\
  \end{array}
 \)
\end{table}

  Finally, extracting coefficients leads to asymptotics of digraphs with prescribed number of strongly connected components or source-like strongly connected components.
  Thus, if we extract \( k \)th coefficient of \( (\Convert{2}{1}\, \widehat{\D})(z, w; s, t) \) in \( t \),
  we obtain asymptotics of digraphs with \( k \) strongly connected components with respect to source-like strongly connected components.
  For \( k = 1 \), these digraphs are strongly connected, and the result is given by \cref{tab:scc} whose entries are multiplied by \( s \).
  For the cases \( k = 2 \) and \( k = 3 \), the asymptotic coefficients are presented in \cref{tab:2-scc_in_d(s)} and \cref{tab:3-scc_in_d(s)}, respectively.
  On the other hand, if we extract \( k \)th coefficient of \( (\Convert{2}{1}\, \widehat{\D})(z, w; s, t) \) in \( s \),
  we obtain asymptotics of digraphs with \( k \) source-like strongly connected components with respect to all strongly connected components.
  The corresponding asymptotics for the cases \( k = 1, 2, 3 \) are summarized in \cref{tab:1-scc_in_d(t)}, \cref{tab:2-scc_in_d(t)} and \cref{tab:3-scc_in_d(t)}. 

\begin{table}[ht!]
  \caption{\label{tab:2-scc_in_d(s)} Values of the coefficients \( [t^2]\hat{\dd}^\circ_{m,\ell}(s,t) \) for \( m,\ell \leq 5 \).}
 \(
  \begin{array}{c|cccccc}
   \ell & 0 & 1 & 2 & 3 & 4 & 5 \\
   \hline
   \verb"["t^2\verb"]"\hat{\dd}^\circ_{0,\ell}(s,t) & 0 & 0 & 0 & 0 & 0 & 0 \\
   \verb"["t^2\verb"]"\hat{\dd}^\circ_{1,\ell}(s,t) & 0 & 4s & 0 & 0 & 0 & 0 \\
   \verb"["t^2\verb"]"\hat{\dd}^\circ_{2,\ell}(s,t) & 0 & 4(s^2-2s) & -28s & 0 & 0 & 0 \\
   \verb"["t^2\verb"]"\hat{\dd}^\circ_{3,\ell}(s,t) & 0 & 0 & -32(s^2-4s) & 112s & 0 & 0 \\
   \verb"["t^2\verb"]"\hat{\dd}^\circ_{4,\ell}(s,t) & 0 & 0 & 96(s^2-2s) & 128(s^2-8s) & -248s & 0 \\
   \verb"["t^2\verb"]"\hat{\dd}^\circ_{5,\ell}(s,t) & 0 & 0 & 0 & -512(3s^2-14s) & -\frac{4\,096}{3}(s^2-11s) & \frac{271\,808}{3}s \\
  \end{array}
 \)
\end{table}

\begin{table}[ht!]
  \caption{\label{tab:3-scc_in_d(s)} Values of the coefficients \( [t^3]\hat{\dd}^\circ_{m,\ell}(s,t) \) for \( m,\ell \leq 5 \).}
 \small
 \(
  \begin{array}{c|cccccc}
   \ell & 0 & 1 & 2 & 3 & 4 & 5 \\
   \hline
   \verb"["t^3\verb"]"\hat{\dd}^\circ_{0,\ell}(s,t) & 0 & 0 & 0 & 0 & 0 & 0 \\
   \verb"["t^3\verb"]"\hat{\dd}^\circ_{1,\ell}(s,t) & 0 & 0 & 0 & 0 & 0 & 0 \\
   \verb"["t^3\verb"]"\hat{\dd}^\circ_{2,\ell}(s,t) & 0 & 0 & 2(s^2+9s) & 0 & 0 & 0 \\
   \verb"["t^3\verb"]"\hat{\dd}^\circ_{3,\ell}(s,t) & 0 & 0 & 32(s^2-3s) & -4(7s^2+51s) & 0 & 0 \\
   \verb"["t^3\verb"]"\hat{\dd}^\circ_{4,\ell}(s,t) & 0 & 0 & 32(s^3-3s^2+3s) & -256(s^2-9s) & 2(89s^2+981s) & 0 \\
   \verb"["t^3\verb"]"\hat{\dd}^\circ_{5,\ell}(s,t) & 0 & 0 & 0 & -512(s^3-9s^2+21s) & 3\,072(s^2-11s) & -8(553s^2-7\,643s) \\
  \end{array}
 \)
\end{table}

\begin{table}[ht!]
  \caption{\label{tab:1-scc_in_d(t)} Values of the coefficients \( [s]\hat{\dd}^\circ_{m,\ell}(s,t) \) for \( m,\ell \leq 3 \).}
 \(
  \begin{array}{c|cccc}
   \ell & 0 & 1 & 2 & 3 \\
   \hline
   \verb"["s\verb"]"\hat{\dd}^\circ_{0,\ell}(s,t) & t & 0 & 0 & 0 \\
   \verb"["s\verb"]"\hat{\dd}^\circ_{1,\ell}(s,t) & 0 & 4(t^2-t) & 0 & 0 \\
   \verb"["s\verb"]"\hat{\dd}^\circ_{2,\ell}(s,t) & 0 & -4(2t^2-t) & 2(9t^3-14t^2+4t) & 0 \\
   \verb"["s\verb"]"\hat{\dd}^\circ_{3,\ell}(s,t) & 0 & 0 & -32(3t^3-4t^2+t) & 4\big(\frac{100}{3}t^4-51t^3+28t^2-\frac{32}{3}t\big) \\
  \end{array}
 \)
\end{table}

\begin{table}[ht!]
  \caption{\label{tab:2-scc_in_d(t)} Values of the coefficients \( [s^2]\hat{\dd}^\circ_{m,\ell}(s,t) \) for \( m,\ell \leq 4 \).}
 \(
  \begin{array}{c|ccccc}
   \ell & 0 & 1 & 2 & 3 & 4 \\
   \hline
   \verb"["s^2\verb"]"\hat{\dd}^\circ_{0,\ell}(s,t) & 0 & 0 & 0 & 0 & 0 \\
   \verb"["s^2\verb"]"\hat{\dd}^\circ_{1,\ell}(s,t) & 0 & 0 & 0 & 0 & 0 \\
   \verb"["s^2\verb"]"\hat{\dd}^\circ_{2,\ell}(s,t) & 0 & 4t^2 & 2t^3 & 0 & 0 \\
   \verb"["s^2\verb"]"\hat{\dd}^\circ_{3,\ell}(s,t) & 0 & 0 & 32(t^3-t^2) & 28(t^4-t^3) & 0 \\
   \verb"["s^2\verb"]"\hat{\dd}^\circ_{4,\ell}(s,t) & 0 & 0 & -96(t^3-t^2) & 128(2t^4-2t^3+t^2) & 2(258t^5-326t^4+89t^3) \\
  \end{array}
 \)
\end{table}

\begin{table}[ht!]
  \caption{\label{tab:3-scc_in_d(t)} Values of the coefficients \( [s^3]\hat{\dd}^\circ_{m,\ell}(s,t) \) for \( m,\ell \leq 5 \).}
 \(
  \begin{array}{c|cccccc}
   \ell & 0 & 1 & 2 & 3 & 4 & 5 \\
   \hline
   \verb"["s^3\verb"]"\hat{\dd}^\circ_{0,\ell}(s,t) & 0 & 0 & 0 & 0 & 0 & 0 \\
   \verb"["s^3\verb"]"\hat{\dd}^\circ_{1,\ell}(s,t) & 0 & 0 & 0 & 0 & 0 & 0 \\
   \verb"["s^3\verb"]"\hat{\dd}^\circ_{2,\ell}(s,t) & 0 & 0 & 0 & 0 & 0 & 0 \\
   \verb"["s^3\verb"]"\hat{\dd}^\circ_{3,\ell}(s,t) & 0 & 0 & 0 & \frac{4}{3}t^4 & 0 & 0 \\
   \verb"["s^3\verb"]"\hat{\dd}^\circ_{4,\ell}(s,t) & 0 & 0 & 32t^3 & 0 & 4(10t^5-\frac{29}{3}t^4) & 0 \\
   \verb"["s^3\verb"]"\hat{\dd}^\circ_{5,\ell}(s,t) & 0 & 0 & 0 & 512(2t^4-t^3) & -\frac{1024}{3}(t^5-t^4) & \frac{4}{3}(1\,154t^6-1\,480t^5+359t^4) \\
  \end{array}
 \)
\end{table}

\subsection{Digraphs, counting strongly connected components, part 3}
\label{Appendix: digraphs with many marking variables}
  According to \cref{theorem:cgf:multivariate:digraphs}, the Coefficient GF of type \( (2,2) \) of digraphs
  with the marking variables  \( u, v, y \) and \( t \) for the numbers of purely source-like, purely sink-like, isolated and all strongly connected components, respectively,
 is given by
  \[
    (\Convert{2}{2}\, \D)(z, w; u, v, y, t) =
    \D^\circ_1
     +
    \D^\circ_{20} \cdot
    \Change{2}{1}{2}
    \Big(
      \D^\circ_{21} + \D^\circ_{22} + \D^\circ_{23}
    \Big)
    \, ,
  \]
  where
  \begin{align*}
    \D_1^\circ(z,w; u,v,y,t) &=
      (y-u-v+1)t
       \cdot
      \D(2^{3/2} z^2w; u,v,y,t)
       \cdot
      (\Convert{2}{2}\, \SCC)(z,w)
    \, ; \\
    \D_{20}^\circ(z,w; u,v,y,t) &=
      \E\big( 2^{3/2}z^2w; (y-u-v+1)t \big)
    \, ;  \\
    \D_{21}^\circ(z,w; u,v,y,t) &=
      \widehat{\D}(2zw; u,t)
       \cdot
      \Change{2}{2}{1}
       \Big(
         (\Convert{2}{2}\, \E) \big( z,w; (v-1)t \big) 
       \Big)
    \, ; \\
    \D_{22}^\circ(z,w; u,v,y,t) &=
      \widehat{\D}(2zw; v,t)
       \cdot
      \Change{2}{2}{1}
       \Big(
         (\Convert{2}{2}\, \E) \big( z,w; (u-1)t \big) 
       \Big)
    \, ; \\
    \D_{23}^\circ(z,w; u,v,y,t) &=
      \widehat{\D}(2zw; u,t)
       \cdot
      \widehat{\D}(2zw; v,t)
       \cdot
      \Change{2}{2}{1}
       \Big( (\Convert{2}{2}\, \E)(z,w; -t) \Big)
    \, .
  \end{align*}

  The first non-zero values of the corresponding coefficients \( \dd^\circ_{m,\ell}(u,v,y,t) \), which are polynomials in \( u,v,y \) and \( t \), are listed below.
  \begin{align*}
    \dd^\circ_{0,0}(u,v,y,t) &= yt, \\
    \dd^\circ_{1,1}(u,v,y,t) &= 4uvt^2 - 4yt, \\
    \dd^\circ_{2,1}(u,v,y,t) &= 4(y^2-2uv)t^2+4yt, \\
    \dd^\circ_{2,2}(u,v,y,t) &= 2uv(u+v+8)t^3-28uvt^2+8yt, \\
    \dd^\circ_{3,2}(u,v,y,t) &= 32uv(y-3)t^3 + 32(4uv-y^2)t^2 - 32yt, \\
    \dd^\circ_{3,3}(u,v,y,t) &= \frac{4}{3}uv\big(u^2+v^2+21(u+v)+78\big)t^4 - 4uv\big(7(u+v)+44\big)t^3 + 112uvt^2 - \frac{128}{3}yt \, .
  \end{align*}
  Note that, after the substitution \( u = s \), \( v = 1 \), \( y = s \),
  we get the coefficients \( \hat{\dd}^\circ_{m,\ell}(s,t) \) seen in \cref{Appendix: digraphs with marking variables t and s}.
  In other words,
  \[
    \dd^\circ_{m,\ell}(s,1,s,t)  = \hat{\dd}^\circ_{m,\ell}(s,1)
    \, .
  \]
  Many other possible substitutions can be considered.
  For instance, to obtain the asymptotics of sink-like strongly connected components, one may put \( u = 1 \), \( v = s \), \( y = s \), which leads to exactly the same result for symmetry reasons.
  Another example is the asymptotics of strongly connected components that are neither source-like, nor sink-like.
  This asymptotics can be reached by the substitution \( u = \bar{s}^{-1} \), \( v = \bar{s}^{-1} \), \( y = \bar{s}^{-1} \), \( t = \bar{s}\bar{t} \),
  where \( \bar{s} \) is the marking variable for the target components
  (see the list of first non-zero values below).
  \begin{align*}
    \dd^\circ_{0,0}(\bar{s}^{-1},\bar{s}^{-1},\bar{s}^{-1},\bar{s}\bar{t}) &= \bar{t}, \\
    \dd^\circ_{1,1}(\bar{s}^{-1},\bar{s}^{-1},\bar{s}^{-1},\bar{s}\bar{t}) &= 4\bar{t}^2 - 4\bar{t}, \\
    \dd^\circ_{2,1}(\bar{s}^{-1},\bar{s}^{-1},\bar{s}^{-1},\bar{s}\bar{t}) &= -4\bar{t}^2+4\bar{t}, \\
    \dd^\circ_{2,2}(\bar{s}^{-1},\bar{s}^{-1},\bar{s}^{-1},\bar{s}\bar{t}) &= 4\big(1+4\bar{s}\big)\bar{t}^3-28\bar{t}^2+8\bar{t}, \\
    \dd^\circ_{3,2}(\bar{s}^{-1},\bar{s}^{-1},\bar{s}^{-1},\bar{s}\bar{t}) &= 32\big(1-3\bar{s}\big)\bar{t}^3 - 96\bar{t}^2 - 32\bar{t}, \\
    \dd^\circ_{3,3}(\bar{s}^{-1},\bar{s}^{-1},\bar{s}^{-1},\bar{s}\bar{t}) &= \frac{8}{3}\big(1+21\bar{s}+39\bar{s}^2\big)\bar{t}^4 - 8\big(7+22\bar{s}\big)\bar{t}^3 + 112\bar{t}^2 - \frac{128}{3}\bar{t} \, .
  \end{align*}
  
  As usually, extracting coefficients leads to asymptotics of digraphs with prescribed number of strongly connected components.
  We will not give an exhaustive list of all possibilities and will limit ourselves to a few remarks.
  First of all, the expression
  \[
    [u^pv^qy^rt^k](\Convert{2}{2}\, \D)(z, w; u, v, y, t)
  \]
  represents the asymptotics of digraphs with \( k \) strongly connected components such that
  \( p \) of them are strongly source-like,
  \( q \) of them are strongly sink-like, and
  \( r \) of them are isolated.
  As an example, we provide asymptotic coefficients for the digraphs with one purely source-like component, one purely sink-like component, one isolated component and one component of the general type (neither source-like, nor sink-like).
  This corresponds to the case where \( u = v = y = 1 \) and \( t = 4 \), see \cref{tab:1-1-1-4-scc_in_d(all)}.

\begin{table}[ht!]
  \caption{\label{tab:1-1-1-4-scc_in_d(all)} Values of the coefficients \( \alpha_{m,\ell}=[uvyt^4]\dd^\circ_{m,\ell}(u,v,y,t) \) for \( m,\ell \leq 8 \).}
 \(
  \begin{array}{c|ccccccccc}
   \ell & 0 & 1 & 2 & 3 & 4 & 5 & 6 & 7 & 8 \\
   \hline
   \alpha_{0,\ell} & 0 & 0 & 0 & 0 & 0 & 0 & 0 & 0 & 0 \\
   \alpha_{1,\ell} & 0 & 0 & 0 & 0 & 0 & 0 & 0 & 0 & 0 \\
   \alpha_{2,\ell} & 0 & 0 & 0 & 0 & 0 & 0 & 0 & 0 & 0 \\
   \alpha_{3,\ell} & 0 & 0 & 0 & 0 & 0 & 0 & 0 & 0 & 0 \\
   \alpha_{4,\ell} & 0 & 0 & 0 & 256 & 0 & 0 & 0 & 0 & 0 \\
   \alpha_{5,\ell} & 0 & 0 & 0 & -3\,072 & -5\,632 & 0 & 0 & 0 & 0 \\
   \alpha_{6,\ell} & 0 & 0 & 0 & 16\,384 & 143\,360 & 114\,176 & 0 & 0 & 0 \\
   \alpha_{7,\ell} & 0 & 0 & 0 & 0 & -1\,966\,080 & -4\,620\,288 & 8\,392\,704 & 0 & 0 \\
   \alpha_{8,\ell} & 0 & 0 & 0 & 0 & 31\,457\,280 & 139\,460\,608 & -\frac{859\,635\,712}{3} & -\frac{23\,526\,842\,368}{3} & 0 \\
  \end{array}
 \)
\end{table}

  If we extract \( k \)th coefficient of \( (\Convert{2}{2}\, \D)(z, w; u, v, y, t) \) in \( t \),
  we obtain the asymptotics of digraphs with \( k \) strongly connected components with respect to strongly connected components of different types.
  In particular, for \( k = 1 \), we revisit the asymptotics of strongly connected digraphs, and the result is given by \cref{tab:scc} whose entries are multiplied by \( y \).
  For the cases \( k = 2 \) and \( k = 3 \), the asymptotic coefficients are presented in \cref{tab:2-scc_in_d(all)} and \cref{tab:3-scc_in_d(all)}, respectively.

\begin{table}[ht!]
  \caption{\label{tab:2-scc_in_d(all)} Values of the coefficients \( \beta_{m,\ell}=[t^2]\dd^\circ_{m,\ell}(u,v,y,t) \) for \( m,\ell \leq 5 \).}
 \(
  \begin{array}{c|cccccc}
   \ell & 0 & 1 & 2 & 3 & 4 & 5 \\
   \hline
   \beta_{0,\ell} & 0 & 0 & 0 & 0 & 0 & 0 \\
   \beta_{1,\ell} & 0 & 4uv & 0 & 0 & 0 & 0 \\
   \beta_{2,\ell} & 0 & 4(y^2-2uv) & -28uv & 0 & 0 & 0 \\
   \beta_{3,\ell} & 0 & 0 & -32(y^2-4uv) & 112uv & 0 & 0 \\
   \beta_{4,\ell} & 0 & 0 & 96(y^2-2uv) & 128(y^2-8uv) & -248uv & 0 \\
   \beta_{5,\ell} & 0 & 0 & 0 & -512(3y^2-14uv) & -\frac{4\,096}{3}(y^2-11uv) & \frac{271\,808}{3}uv \\
  \end{array}
 \)
\end{table}

\begin{table}[ht!]
  \caption{\label{tab:3-scc_in_d(all)} Values of the coefficients \( \gamma_{m,\ell}=[t^3]\dd^\circ_{m,\ell}(u,v,y,t) \) for \( m,\ell \leq 4 \).}
 \(
  \begin{array}{c|ccccc}
   \ell & 0 & 1 & 2 & 3 & 4 \\
   \hline
   \gamma_{0,\ell} & 0 & 0 & 0 & 0 & 0 \\
   \gamma_{1,\ell} & 0 & 0 & 0 & 0 & 0 \\
   \gamma_{2,\ell} & 0 & 0 & 2uv(u+v+8) & 0 & 0 \\
   \gamma_{3,\ell} & 0 & 0 & 32uv(y-3) & -4uv\big(7(u+v)+44\big) & 0 \\
   \gamma_{4,\ell} & 0 & 0 & 32\big(y^3-uv(2y+u+v-4)\big) & -64uv\big(7y-3(u+v+11)\big) & 2uv\big(89(u+v)+892\big) \\
  \end{array}
 \)
\end{table}

\begin{table}[ht!]
  \caption{\label{tab:1-u-scc_in_d(all)} Values of the coefficients \( \eta_{m,\ell}=[u]\dd^\circ_{m,\ell}(u,1,1,t) \) for \( m,\ell \leq 4 \).}
 \(
  \begin{array}{c|ccccc}
   \ell & 0 & 1 & 2 & 3 & 4 \\
   \hline
   \eta_{0,\ell} & 0 & 0 & 0 & 0 & 0 \\
   \eta_{1,\ell} & 0 & 4t^2 & 0 & 0 & 0 \\
   \eta_{2,\ell} & 0 & -8t^2 & 2(9t^3-14t^2) & 0 & 0 \\
   \eta_{3,\ell} & 0 & 0 & -64(t^3-2t^2) & 4\big(\frac{100}{3}t^4-51t^3+28t^2\big) & 0 \\
   \eta_{4,\ell} & 0 & 0 & 32(t^3-6t^2) & -32(39t^4-58t^3+32t^2) & 2\big(905t^5-\frac{3\,304}{3}t^4+981t^3-124t^2\big) \\
  \end{array}
 \)
\end{table}

\begin{table}[ht!]
  \caption{\label{tab:1-y-scc_in_d(all)} Values of the coefficients \( \theta_{m,\ell}=[y]\dd^\circ_{m,\ell}(1,1,y,t) \) for \( m,\ell \leq 5 \).}
 \(
  \begin{array}{c|cccccc}
   \ell & 0 & 1 & 2 & 3 & 4 & 5 \\
   \hline
   \theta_{0,\ell} & t & 0 & 0 & 0 & 0 & 0 \\
   \theta_{1,\ell} & 0 & -4t & 0 & 0 & 0 & 0 \\
   \theta_{2,\ell} & 0 & 4t & 8t & 0 & 0 & 0 \\
   \theta_{3,\ell} & 0 & 0 & 32(t^3-t) & \frac{128}{3}t & 0 & 0 \\
   \theta_{4,\ell} & 0 & 0 & -64(t^3-t) & 64(5t^4-7t^3+2t) & -\frac{4\,096}{3}t & 0 \\
   \theta_{5,\ell} & 0 & 0 & 0 & -512(6t^4-7t^3+2t) & 256\big(\frac{61}{3}t^5-29t^4+14t^3-\frac{16}{3}t\big) & -\frac{3\,473\,408}{15}t \\
  \end{array}
 \)
\end{table}

  Similarly, extracting \( p \)th coefficient in \( u \),
  we obtain the asymptotics of digraphs with \( p \) purely source-like strongly connected components.
  Thus, \cref{tab:1-u-scc_in_d(all)} shows this asymptotics for the case \( p = 1 \) with respect to the number of all strongly connected components.
  Here, to avoid keeping references to other types of strongly connected components, we also make the substitution \( v = 1 \) and \( y = 1 \).
  In the same way, to get the asymptotics of digraphs with \( r = 1 \) isolated components that keeps track of the total number of connected components,
  we extract \( r \)th coefficient in \( y \) and substitute \( u = 1 \) and \( v = 1 \), see \cref{tab:1-y-scc_in_d(all)}.
  In particular, we can see from these tables that purely source-like and isolated components behave differently.

\subsection{Satisfiable 2-CNFs}
\label{Appendix: satisfiable 2-CNFs}

  According to \cref{theorem:cgf:sat}, the Coefficient GF of type \( (2,1) \) of satisfiable 2-CNF formulae satisfies
  \[
    (\Convert{2}{1}\, \SAT)(z, w)
     = 
    \SAT(2zw) \big(1 - \I(2zw) \big)
    \, .
  \]
  As a consequence, the corresponding asymptotic coefficients \( \dsat^\circ_{m,\ell} \) are of form
  \[
    \dsat^\circ_{m,\ell}
     =
    \mathbf 1_{m = \ell = 0}
     -
    \mathbf 1_{m = \ell > 0}
     \cdot 
	\dfrac{\bar{s}_m}{2^{\binom{m-1}{2}} \cdot m!}
	\, ,
  \]
  where (compare to \cref{corollary:asymptotics:sat_n})
  \[
    \bar{s}_m
     = 
    \left(
     \sum_{k=0}^{m-1}
      \binom{m}{k} \cdot 2^{m^2-k^2} \cdot
       \sat_k \cdot \ir_{m-k}
    \right)
     -
    \sat_m
    \, .
  \]
  Here, \( (\sat_m)_{m=0}^{\infty} \) is the counting sequence of satisfiable 2-CNF formulae,
  \[
    (\sat_m)
     = 
    1,\,
    1,\,
    15,\,
    2\,397,\,
    3\,049\,713,\,
    28\,694\,311\,447,\,
    2\,034\,602\,766\,692\,687,\,
    \ldots
  \]
  and \( (\ir_m)_{m=1}^{\infty} \) is the counting sequence of irreducible tournaments described in \cref{Appendix: connected graphs}.
  Thus, the sequence \( (\bar{s}_m)_{m=1}^{\infty} \) starts with
  \[
    (\bar{s}_m)
     = 
    1,\,
    1,\,
    67,\,
    12\,559,\,
    8\,976\,361,\,
    23\,458\,307\,761,\,
    225\,313\,054\,216\,027,\,
    \ldots
  \]
  and
  \[
    \big(\dsat^\circ_{m,m}\big)
     =
    1,\,
    -1,\,
    -\dfrac{1}{4},\,
    -\dfrac{67}{48},\,
    -\dfrac{12\,559}{1\,536},
    -\dfrac{8\,976\,361}{122\,880},\,
    -\dfrac{23\,458\,307\,761}{23\,592\,960},\,
    -\dfrac{225\,313\,054\,216\,027}{10\,569\,646\,080},\,
    \ldots
  \]

\subsection{Contradictory strongly connected implication digraphs}
\label{Appendix: CSCC}

  Due to \cref{theorem:asymptotics:cscc}, the  Coefficient GF of type \( (2,4) \) of contradictory strongly connected implication digraphs is given by
  \[
    (\Convert{2}{4}\, \CSCC)(z, w)
     =
    \exp \left(
      \dfrac{1}{2} \SCC(2^{7/2} z^4 w)
       -
      \CSCC(2^{5/2} z^4 w)
    \right)
     \cdot
    \Change{2}{2}{4}
     \Big( 1 - \I(2^{5/2}z^2w) \Big)
    \, .
  \]
  From this relation, it is cleat that the corresponding asymptotic coefficients \( \cscc^\circ_{m,\ell} \) are zeroes for all the odd values of \( m \).
  Another necessary condition for these coefficients to be non-zero is \( 2\ell \leq m \leq 4\ell \).
  We provide the numerical values of \( \cscc^\circ_{2m,\ell} \) for \( m,\ell \leq 6 \) in \cref{tab:cscc}.
  
\begin{table}[ht!]
  \caption{\label{tab:cscc} Values of the coefficients \( \cscc^\circ_{2m,\ell} \) for \( m,\ell \leq 6 \).}
 \(
  \begin{array}{c|ccccccc}
   \ell & 0 & 1 & 2 & 3 & 4 & 5 & 6 \\
   \hline
   \cscc^\circ_{0,\ell} & 1 & 0 & 0 & 0 & 0 & 0 & 0 \\
   \cscc^\circ_{2,\ell} & 0 & -8 & 0 & 0 & 0 & 0 & 0 \\
   \cscc^\circ_{4,\ell} & 0 & 16 & 0 & 0 & 0 & 0 & 0 \\
   \cscc^\circ_{6,\ell} & 0 & 0 & -512 & -\frac{32\,768}{3} & 0 & 0 & 0 \\
   \cscc^\circ_{8,\ell} & 0 & 0 & 4\,096 & 0 & -16\,777\,216 & 0 & 0 \\
   \cscc^\circ_{10,\ell} & 0 & 0 & 0 & -524\,288 & -\frac{33\,554\,432}{3} & -\frac{2\,336\,462\,209\,024}{15} & 0 \\
   \cscc^\circ_{12,\ell} & 0 & 0 & 0 & -4\,278\,190\,080 & 0 & -68\,719\,476\,736 & -8\,725\,724\,278\,030\,336 \\
  \end{array}
 \)
\end{table}

  As seen from \cref{tab:cscc}, the counting sequence \( (\cscc_n)_{n=0}^{\infty} \) of contradictory strongly connected implication digraphs,
  \[
    (\cscc_n)
     =
    0,\,
    0,\,
    1,\,
    1\,606,\,
    12\,864\,042,\,
    1\,035\,697\,286\,504,\,
    1\,137\,724\,245\,192\,445\,576,\,
    \ldots
  \]
  behaves as \( 2^{4\binom{n}{2}} \), as \( n \to \infty \).
  Similarly to the case of strongly connected digraphs, we can establish its asymptotic behavior more precisely, namely,
  \[
    \cscc_n \approx 2^{4\binom{n}{2}}
      \sum_{m\geq0} \dfrac{w_m(n)}{4^{nm}}
    \, ,
	\qquad\mbox{where}\quad
    w_m(n) = \sum_{\ell=\lceil m/2 \rceil}^{m}
      n^{\underline{\ell}}\,
      \cscc^\circ_{2m,\ell}
    \, .
  \] 
  For \( m\leq6 \), the polynomials \( w_m(n) \) are the following:
  \begin{align*}
    w_0(n) &= 1, \\
    w_1(n) &= -8n, \\
    w_2(n) &= 16n, \\
    w_3(n) &= -\dfrac{512}{3} n(n-1)(64n-125), \\
    w_4(n) &= -4\,096 n(n-1)(4\,096n^2 - 20\,480n + 24\,575), \\
    w_5(n) &= -\dfrac{524\,288}{15} n(n-1)(n-2) (4\,456\,448 n^2 - 31\,194\,816n + 53\,476\,431) \, , \\
    w_6(n) &= -16\,777\,216 n(n-1)(n-2) (520\,097\,792n^2 + 6\,241\,153\,024n - 14\,042\,579\,199) \, .
  \end{align*}

\subsection{2-CNFs, counting strongly connected components}
\label{Appendix: general case of 2-CNFs}

  According to \cref{theorem:cgf:cnf:types}, the Coefficient GF of type \( (2,2) \) of 2-CNF formulae is given by
  \begin{multline*}
    (\Convert{2}{2}\, \CNF)(z, w; s, t)
     =
    s \cdot \widehat{\D}(2^{3/2} z^2 w; t)
     \cdot \\
    \Change{2}{4}{2}
    \left[
      z \cdot
      \exp \left(
        \big(s-1\big) \cdot \CSCC(2^{3/2} z^4 w)
         +
        \dfrac{(1-t)}{2} \cdot \SCC(2^{5/2} z^4 w)
      \right)
       \cdot
      \Change{2}{2}{4}
       \Big( 1 - \I(4z^2w) \Big)
    \right]
    \, .
  \end{multline*}
  Here, the variables \( s \) and \( t \) mark, respectively,
  the numbers of contradictory strongly connected components and
  pairs of ordinary strongly connected components in the corresponding implication digraph.
  The asymptotic coefficients \( \cnf_{m,\ell}(s,t) \) are polynomials in \( s \) and \( t \).
  It is clear from the above relation that these polynomials are zeroes for even values of \( m \).
  For small odd values of \( m \), the polynomials \( \cnf_{m,\ell}(s,t) \) are shown in \cref{tab:2-cnf(all)}.
  
\begin{table}[ht!]
  \caption{\label{tab:2-cnf(all)} Values of the coefficients \( \cnf_{2m+1,\ell}(s,t) \) for \( m,\ell \leq 3 \).}
 \(
  \begin{array}{c|ccccc}
   \ell & 0 & 1 & 2 & 3 \\
   \hline
   \cnf_{1,\ell}(s,t) & s & 0 & 0 & 0 \\
   \cnf_{3,\ell}(s,t) & 0 & 8s(t-1) & 0 & 0 \\
   \cnf_{5,\ell}(s,t) & 0 & -16s(t-1) & 192s(t^2-t) & 0 \\
   \cnf_{7,\ell}(s,t) & 0 & 0 & -512s(4t^2-5t+1) & 2\,048s\big(\frac{25}{3}t^3-5t^2+2t-\frac{16}{3}\big) \\
  \end{array}
 \)
\end{table}

  Substituting \( t = 0 \), we obtain the asymptotic coefficients of implication digraphs that do not contain ordinary strongly connected components, see \cref{tab:2-cnf(s)}.
  In other words, all strongly connected components of such graphs are contradictory.
  
\begin{table}[ht!]
  \caption{\label{tab:2-cnf(s)} Values of the coefficients \( \cnf_{2m+1,\ell}(s,0) \) for \( m,\ell \leq 5 \).}
 \(
  \begin{array}{c|cccccc}
   \ell & 0 & 1 & 2 & 3 & 4 & 5 \\
   \hline
   \cnf_{1,\ell}(s,0) & s & 0 & 0 & 0 & 0 & 0 \\
   \cnf_{3,\ell}(s,0) & 0 & -8s & 0 & 0 & 0 & 0 \\
   \cnf_{5,\ell}(s,0) & 0 & 16s & 0 & 0 & 0 & 0 \\
   \cnf_{7,\ell}(s,0) & 0 & 0 & -512s & -\frac{32\,768}{3}s & 0 & 0 \\
   \cnf_{9,\ell}(s,0) & 0 & 0 & 2\,048s(s+2) & 0 & -16\,777\,216s & 0 \\
   \cnf_{11,\ell}(s,0) & 0 & 0 & 0 & -262\,144s(s+2) & -\frac{33\,554\,432}{3}s & -\frac{2\,336\,462\,209\,024}{15}s \\
  \end{array}
 \)
\end{table}
  
  Extracting \( k \)th coefficient in \( s \) from \( \cnf_{m,\ell}(s,0) \), we get the asymptotics of implication digraphs with exactly \( k \) contradictory strongly connected components.
  In particular, taking \( [s]\cnf_{2m+1,\ell}(s,0) \) leads us to the asymptotics of contradictory strongly connected implication digraphs whose coefficients are described by \cref{tab:cscc},
  while \( [s^2]\cnf_{2m+1,\ell}(s,0) \) corresponds to implication digraphs with two contradictory strongly connected components \cref{tab:2-cnf(ss)}
  
\begin{table}[ht!]
  \caption{\label{tab:2-cnf(ss)} Values of the coefficients \( [s^2]\cnf_{2m+1,\ell}(s,0) \) for \( m,\ell \leq 7 \).}
 \(
  \begin{array}{c|cccccccc}
   \ell & 0 & 1 & 2 & 3 & 4 & 5 & 6 & 7 \\
   \hline
   \verb"["s^2\verb"]"\cnf_{1,\ell}(s,0) & 0 & 0 & 0 & 0 & 0 & 0 & 0 & 0 \\
   \verb"["s^2\verb"]"\cnf_{3,\ell}(s,0) & 0 & 0 & 0 & 0 & 0 & 0 & 0 & 0 \\
   \verb"["s^2\verb"]"\cnf_{5,\ell}(s,0) & 0 & 0 & 0 & 0 & 0 & 0 & 0 & 0 \\
   \verb"["s^2\verb"]"\cnf_{7,\ell}(s,0) & 0 & 0 & 0 & 0 & 0 & 0 & 0 & 0 \\
   \verb"["s^2\verb"]"\cnf_{9,\ell}(s,0) & 0 & 0 & 2\,048 & 0 & 0 & 0 & 0 & 0 \\
   \verb"["s^2\verb"]"\cnf_{11,\ell}(s,0) & 0 & 0 & 0 & -262\,144 & 0 & 0 & 0 & 0 \\
   \verb"["s^2\verb"]"\cnf_{13,\ell}(s,0) & 0 & 0 & 0 & \frac{13\,497\,270\,272}{3} & 0 & 0 & 0 & 0 \\
   \verb"["s^2\verb"]"\cnf_{15,\ell}(s,0) & 0 & 0 & 0 & 0 & -\frac{6\,910\,602\,379\,264}{3} & -\frac{274\,877\,906\,944}{3} & 0 & 0 \\
  \end{array}
 \)
\end{table}
  
  Similarly, it is possible to consider implication digraphs with a given number \( k \) of pairs of ordinary strongly connected components.
  To obtain the corresponding asymptotics, it is sufficient to extract \( k \)th coefficient in \( t \) from \( \cnf_{m,\ell}(s,t) \).
  For \( k = 1,2,3 \), the reader can find the result of calculations in \cref{tab:2-cnf(t)}, \cref{tab:2-cnf(tt)} and \cref{tab:2-cnf(ttt)}.
  
\begin{table}[ht!]
  \caption{\label{tab:2-cnf(t)} Values of the coefficients \( [t]\cnf_{2m+1,\ell}(s,t) \) for \( m,\ell \leq 5 \).}
 \(
  \begin{array}{c|cccccc}
   \ell & 0 & 1 & 2 & 3 & 4 & 5 \\
   \hline
   \verb"["t\verb"]"\cnf_{1,\ell}(s,t) & 0 & 0 & 0 & 0 & 0 & 0 \\
   \verb"["t\verb"]"\cnf_{3,\ell}(s,t) & 0 & 8s & 0 & 0 & 0 & 0 \\
   \verb"["t\verb"]"\cnf_{5,\ell}(s,t) & 0 & -16s & -192s & 0 & 0 & 0 \\
   \verb"["t\verb"]"\cnf_{7,\ell}(s,t) & 0 & 0 & 2\,560s & 4\,096s & 0 & 0 \\
   \verb"["t\verb"]"\cnf_{9,\ell}(s,t) & 0 & 0 & -8\,192s & 0 & \frac{16\,973\,824}{3}s & 0 \\
   \verb"["t\verb"]"\cnf_{11,\ell}(s,t) & 0 & 0 & 0 & 1\,048\,576s(4s+9) & \frac{2\,046\,820\,352}{3}s & 84\,859\,158\,528s \\
  \end{array}
 \)
\end{table}
  
\begin{table}[ht!]
  \caption{\label{tab:2-cnf(tt)} Values of the coefficients \( [t^2]\cnf_{2m+1,\ell}(s,t) \) for \( m,\ell \leq 5 \).}
 \(
  \begin{array}{c|cccccc}
   \ell & 0 & 1 & 2 & 3 & 4 & 5 \\
   \hline
   \verb"["t^2\verb"]"\cnf_{1,\ell}(s,t) & 0 & 0 & 0 & 0 & 0 & 0 \\
   \verb"["t^2\verb"]"\cnf_{3,\ell}(s,t) & 0 & 0 & 0 & 0 & 0 & 0 \\
   \verb"["t^2\verb"]"\cnf_{5,\ell}(s,t) & 0 & 0 & 192s & 0 & 0 & 0 \\
   \verb"["t^2\verb"]"\cnf_{7,\ell}(s,t) & 0 & 0 & -2\,048s & -10\,240s & 0 & 0 \\
   \verb"["t^2\verb"]"\cnf_{9,\ell}(s,t) & 0 & 0 & 2\,048s & 786\,432s & 5\,472\,256s & 0 \\
   \verb"["t^2\verb"]"\cnf_{11,\ell}(s,t) & 0 & 0 & 0 & -17\,039\,360s & -134\,217\,728s & \frac{123\,505\,475\,584}{3}s \\
  \end{array}
 \)
\end{table}
  
\begin{table}[ht!]
  \caption{\label{tab:2-cnf(ttt)} Values of the coefficients \( [t^3]\cnf_{2m+1,\ell}(s,t) \) for \( m,\ell \leq 6 \).}
  \small
 \(
  \begin{array}{c|ccccccc}
   \ell & 0 & 1 & 2 & 3 & 4 & 5 & 6 \\
   \hline
   \verb"["t^3\verb"]"\cnf_{1,\ell}(s,t) & 0 & 0 & 0 & 0 & 0 & 0 & 0 \\
   \verb"["t^3\verb"]"\cnf_{3,\ell}(s,t) & 0 & 0 & 0 & 0 & 0 & 0 & 0 \\
   \verb"["t^3\verb"]"\cnf_{5,\ell}(s,t) & 0 & 0 & 0 & 0 & 0 & 0 & 0 \\
   \verb"["t^3\verb"]"\cnf_{7,\ell}(s,t) & 0 & 0 & 0 & \frac{51\,200}{3}s & 0 & 0 & 0 \\
   \verb"["t^3\verb"]"\cnf_{9,\ell}(s,t) & 0 & 0 & 0 & -786\,432s & -\frac{851\,968}{3}s & 0 & 0 \\
   \verb"["t^3\verb"]"\cnf_{11,\ell}(s,t) & 0 & 0 & 0 & 4\,194\,304s & \frac{1\,744\,830\,464}{3}s & 16\,965\,959\,680s & 0 \\
   \verb"["t^3\verb"]"\cnf_{13,\ell}(s,t) & 0 & 0 & 0 & -\frac{8\,388\,608}{3} & -96\,636\,764\,160s & -\frac{14\,809\,047\,236\,608}{3}s & \frac{4\,602\,751\,631\,753\,216}{9}s \\
  \end{array}
 \)
\end{table}

\printbibliography

\end{document}